\documentclass[11pt,a4paper]{article}
\usepackage{hyperref}

\makeatletter
\let\Hy@raisedlink@orig\Hy@raisedlink
\def\Hy@eqprefix{equation.}
\def\Hy@raisedlink#1{%
  \edef\Hy@tmp{\noexpand\in@{\Hy@eqprefix}{\@currentHref}}\Hy@tmp
  \ifin@\ifinner
    \Hy@raisedlink@orig{\hbox to\z@{\hskip-\dimexpr\linewidth-2em\relax#1\hss}}%
  \else\Hy@raisedlink@orig{#1}\fi
  \else\Hy@raisedlink@orig{#1}\fi}
\makeatother
\usepackage{amssymb,cite}
\usepackage{amsthm}
\usepackage{graphicx}
\usepackage{color,enumerate}
\usepackage{mathtools}
\mathtoolsset{showonlyrefs=true}
\usepackage{bm}
\usepackage{tikz}
\usetikzlibrary{decorations.pathreplacing}
\usepackage[margin=1.in]{geometry}
\usepackage{mathrsfs}
\newtheorem{theorem}{Theorem}[section]
\newtheorem{lemma}[theorem]{Lemma}
\newtheorem{proposition}[theorem]{Proposition}
\newtheorem{corollary}[theorem]{Corollary}
\theoremstyle{definition}
\newtheorem{example}[theorem]{Example}
\newtheorem{remark}[theorem]{Remark}

\newcommand{\tu}{\tilde{u}} 

\newcommand{\bydef}{\,\stackrel{\mbox{\tiny\textnormal{\raisebox{0ex}[0ex][0ex]{def}}}}{=}\,} 

\newcommand{\R}{\mathbb{R}}
\newcommand{\Z}{\mathbb{Z}}

\newcommand{\cL}{\mathcal{L}}
\newcommand{\cN}{\mathcal{N}}
\newcommand{\cQ}{\mathcal{Q}}
\newcommand{\cX}{\mathcal{X}}
\newcommand{\cY}{\mathcal{Y}}
\newcommand{\bA}{\bar{A}}
\newcommand{\bu}{\bar{u}}
\newcommand{\bv}{\bar{v}}
\newcommand{\ichi}{\mathbf{1}}
\newcommand{\dd}{\mbox{d}}
\newcommand{\tV}{\tilde{V}}

\newcommand{\m}{|\FK|} % total number of finite modes: |\FK| = K_1 K_2 ... K_d
\newcommand{\bbm}{\bm{K}}
\newcommand{\bN}{\bm{N}}
\newcommand{\bL}{\bm{L}}
\newcommand{\bk}{\bm{k}}
\newcommand{\bj}{\bm{j}}
\newcommand{\bn}{\bm{n}}

\newcommand{\bU}{\bar{U}}
\newcommand{\bV}{\bar{V}}

\newcommand{\kk}{\text{\textbf{k}}}
\newcommand{\qk}{{(\kk \bL)^{2q}}}

\newcommand{\FN}{\bm{F_N}}
\newcommand{\FK}{\bm{F_K}}
\newcommand{\FKa}{\bm{F}_{\bbm_a}}
\newcommand{\llbracket}{\mathopen{[\mkern-3mu[}}
\newcommand{\rrbracket}{\mathclose{]\mkern-3mu]}}

\begin{document}
\title{Rigorous integration of semilinear parabolic PDEs: forward evolution bounds and projected boundary value problems}
\date{\today}
\author{%
Jean-Philippe~Lessard \thanks{Department of Mathematics and Statistics, McGill University, 805 Sherbrooke West, Montreal, QC H3A 0B9, Canada (\texttt{jp.lessard@mcgill.ca})
}%
\and
Akitoshi~Takayasu \thanks{Institute of Systems and Information Engineering, University of Tsukuba, 1-1-1 Tennodai, Tsukuba, Ibaraki 305-8573, Japan (\texttt{takitoshi@risk.tsukuba.ac.jp})
}
}
\maketitle

\begin{abstract}
In this work, we present a rigorous integrator for initial and boundary value problems for semilinear parabolic PDEs. The method builds on earlier semigroup-based approaches while improving the bounds used in the fixed-point argument and extending the class of problems that can be treated. In particular, it controls the solution operator without requiring backward integration of the variational equation and allows projected boundary value problems to be studied through a simplified construction of an approximate inverse for the linearized operator. To illustrate the scope of the method, we present three applications to the two-dimensional Ohta--Kawasaki equation. For two initial value problems, we prove global existence and obtain explicit enclosures valid for all sufficiently large times around numerical approximations of stripe and spot equilibria. These results use a fixed-point formulation on an unbounded time interval, avoiding the construction of trapping regions. We also rigorously solve a projected boundary value problem with endpoint conditions chosen from linear approximations of the unstable and stable spaces of two saddle equilibria. This provides a step toward the rigorous study of heteroclinic saddle-to-saddle connections.
\end{abstract}



\section{Introduction}

Numerical simulations of semilinear parabolic partial differential equations (PDEs) can suggest the formation of patterns and transitions between equilibria. Establishing these observations rigorously requires computable error bounds between a numerical trajectory and a true solution. A key challenge is to control the accumulation of these bounds over long time intervals, including trajectories whose initial and terminal values are constrained.

In this paper, we develop a constructive method to rigorously solve initial value problems (IVPs) and projected boundary value problems (BVPs) for semilinear parabolic PDEs. Building on the semigroup-based integrator of \cite{MR4852997}, we improve the estimates used to propagate errors across time steps, avoiding backward integration of the linearized equations and retaining information that is lost when only scalar bounds are propagated. We use these estimates to validate trajectories over finite intervals, obtain global enclosures around approximate equilibria, and solve problems with prescribed endpoint constraints.

We consider scalar equations of the form
\begin{equation}\label{eq:general_PDE}
 u_t=(\lambda_0+\lambda_1\Delta+\lambda_2\Delta^2)u+\Delta^qN(u),
\end{equation}
where $q\in\{0,1\}$ and $N$ is a polynomial satisfying $N(0)=DN(0)=0$. The parameters $\lambda_0,\lambda_1,\lambda_2\in\R$ are chosen so that the highest-order linear term is dissipative: either $\lambda_2<0$, or $\lambda_2=0$ and $\lambda_1>0$. We also assume that the derivative order $2q$ in the nonlinear term is strictly below the order of the linear operator. Given $d\in\{1,2,3\}$ and $L_1,\ldots,L_d>0$, the spatial domain is
\[
 \Omega\bydef\prod_{j=1}^d[-\pi/L_j,\pi/L_j]\subset\R^d,
\]
with periodic boundary conditions. This class includes the Swift--Hohenberg, Cahn--Hilliard, and Ohta--Kawasaki equations. To simplify the representation and the computations, we impose even symmetry in each spatial coordinate. Extension to general periodic solutions is not pursued here.

Under these assumptions, expanding the solution in a cosine Fourier series leads to an infinite system of ordinary differential equations for the time-dependent coefficients $u(t)=(u_{\bk}(t))_{\bk\ge0}$
\begin{equation}\label{eq:PDEs}
 \dot u=\cL u+\cQ\cN(u),\qquad u(0)=u_0,\qquad t\in J\bydef[0,\tau],
\end{equation}
where $\cL$ is the diagonal linear operator, $\cQ$ represents $\Delta^q$, and $\cN$ is the polynomial nonlinearity in Fourier coordinates. The phase space is the Banach algebra $\ell_\nu^1$ introduced in Section~\ref{sec:function-space}. Its geometric weight $\nu\ge1$ can encode exponential decay of the Fourier coefficients, although the applications in Section~\ref{sec:OK_application} use $\nu=1$. For IVPs, $u_0$ is prescribed. For the projected BVP, $\tau$ is a prescribed time of flight and the initial value is one of the unknowns.

We first compute a numerical approximation $\bu$ by projecting \eqref{eq:PDEs} onto finitely many Fourier modes and using Chebyshev series expansions in time. Let $U(t,s)$ denote the evolution operator generated by the linearization $\cL+\cQ D\cN(\bu(t))$. Applying the variation of constants formula to \eqref{eq:PDEs} gives the fixed-point operator
\[
 (T(u))(t)\bydef U(t,0)u_0+
 \int_0^t U(t,s)\cQ\bigl(\cN(u(s))-D\cN(\bu(s))u(s)\bigr)\,\dd s.
\]
Its fixed points are solutions of the IVP. Since $U(t,s)$ is defined using the fixed approximation $\bu$, the derivative of the nonlinear remainder vanishes at $\bu$, and hence $DT(\bu)=0$. The idea is then to prove that $T$ is a contraction on a small ball centered at $\bu$ in a suitable Banach space of time-dependent Fourier coefficients. Bounds on the defect $T(\bu)-\bu$ and on the higher-order nonlinear remainder give computable sufficient conditions for the ball to be invariant and for $T$ to be a contraction. Its unique fixed point provides a solution of the IVP, and the radius gives an explicit error bound in the trajectory norm. In particular, the bound controls the spatial supremum error at every time in the validated interval.

Semigroup theory guarantees the existence of $U(t,s)$ under the parabolic assumptions \cite{MR3639578,MR3683781}, but the validation requires explicit bounds on its action on initial data and on the integral terms in $T$. To obtain them, we separate finitely many Fourier modes, enclosed with interval arithmetic, from the remaining modes, which form the infinite tail and are controlled by parabolic smoothing. Following the semigroup comparison of \cite{JBMaxime}, we bound the integrated finite propagator by comparing the time-dependent linearization with a constant matrix. The discrepancy is enclosed rigorously, and only forward integration of the variational equation is required.

An important feature of the multi-step construction is that the finite propagator matrices are multiplied before taking absolute values. This retains cancellations between contributions from different modes, which would be lost by multiplying scalar bounds at each step. The interaction between the finite modes and the infinite tail is bounded separately as a correction to these products. Moreover, the smoothing estimates include the derivative multiplier $\cQ$, allowing derivatives in the nonlinear term to be treated directly. These constructions give the one-step and multi-step estimates of Section~\ref{sec:multi-stepping}, which are then used to compute the defect and contraction bounds. 

For global existence, we join the validated finite-time trajectory to a fixed-point problem on an unbounded interval centered at an approximate equilibrium $\bu_*$. The residual of $\bu_*$ is included in the defect bound, so no separate equilibrium validation or trapping-region construction is required. If the finite-time endpoint enclosure is too large for the final contraction test, we validate intermediate trajectories to obtain an incoming enclosure that satisfies the test.

For projected BVPs, we constrain the initial endpoint to a finite approximation of the unstable space of the origin and the terminal endpoint to an affine space defined by approximate stable directions at a numerical saddle. The initial unstable dimension exceeds the target's numerical Morse index by one, and one scalar amplitude condition completes the endpoint equations. Since $DT(\bu)=0$, the trajectory correction can be eliminated from the linearized BVP. This allows us to construct an inverse by numerically inverting only a finite matrix for the endpoint parameters and bounding the errors rigorously.

The construction is closely related to \cite{JBMaxime,JBMaxime2}. A final infinite time step also appears in \cite{JBMaxime2}, where it is used to prove convergence to stable hyperbolic equilibria. Here the infinite-step argument uses the modewise propagation bounds and includes the residual of the approximate equilibrium. Earlier semigroup arguments using trapping regions include \cite{MizuguchiGlobal2017,MR4852997}, and rigorous computation of attracting equilibria for the Kuramoto-Sivashinsky equation appears in \cite{ZgliczynskiAttracting2002}.

Our approach belongs to the field of computer-assisted proofs (CAPs) in nonlinear analysis and dynamical systems; we refer the interested reader to \cite{notices_jb_jp,MR3971222,MR3990999} for an overview. Several approaches have been developed to obtain CAPs for the flow of parabolic PDEs. Early verification using finite-element and fixed-point techniques appears in \cite{MR1284278}. Topological methods based on covering relations and self-consistent bounds provide $C^0$ integration \cite{MR4113209,MR2049869,MR2788972,MR3167726,MR3408840,MR3338669}, while $C^1$ integration also controls the variational equation \cite{MR2728184}. Other methods use evolution operators and semigroup theory \cite{MR3639578,MR3683781,MR4444839,MR4388950}, finite element methods with Nakao's projection technique \cite{MR4134381,MR4015319}, or self-consistent bounds \cite{MR4257868}. Spectral approaches combine Fourier expansions in space with Chebyshev expansions or interpolation in time \cite{MR4379799,CadiotLessard2025,JBMaxime,JBMaxime2}.

Connecting-orbit problems motivate the projected BVP extension. Computer-assisted results for parabolic PDEs include connections in the one-dimensional Ohta--Kawasaki model \cite{MR3773757}, connections to stable equilibria using Fourier-Taylor parameterizations \cite{MR3906120}, and connections for the Kuramoto-Sivashinsky equation \cite{MR4938703,WilczakZgliczynski2024KSII}. Related forcing results use rigorously computed relative indices of critical points \cite{MorseFloer2024}. Completing a connection to a saddle requires placing the terminal endpoint in its nonlinear stable manifold. Combining the present projected BVP with validated local invariant-manifold parameterizations \cite{MR3906120,MR4668616} is a further step toward such a proof, as in the ODE approach of \cite{MR3207723}.

We illustrate the method with three trajectories of the two-dimensional Ohta--Kawasaki equation \cite{Ohta-Kawasaki} at $\varepsilon=0.39$. In the zero-mean, coordinatewise-even phase space, the origin has four unstable directions. For two prescribed initial data near the origin, we prove global existence and eventual uniform enclosures around numerical stripe and spot approximations (Theorems~\ref{thm:OK_GE} and~\ref{thm:OK_GE_spot}). For the third trajectory, we validate a finite orbit segment satisfying projected endpoint conditions associated with the origin and a saddle approximation with numerical Morse index three (Theorem~\ref{thm:OK_BVP}). The global results do not prove convergence to exact equilibria, and the projected endpoint conditions do not establish a heteroclinic connection. Figures~\ref{fig:bifurcation_diagram} and~\ref{fig:dynamics_overview} show the numerical patterns and the three trajectories.

\begin{figure}[h!]
	\begin{minipage}[b]{\linewidth}
		\centering
		\includegraphics[width=0.7\textwidth]{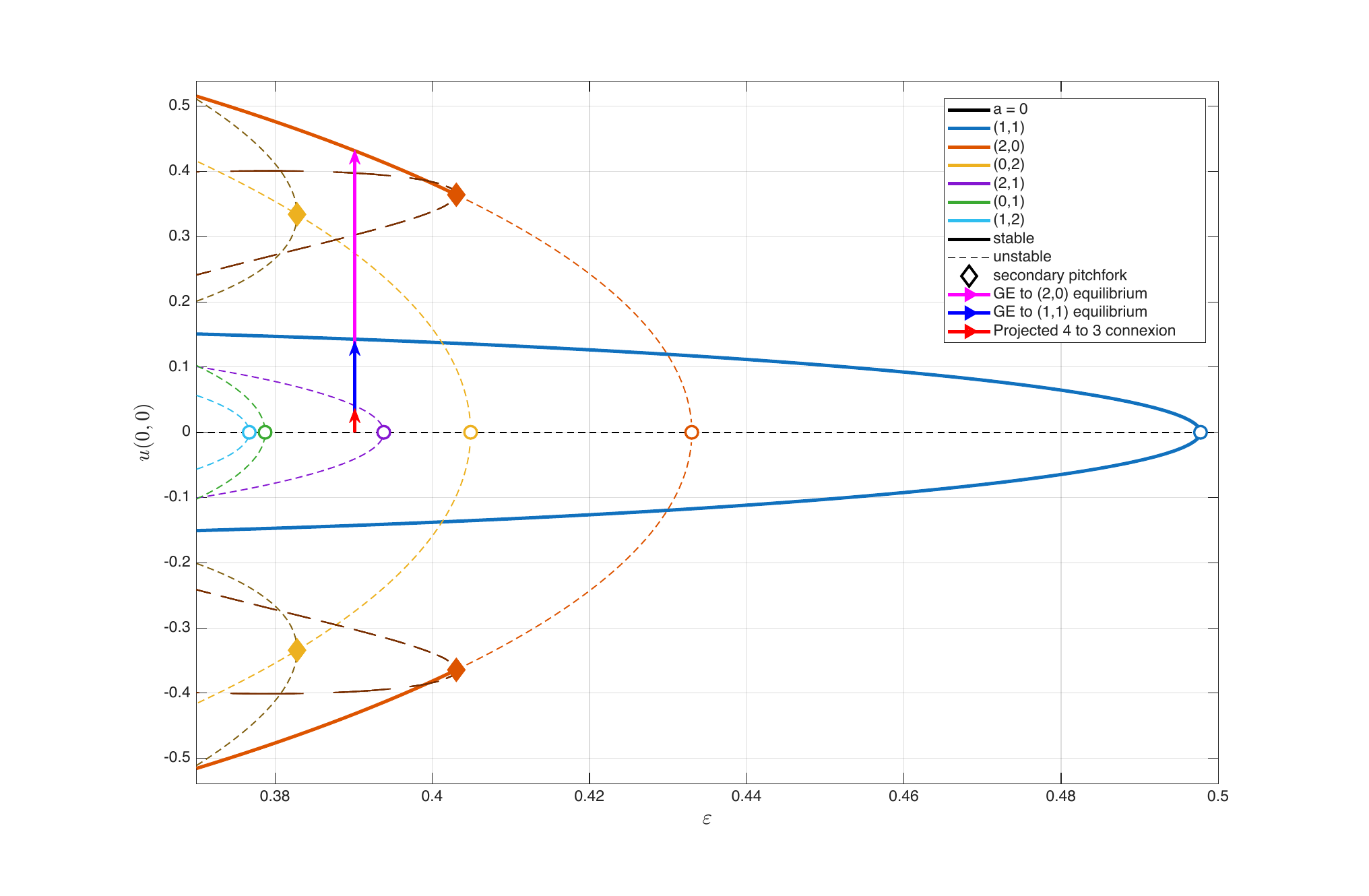}
	\end{minipage}
    \vspace{-1.5cm}
\caption{Numerical bifurcation diagram of two-dimensional Ohta--Kawasaki steady states, showing $u(0,0)$ against $\varepsilon$ for $\sigma=1$, zero mean, and $\bL=(1,1.1)$. The two darker curves emerge from secondary pitchforks on the $(2,0)$ and $(0,2)$ branches, numerically located near $\varepsilon=0.4031$ and $\varepsilon=0.3828$. At $\varepsilon=0.39$, the origin has Morse index four. The blue and magenta trajectories correspond to the IVPs rigorously solved with eventual enclosures near the stripe and spot approximations. The red trajectory corresponds to the projected BVP associated with the Morse-index-three saddle $u_3$. The rigorous conclusions are stated in Theorems~\ref{thm:OK_GE}, \ref{thm:OK_BVP}, and~\ref{thm:OK_GE_spot}.}\label{fig:bifurcation_diagram}
\end{figure}

\begin{figure}[h!]
	\begin{minipage}[b]{\linewidth}
		\centering
		\includegraphics[width=\textwidth]{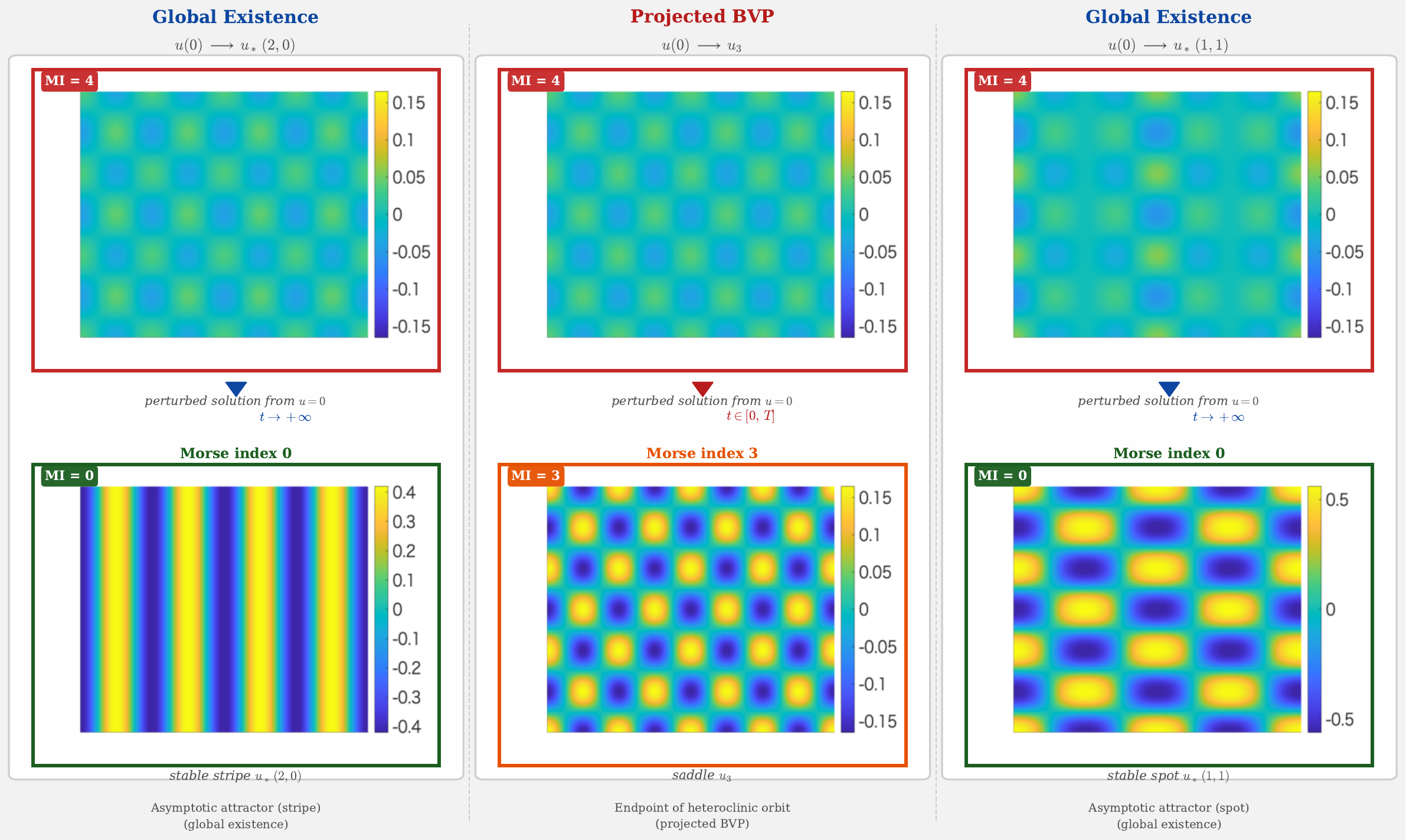}
	\end{minipage}
    \vspace{-.5cm}
\caption{Overview of the three trajectories for the two-dimensional Ohta--Kawasaki equation at $\varepsilon=0.39$. All start near the origin, which has Morse index four. Left: the stripe IVP, rigorously enclosed for all positive times with an eventual enclosure around $\bu_*^{(2,0)}$. Middle: the finite orbit segment satisfying projected endpoint conditions associated with the origin and the saddle $u_3$, whose dominant mode is $(2,1)$. Right: the spot IVP, rigorously enclosed for all positive times with an eventual enclosure around $\bu_*^{(1,1)}$. The numerical trajectory passes near the $(0,2)$ saddle.}\label{fig:dynamics_overview}
\end{figure}

The paper is organized as follows: Section~\ref{sec:functional_analytic_setup} introduces the function spaces and the fixed-point formulation. In Section~\ref{sec:multi-stepping}, we derive the one-step evolution bounds and their composition across time steps. Section~\ref{sec:global_in_time} develops the global enclosure construction, and Section~\ref{sec:BVP} presents the projected BVP and its reduced inverse. The Ohta--Kawasaki applications are given in Section~\ref{sec:OK_application}. Appendix~\ref{sec:Newton-refinement} describes the Newton refinement of the numerical trajectories. Appendices~\ref{app:linear_validation_details} and~\ref{app:global_details} contain the finite-time proofs and the details of the global construction.

\section{Functional analytic setup} \label{sec:functional_analytic_setup}
In this section, we write the PDE as an infinite system of ODEs for its Fourier coefficients and describe the numerical approximation. We then introduce the function spaces (Section~\ref{sec:function-space}), write the system in operator form (Section~\ref{sec:sequence_operators}), and construct the fixed-point map (Section~\ref{sec:fixed-point-operator}).

Recall that the initial value problem associated with \eqref{eq:general_PDE} is given by
\begin{equation} \label{eq:IVP_PDE}
	\begin{cases}
		u_t = (\lambda_0 + \lambda_1 \Delta + \lambda_2 \Delta^2)u + \Delta^q N(u), & t>0,\quad x\in \Omega,\\
		u(0,x) = u_0(x),& x \in \Omega,
	\end{cases}
\end{equation}
where $u_0$ is the prescribed initial condition and $x=(x_1,\dots,x_d)$. We assume that $u_0$ is periodic and even in each spatial coordinate, and that its Fourier coefficient sequence belongs to $\ell_\nu^1$, defined below. 
Under periodic boundary conditions and even symmetry in each spatial coordinate, the solution admits a cosine Fourier expansion in the spatial variables
\begin{equation}\label{eq:cosine-expansion}
    u(t,x) = \sum_{\bk\ge 0} \alpha_{\bk} u_{\bk}(t) \cos(k_1 L_1 x_1)\cdots\cos(k_d L_d x_d),\quad \alpha_{\bk}=\alpha_{k_1,\dots,k_d} \bydef \prod_{j=1}^d(2-\delta_{k_j,0}),
\end{equation}
where $\delta_{i,j}$ is Kronecker's delta.

The factor $\alpha_{\bk}$ accounts for the signed Fourier modes that are grouped together in the cosine expansion. Each positive component of $\bk$ contributes two choices of sign, while a zero component contributes only one. With this normalization, products of functions correspond to the usual convolution of their Fourier coefficients extended to $\Z^d$.
We extend the Fourier coefficients to $\Z^d$ by $u_{\bk}=u_{|\bk|}$, where $|\bk|\bydef(|k_1|,\dots,|k_d|)$ is the componentwise absolute value. This convention reflects even symmetry in each coordinate.
	
\begin{remark}
For simplicity, we use the same notation for the function $u(t,x)$ and its Fourier coefficients $u_{\bk}(t)$. We also use boldface characters to denote multi-indices. In general, $\bk=(k_1,\dots,k_d)\in\Z^d$. The stored cosine coefficients are indexed by $\bk\in\mathbb{N}_0^d$, where $\mathbb{N}_0=\{0,1,2,\dots\}$. Inequalities between multi-indices are understood component-wise. In particular, $\bk<\bn$ means $k_j<n_j$ for all $j=1,\dots,d$. Similarly, the relations $\bk\le\bn$, $\bk>\bn$, and $\bk\ge\bn$ are interpreted componentwise.
\end{remark}

The advantage of this representation is that the Laplacian acts separately on each Fourier mode, multiplying its coefficient by $-(\kk\bL)^2$. The nonlinearity couples the modes through convolution. We therefore obtain a diagonal multiplier for the spatial derivatives and convolution terms for the polynomial $N$.

Plugging the expansion \eqref{eq:cosine-expansion} into the nonlinear term $\Delta^q N(u)$ in the PDE \eqref{eq:general_PDE} yields the cosine Fourier representation
\begin{equation}\label{eq:nonlinearN}
\Delta^q N(u)=\sum_{\bk\ge 0} (-1)^q(\kk\bL)^{2q}\alpha_{\bk}\,\cN_{\bk}(u)\cos(k_1 L_1 x_1)\cdots\cos(k_d L_d x_d),
\end{equation}
where $\kk\bL\bydef \left((k_1L_1)^2+\cdots+(k_dL_d)^2\right)^{1/2}$, $u=(u_{\bk})_{\bk\ge 0}$, and $\cN_{\bk}(u)$ denotes a nonlinear term involving discrete convolutions satisfying $\cN_{\bk}(u)=\cN_{|\bk|}(u)$, $\cN(0)=0$, and $D\cN(0)=0$.

Substituting \eqref{eq:cosine-expansion} into \eqref{eq:general_PDE} then results in the infinite-dimensional system of ordinary differential equations (ODEs)
\begin{equation}\label{eq:general-ODE}
    \dot{u}_{\bk}(t) = \mu_{\bk}u_{\bk}(t) + (-1)^q(\kk\bL)^{2q}\,\cN_{\bk}(u(t)),\quad\bk\ge 0,
\end{equation}
where $\mu_{\bk}\bydef \lambda_0 - \lambda_1(\kk\bL)^2 + \lambda_2(\kk\bL)^4$ is determined by the linear term $\lambda_0 + \lambda_1 \Delta + \lambda_2 \Delta^2$ in~\eqref{eq:general_PDE}.
Therefore, solving the IVP \eqref{eq:IVP_PDE} is equivalent to determining the time-dependent Fourier coefficients $u(t)\bydef(u_{\bk}(t))_{\bk\ge 0}$ satisfying \eqref{eq:general-ODE} with initial condition $u_{\bk}(0)=(u_0)_{\bk}$, where $(u_0)_{\bk}$ denotes the cosine Fourier coefficients of the initial values $u_0(x)$.

The first step of the a posteriori method is to compute a good numerical approximation. We do this by retaining finitely many Fourier modes. We then prove that a solution of the full system lies near this numerical trajectory, including bounds on the modes omitted from the computation.

Throughout this paper, we denote an approximate solution of \eqref{eq:general-ODE} by $\bu$.
To compute $\bu$, we first introduce a finite truncation of the Fourier indices.
For $\bN=(N_1,\dots,N_d)$, define $\FN \bydef \{\bk\ge 0:\bk<\bN\}$. Note that $\FN=F_{N_1}\times\dots \times F_{N_d}$, where $F_{N_j}\bydef\{k_j\ge 0:k_j<N_j\} = \{0,\dots,N_j - 1\}$.
Using this truncation in \eqref{eq:general-ODE}, we consider the following finite-dimensional IVP. When evaluating $\cN_{\bk}(u)$, the finite vector $u=(u_{\bk})_{\bk\in\FN}$ is extended by zero outside $\FN$:
\begin{equation} \label{eq:ODEs_finite}
	\begin{cases}
		\dot{u}_{\bk}(t) = \mu_{\bk}u_{\bk}(t) + (-1)^q(\kk\bL)^{2q}\,\cN_{\bk}(u(t)), \\
		u_{\bk}(0)=(u_0)_{\bk},
	\end{cases}\quad\quad \bk\in\FN.
\end{equation}
We compute a numerical solution of \eqref{eq:ODEs_finite} using a standard ODE solver (e.g., MATLAB ODE solvers), refined as described below, and denote it by $\left(\bu^{(\bN)}_{\bk }(t)\right)_{\bk\in\FN}$.
The approximate solution $\bu$ of \eqref{eq:general-ODE} is then defined by extending this numerical solution as
\begin{equation} \label{eq:a_bar}
\bu_{\bk}(t) \bydef 
\begin{cases}
\bu^{(\bm{N})}_{\bk}(t),& \bk\in\FN 
\\
0, & \bk\not\in\FN.
\end{cases}
\end{equation}

Extending $\bu$ by zero allows us to compare it directly with a solution of the full system. Even though the approximation has a zero tail, the nonlinearity generally produces nonzero coefficients outside $\FN$. These coefficients contribute to the residual that must be bounded in the fixed-point argument.

Each component $\bu^{(\bm{N})}_{\bk}(t)$ is represented by a finite Chebyshev series in time, yielding a Chebyshev-Fourier representation of the approximate solution. Its coefficients are obtained by Newton iteration on the corresponding zero-finding problem, with the ODE solution as initial guess (see Appendix~\ref{sec:Newton-refinement}). General background on Fourier-Chebyshev spectral discretizations is given in \cite{boyd}, while Chebyshev approximation and interpolation are treated in \cite{trefethen}.

\subsection{Function spaces}\label{sec:function-space}
To prove that a true solution lies close to $\bu$, we need a norm that controls both the Fourier coefficients and their evolution in time. We first introduce the spatial sequence space. Its Banach-algebra property will allow us to estimate the polynomial nonlinearity directly in terms of the size of the solution.

To formulate the ODEs \eqref{eq:general-ODE} in sequence space, we define a Banach space of multi-index sequences by
\begin{equation}\label{eq:ell-one-space}
	\ell_{\nu}^1\bydef\left\{a=(a_{\bk })_{\bk\ge 0}:a_{\bk}\in \R,~\|a\|_{\nu}\bydef\sum_{\bk\ge 0}\alpha_{\bk}|a_{\bk }|\nu^{k_1+\dots+k_d}<+\infty\right\},
\end{equation}
where $\nu\ge 1$ denotes the geometric weight. We use the shorthand
$\nu^{\bk}\bydef\nu^{k_1+\dots+k_d}$.

The normalization in \eqref{eq:cosine-expansion} makes $\|a\|_\nu$ equal to the weighted $\ell^1$ norm of the corresponding sequence on $\Z^d$. It also implies absolute and uniform convergence of the Fourier series, whose spatial supremum norm is bounded by $\|a\|_\nu$. The choice $\nu>1$ encodes exponential decay of the Fourier coefficients.
The space $\ell_{\nu}^1$ is a Banach algebra under discrete convolution, that is,
\begin{equation}\label{eq:Banach_algebra}
	\|a\ast b\|_{\nu}\le \|a\|_{\nu}\|b\|_{\nu}
\end{equation}
holds for all $a, b\in \ell_{\nu}^1$, where the discrete convolution of sequences $a$ and $b$ is defined by
\begin{equation}\label{eq:discrete_convolution}
	\left(a\ast b\right)_{\bk}\bydef\sum_{\substack{\bk_1+\bk_2= \bk\\ \bk_1,\bk_2\in \Z^d}}a_{|\bk_1|}b_{|\bk_2|}.
\end{equation}

The convolution estimate controls $N(u)$, but the additional Laplacian in the case $q=1$ introduces a factor growing quadratically with the frequency. To keep track of this derivative loss, we introduce the weaker sequence space
\[
	\ell_{\nu,-2q}^1\bydef\left\{a:\|a\|_{\nu,-2q}\bydef\sum_{\bk\ge 0}|a_{\bk}|\frac{\alpha_{\bk}\nu^{k_1+\dots+k_d}}{\max\{1,(\kk\bL)^{2q}\}}<+\infty\right\}.
\]
This space is also a Banach space, but for $q>0$ it is not a Banach algebra under discrete convolution. In addition, $\ell_{\nu}^1=\ell_{\nu,0}^1$ ($q=0$) holds.

Since $\max\{1,(\kk\bL)^{2q}\}\ge1$, we have $\ell_\nu^1\subset\ell_{\nu,-2q}^1$ and $\|a\|_{\nu,-2q}\le\|a\|_\nu$. The solution and its polynomial nonlinearity are measured in $\ell_\nu^1$, while the differentiated nonlinear term is measured in $\ell_{\nu,-2q}^1$. The smoothing properties of the parabolic equation will allow us to recover the stronger norm after integration in time.

Let $\tau>0$ be the final time of integration and $J=[0,\tau]$.
We define a function space of the time-dependent sequences as $C(J,\ell_{\nu}^1)$, which is the Banach space with the norm
\begin{equation}\label{eq:Xnorm}
	\|u\|\bydef\sup _{t \in J}\|u(t)\|_{\nu}.%
\end{equation}
For the computable bounds, it is more convenient to take the supremum of each Fourier coefficient before summing over the modes. This gives the stronger trajectory space
\begin{equation}
X\bydef\ell_\nu^1(C^0(J))=\left\{u=(u_{\bk})_{\bk\ge0}:u_{\bk}\in C^0(J;\R),\quad\|u\|_X\bydef\sum_{\bk\ge0}\alpha_{\bk}|u_{\bk}|_\infty\nu^{k_1+\dots+k_d}<\infty\right\},
\end{equation}
where $|\cdot|_\infty$ denotes the supremum norm over $t\in J$. The norm $\|\cdot\|_X$ is obtained by first taking the supremum of each Fourier coefficient and then applying the weighted sequence norm. Below, we also write $X=\ell^1(C^0(J))$, with these weights understood. Finite-mode truncations of any $u\in X$ are continuous as $\ell_\nu^1$-valued functions and converge uniformly in that space. Thus $X\subset C(J,\ell_\nu^1)$, and $\|u\|\le\|u\|_X$.

Taking the supremum before summing is useful because our computations enclose each Fourier coefficient over the whole time interval. Different coefficients may attain their largest values at different times. The $X$-norm accounts for all these bounds and, by the inclusion above, also controls the error in $\ell_\nu^1$ at every time.

Finally, for general Banach spaces $\cX$ and $\cY$, the operator norm for a bounded linear operator $B$ from $\cX$ to $\cY$ is defined by
\begin{equation}
	\|B\|_{B(\cX,\cY)}\bydef \sup_{u\in \cX\setminus\{0\}}\frac{\|B u\|_{\cY}}{\|u\|_{\cX}}.
\end{equation}

\subsection{Infinite-dimensional ODEs on sequence space}\label{sec:sequence_operators}
We can now write the Fourier system using the three operators introduced in the introduction. The operator $\cL$ contains the linear part of the PDE, $\cN$ represents the polynomial nonlinearity, and $\cQ$ contains the derivatives acting on that nonlinearity. Keeping these contributions separate will allow us to use the dissipative linear dynamics to control the derivative loss in the nonlinear term.

Let $\cL:D(\cL)\subset \ell_{\nu}^1\to \ell_{\nu}^1$ be a densely defined closed operator%
\footnote{%
Indeed, $D(\cL)$ is dense in $\ell_\nu^1$ because the truncated subspace is dense in $\ell_\nu^1$, i.e., $\forall a\in \ell_\nu^1$, $\forall \varepsilon>0$, $\exists \bN$ s.t.\ its truncation $a^{(\bN)}\in D(\cL)$ satisfies $\|a-a^{(\bN)}\|_\nu < \varepsilon$.
It is also closed because if
$a^n\to a$ and $\cL a^n\to b$ in $\ell_\nu^1$ with $a^n\in D(\cL)$ as $n\to\infty$, then, componentwise, $b_{\bk}=\mu_{\bk}a_{\bk}$ for all $\bk$, and therefore $a\in D(\cL)$ and
$\cL a=b$.%
}
acting on a sequence $a=(a_{\bk})_{\bk\ge0}$ as
\[
    \left(\cL a\right)_{\bk}\bydef\mu_{\bk}a_{\bk} = (\lambda_0 - \lambda_1 (\kk\bL)^2 + \lambda_2 (\kk\bL)^4) a_{\bk},
\]
where the domain of the operator $\cL$ is denoted by $D(\cL)\bydef\left\{a\in\ell_\nu^1:(\mu_{\bk}a_{\bk})_{\bk\ge0}\in\ell_\nu^1\right\}$.
We also define $\cQ:\ell_{\nu}^1\to \ell_{\nu,-2q}^1$ to be a multiplication operator acting on a sequence $a=(a_{\bk})_{\bk\ge0}$ as
\begin{equation}\label{eq:def_Q}
	(\cQ a)_{\bk} \bydef (-1)^q (\kk\bL)^{2q}a_{\bk}.
\end{equation}
For $q=0$, we set $\cQ=I$, including on the zero mode. The factor $(\kk\bL)^0$ is interpreted as $1$ throughout. In either case, $\|\cQ a\|_{\nu,-2q}\le\|a\|_\nu$.

Furthermore, we define $\cN:\ell_{\nu}^1\to \ell_{\nu}^1$ to be a Fr\'echet differentiable nonlinear operator acting on a sequence $a=(a_{\bk})_{\bk\ge0}$ as
\[
\left(\cN(a)\right)_{\bk}\bydef \cN_{\bk}(a),
\]
where $\cN_{\bk}$ is defined in \eqref{eq:nonlinearN}. For the polynomial $N(z)=\sum_{p=2}^{P}c_pz^p$, its Fourier representation is $\cN(a)=\sum_{p=2}^{P}c_pa^{\ast p}$, where the convolution powers are defined recursively by $a^{\ast 1}\bydef a$ and $a^{\ast p}\bydef a^{\ast(p-1)}\ast a$ for $p\ge2$, using the convolution in \eqref{eq:discrete_convolution}. For example, $a^{\ast 2}=a\ast a$ and $a^{\ast 3}=(a\ast a)\ast a$. In particular, $\cN(0)=0$ and $D\cN(0)=0$. The Banach-algebra estimate shows that the non-decreasing function
\[
 g:[0,\infty)\to[0,\infty),\qquad g(r)\bydef\sum_{p=2}^{P}p|c_p|r^{p-1},
\]
satisfies, for every $a\in\ell_\nu^1$,
\begin{equation} \label{eq:N_assumption}
\left\|D\cN(a)\phi\right\|_{\nu}\le g(\| a \|_{\nu})\|\phi\|_{\nu},\quad\forall\phi\in\ell_{\nu}^1.
\end{equation}

The function $g$ bounds the derivative of the nonlinearity in terms of the norm of its argument. This is the bound used later to control the coupling between the finite modes and the tail. With these operators defined, the coefficient system \eqref{eq:general-ODE} takes the form

\begin{equation}\label{eq:ODE-sequence-space}
	\dot{u} = f(u)\bydef \cL u + \cQ \cN(u),
\end{equation}
where $u(t)=(u_{\bk}(t))_{\bk\ge 0}$ is the unknown sequence of time-dependent Fourier coefficients.
Together with $u(0)=u_0$, this gives the abstract IVP \eqref{eq:PDEs} with phase space $\ell_\nu^1$. The weaker space $\ell_{\nu,-2q}^1$ accommodates the derivative loss in the nonlinear term $\cQ\cN(u)$. The unbounded linear term is interpreted through the mild formulation below, rather than as a vector field defined on the whole sequence space.

\subsection{Fixed-point operator}\label{sec:fixed-point-operator}

We reformulate the IVP as a fixed-point problem by linearising around the approximate solution $\bu$ in \eqref{eq:a_bar}.
Rewriting the ODE \eqref{eq:ODE-sequence-space} with initial condition $u(0)=u_0$ in the form
\begin{equation}\label{eq:linearised_IVP}
	\begin{cases}
		\dot{u} - \left(\cL + \cQ D\cN(\bu)\right)u = \cQ\!\left(\cN(u) - D\cN(\bu)\,u\right),\\
		u(0) = u_0,
	\end{cases}
\end{equation}
the left-hand side is governed by the time-dependent linear operator $\cL+\cQ D\cN(\bu(t))$, while the right-hand side contains the nonlinear remainder
\begin{equation}\label{eq:G_def}
	(G(u))(t)\bydef\cN(u(t))-D\cN(\bu(t))u(t).
\end{equation}
Here $G(u(s))$ denotes $(G(u))(s)$. The Banach-algebra property implies that $G:X\to X$ is smooth, with $DG(\bu)=0$.
This vanishing derivative will give a small contraction rate near $\bu$. The error of the approximation is measured separately by $T(\bu)-\bu$.

Let $\{U(t,s)\}_{0\le s\le t\le\tau}$ denote the evolution operator generated by $\cL+\cQ D\cN(\bu(t))$. Its existence follows from the parabolic assumptions and the smooth, finite Fourier-Chebyshev approximation $\bu$. See \cite{MR3639578,MR3683781} for the semigroup framework. For each $\phi\in\ell_\nu^1$, $v(t)=U(t,s)\phi$ is the mild solution of the following IVP, with the differential equation understood classically for $t>s$:
\begin{equation}\label{eq:evolution_operator}
	\dot{v}(t) = \left(\cL + \cQ D\cN(\bu(t))\right)v(t),\quad v(s)=\phi.
\end{equation}
The evolution operator $U(t,s)$ maps the linearized solution at time $s$ to its value at time $t$. Applying the variation of constants formula to \eqref{eq:linearised_IVP} gives a contribution from the initial condition and an integral contribution from the nonlinear remainder. This leads to the fixed-point formulation

\begin{equation}\label{eq:T_single}
	u(t) = (T(u))(t) \bydef U(t,0)\,u_0 + \int_0^t U(t,s)\,\cQ G(u(s))\,ds.
\end{equation}
For $q=1$, parabolic smoothing compensates for the derivatives in $\cQ$, making the integral well defined. The estimates in Section~\ref{sec:multi-stepping} show that both terms belong to $X$. Thus $T:X\to X$, and its fixed points are precisely the solutions in $X$ of the IVP \eqref{eq:PDEs}, understood through the integral formulation \eqref{eq:T_single}.
We therefore seek bounds on $T(\bu)-\bu$ and on $DT$ near $\bu$. Since $DT(\bu)=0$, the contraction rate depends on higher-order nonlinear terms. These give the $Y$- and $Z$-bounds derived in Section~\ref{sec:multi-stepping}.

In principle, one could apply the contraction mapping theorem directly to $T$ defined in~\eqref{eq:T_single} on the time interval $J=[0,\tau]$.
However, when $\tau$ is large, the computable upper bounds on $\|U(t,0)\|$ typically grow substantially with the length of the interval, making the resulting $Y$- and $Z$-bounds too large to verify the contraction argument.
To overcome this difficulty, in Section~\ref{sec:multi-stepping} we decompose $J$ into short time steps $J_1,J_2,\dots,J_M$ and construct a multi-step fixed-point map $\bm{T}=(T_m)_{m=1}^M$ (see \eqref{eq:fixed-point-form}), so that the contraction argument can be applied on the product of the step spaces over a long time interval.
This technique is analogous to domain decomposition in time, as used in computer-assisted integration of differential equations. See, for example, \cite{JBMaxime,JBMaxime2}.
The key challenge is then to propagate rigorous bounds on the evolution operator across multiple steps, which is the subject of Section~\ref{sec: solution_operator}.

\section{Non-autonomous semigroup approach for integrating PDEs}\label{sec:multi-stepping}

In this section, we obtain a rigorous enclosure of a numerical trajectory over several time steps. Section~\ref{sec:trajectory_validation} introduces the recursive fixed-point map \eqref{eq:fixed-point-form} and the Radii Polynomial Theorem~\ref{thm:RadiiPolynomial}, a Newton-Kantorovich type argument that requires a defect bound $Y_m$ and derivative bounds $Z_m^j$. Section~\ref{sec:defect_derivative_bounds} expresses these bounds through linear propagation estimates, which are computed in Section~\ref{sec: solution_operator}. The proofs and the Ohta--Kawasaki tail constants are given in Appendix~\ref{app:linear_validation_details}.

\subsection{A fixed-point formulation over several time steps}\label{sec:trajectory_validation}

We divide $J=[0,\tau]$ into $M$ short \emph{time steps} $J=\bigcup_{m=1}^M J_m$, where $J_m \bydef [\tau_{m-1}, \tau_m]$ ($m=1,\dots,M$), $0 = \tau_0 < \tau_1 < \dots<\tau_M=\tau$, and the \emph{step size} $h_m\bydef \tau_m-\tau_{m-1}$ can be chosen adaptively.
On each time step $J_m$, let $u_{m}=u|_{J_{m}}$ and let $\bu_{m}$ be the corresponding Fourier-Chebyshev approximation. We regard $\bu=(\bu_{1},\dots,\bu_{M})$ as a tuple of functions. The intervals share their endpoints, but the approximations need not agree there. These mismatches contribute to the defect bound.

On $J_m$, the nonlinear remainder is
\begin{equation}\label{eq:Gm_local}
 (G_{m}(u_{m}))(t)\bydef{\cN}(u_m(t))-D\cN(\bu_m(t))u_{m}(t).
\end{equation}
Thus $DG_m(\bu_m)=0$. Each $G_m$ depends only on the trajectory on its own interval.

To obtain a rigorous enclosure on each step, we use the spaces
\begin{equation}
	X_m\bydef \ell^1\left(C^0(J_m)\right)=\left\{u:\|u\|_{X_m}\bydef\left\|\|u\|_{C^0(J_m)}\right\|_\nu<+\infty\right\},
\end{equation}
where the norm is obtained by taking the supremum of each Fourier coefficient and then applying the weighted sequence norm defined in Section~\ref{sec:function-space}. 
Around each approximation, we consider the ball of radius $r_m>0$
\begin{equation}\label{eq:local_validation_ball}
	B_{r_m}(\bu_{m})\bydef\{u_{m}\in X_m:\|u_{m}-\bu_{m}\|_{X_m}\le r_m\}.
\end{equation}
The product space is
\begin{equation}
	\bm{X}\bydef\prod_{m=1}^M \ell^1\left(C^0(J_m)\right)\left(=\prod_{m=1}^M X_m\right)\subset \prod_{m=1}^M C(J_m,\ell_\nu^1),
\end{equation}
The inclusion follows from $X_m\subset C(J_m,\ell_\nu^1)$, as in Section~\ref{sec:function-space}. With the norm $\|u\|_{\bm X}=\max_m\|u_m\|_{X_m}$, this finite product is a Banach space. From now on, $u=(u_1,\dots,u_M)$ denotes an arbitrary element of $\bm X$. Its components need not agree at shared endpoints. We define the fixed-point map recursively by
\begin{equation}
\begin{aligned}
	\bm{T}(u)&\bydef\left(T_m(u)\right)_{m=1}^M,\\
    \left(T_m(u)\right)(t)&\bydef U_m(t,\tau_{m-1})\sigma_{m-1}(u)+\int_{\tau_{m-1}}^{t}U_m(t,s)\cQ G_m(u_{m}(s))ds\quad (t\in J_m).
\end{aligned}\label{eq:fixed-point-form}
\end{equation}
Here $U_m(t,s)$ denotes the evolution operator generated by $\cL+\cQ D\cN(\bu_m(t))$. For $w_m\in X_m$, define the endpoint evaluation $e_m(w_{m})\bydef w_m(\tau_m)$. The endpoint of the output is
\begin{equation}\label{eq:sigma_recursion}
 \sigma_{m}(u)\bydef e_{m}(T_{m}(u)),\qquad \sigma_{0}(u)\bydef u_{0}.
\end{equation}
Equivalently, $\sigma_m(u)$ is obtained by setting $t=\tau_m$ in \eqref{eq:fixed-point-form}. It depends on $u_1,\dots,u_m$, whereas $e_m$ is simply evaluation of a single function.

\begin{remark}[{\bf Continuity of the recursive map output}]\label{rem:matching_traces}
{The components of $u$ and those of $\bu$ need not have matching endpoint values. The recursive definition gives
\begin{equation}\label{eq:output_matching}
 (T_m(u))(\tau_{m-1})=(T_{m-1}(u))(\tau_{m-1}),\qquad 2\le m\le M,
\end{equation}
and $(T_1(u))(0)=u_0$. Thus the output components join continuously, and a fixed point gives one continuous solution with the prescribed initial condition.}
\end{remark}

For $0\le j<i\le M$, define the endpoint propagator product by
\begin{equation}\label{eq:true-propagator}
 \Phi_{i,j}\bydef U_i(\tau_i,\tau_{i-1})U_{i-1}(\tau_{i-1},\tau_{i-2})\cdots U_{j+1}(\tau_{j+1},\tau_j),
 \quad\Phi_{j,j}\bydef\mathrm{Id}.
\end{equation}

Expanding the recurrence gives the contribution from the initial condition and from every source interval{, using these endpoint products}:
\begin{equation}\label{eq:Tm_unrolled}
\begin{aligned}
 (T_m(u))(t)
 &=U_m(t,\tau_{m-1})\left(\Phi_{m-1,0}u_0
 +\sum_{j=1}^{m-1}\Phi_{m-1,j}\int_{\tau_{j-1}}^{\tau_j}
 U_j(\tau_j,s)\cQ G_j(u_{j}(s))\,ds\right)\\
 &\quad+\int_{\tau_{m-1}}^t U_m(t,s)\cQ G_m(u_{m}(s))\,ds{,}
\end{aligned}
\end{equation}

{The product $\Phi_{i,j}$ propagates data from $\tau_j$ to $\tau_i$. In \eqref{eq:Tm_unrolled}, each integral contribution is propagated to the start of the current step and then to the observation time $t$.

For $s\in J_\ell$ and $t\in J_k$ with $\ell<k$, define
\begin{equation}\label{eq:piecewise_evolution}
 U(t,s)\bydef U_k(t,\tau_{k-1})\Phi_{k-1,\ell}U_\ell(\tau_\ell,s).
\end{equation}
For $s,t\in J_k$ with $s\le t$, use $U(t,s)=U_k(t,s)$. These definitions agree at shared endpoints and satisfy the composition law $U(t,r)U(r,s)=U(t,s)$ for $s\le r\le t$. The evolution uses the linearization along $\bu_k$ on each step and continues from the endpoint reached on the preceding step. }

We seek a fixed point of \eqref{eq:fixed-point-form} in the product of the balls $B_{r_m}(\bu_m)$ defined in \eqref{eq:local_validation_ball}. The Radii Polynomial Theorem gives conditions for this product to be invariant and for the map to be a contraction. We write $D_iT_m(u):X_i\to X_m$ for the derivative with respect to the $i$\,th component. By the recursive construction, $D_iT_m=0$ for $i>m$.

\begin{theorem}[{\bf Radii Polynomial Theorem}]\label{thm:RadiiPolynomial}
Assume that $\bm T:\bm X\to\bm X$ is continuously Fr\'echet differentiable and that $T_m$ depends only on $u_1,\dots,u_m$.

Let $\bu\in{\bm X}$. Assume there are constants $Y_m\ge{0}$ and non-decreasing functions $Z_m^i:[0,\infty)\to[0,\infty)$, for $1\le i\le m\le M$, such that
\begin{align}\label{eq:validation_hypotheses}
 \|T_m(\bu)-\bu_{m}\|_{X_m}&\le Y_m,\\
 \|D_{i}T_m(u)\|_{B(X_i,X_m)}&\le Z_m^i(r_i)
 \quad{\text{whenever }}\|u_{i}-\bu_{i}\|_{X_{i}}\le r_i.
\end{align}
For every radius $r_i>0$, the derivative bound is assumed to hold independently of the other components of $u$. If positive vectors $r=(r_m)$ and $\zeta=(\zeta_m)$ satisfy
\begin{align}\label{eq:validation_inequalities}
 Y_m+\sum_{i=1}^m\int_0^{r_i}Z_m^i(s)\,ds&\le r_m,\\
 \sum_{i=1}^m Z_m^i(r_i)\zeta_i&<\zeta_m
 \qquad (1\le m\le M),
\end{align}
then $\bm T$ has a unique fixed point in $\prod_{m=1}^M B_{r_m}(\bu_{m})$.
\end{theorem}
For the recursive map \eqref{eq:fixed-point-form}, the polynomial nonlinearity gives continuous Fr\'echet differentiability once the integral operators are bounded. We establish these bounds in Sections~\ref{sec:defect_derivative_bounds} and~\ref{sec: solution_operator}. The recursive dependence gives the triangular structure required by the theorem.
Separate radii allow different error sizes on different time steps. The two conditions in \eqref{eq:validation_inequalities} ensure invariance of the product of balls and contraction in a weighted product norm, respectively. The weights $\zeta_m$ need not equal the radii $r_m$. A closely related theorem and a very similar proof appear in \cite[Theorem~2.12]{JBMaxime}. Our proof is given in Appendix~\ref{app:validation_proofs}.

For the polynomial nonlinearity considered here, the derivative bounds in \eqref{eq:Zi_polynomial} and \eqref{eq:Zmi_bound} are polynomials in the radii. Hence the first condition is $p_m(r)\le0$, where
\begin{equation}\label{eq:radii_polynomials}
 p_m(r)\bydef Y_m+\sum_{j=1}^m\int_0^{r_j}Z_m^j(s)\,ds-r_m.
\end{equation}
These are the radii polynomials that give the theorem its name, although its statement also allows general non-decreasing derivative bounds.

\begin{remark}[{\bf Choosing the contraction weights}]
Define the nonnegative lower triangular matrix $\mathsf Z(r)$ by $[\mathsf Z(r)]_{m,i}=Z_{m}^{i}(r_{i})$ for $i\le m$ and zero otherwise. A positive vector $\zeta$ with $\mathsf Z(r)\zeta<\zeta$ exists exactly when $\rho(\mathsf Z(r))<1$, where $\rho$ denotes the spectral radius. For example, if this condition holds, $\zeta=(I-\mathsf Z(r))^{-1}{\mathbf1}$, where $\mathbf 1\bydef(1,\dots,1)^T\in\R^M$ denotes the vector of ones, is positive and satisfies the strict inequality. Since the matrix is lower triangular, its spectral radius is $\max_{m} Z_{m}^{m}(r_{m})$. Large off-diagonal entries still affect the invariant-ball inequalities and the size of the weights.
\end{remark}

To apply the theorem, we compute $Y_m$ and $Z_m^j$ using the defect and derivative estimates in Section~\ref{sec:defect_derivative_bounds}. The required linear propagation bounds are obtained in Section~\ref{sec: solution_operator}. It then remains to choose positive radii and weights satisfying \eqref{eq:validation_inequalities}.

\subsection{Defect and derivative bounds}\label{sec:defect_derivative_bounds}
In this subsection, we derive the bounds $Y_m$ and $Z_m^j$ required by Theorem~\ref{thm:RadiiPolynomial}, assuming that the linear operator bounds \eqref{eq:WVbounds} are available.

To state the defect and derivative bounds, we first introduce two linear operators. For $0\le j\le m-1$, let $W_{m,j}:\ell_\nu^1\to X_m$ be defined by
\begin{equation}\label{eq:W_operator}
 W_{m,j}{\phi}\bydef U_{m}(\cdot,\tau_{m-1})\Phi_{m-1,j}{\phi}.
\end{equation}
We next describe the contribution of an inhomogeneous term on $J_{j+1}$. Let $f\in X_{j+1}$ be a given function. We first define $V_{0,m,j}:X_{j+1}\to X_m$, which integrates $f$ and propagates its contribution to $J_m$, by
\begin{equation}\label{eq:V_operator}
 (V_{0,m,j}{f})(t)\bydef
 \begin{cases}
 U_m(t,\tau_{m-1})\Phi_{m-1,j+1}
 \displaystyle\int_{\tau_j}^{\tau_{j+1}}U_{j+1}(\tau_{j+1},s) f(s)\,ds,
 &j\le m-2,\\
 \displaystyle\int_{\tau_{m-1}}^{t} U_{m}(t,s) f(s)\,{d}{s},
 &j=m-1.
 \end{cases}
\end{equation}
Here $t\in J_m$. We also define $V_{q,m,j}:X_{j+1}\to X_m$ by replacing $f(s)$ with $\cQ f(s)$ in \eqref{eq:V_operator}. The subscript $q$ refers to the fixed parameter $q\in\{0,1\}$ in the PDE \eqref{eq:general_PDE}. Since $\cQ$ represents $\Delta^q$ in Fourier coordinates (see \eqref{eq:def_Q}), $V_{q,m,j}$ integrates $f$ when $q=0$ and the Laplacian of $f$ when $q=1$. In particular, $V_{q,m,j}=V_{0,m,j}$ when $q=0$. Both operators use the same evolution operators $U_m$.
When $q=1$, $\cQ f$ need not belong to $X_{j+1}$. Its integral is understood coefficientwise, or in the weaker sequence space from Section~\ref{sec:function-space}. The smoothing estimates in Section~\ref{sec:one-step} show that the output belongs to $X_m$. We write $V_{\star,m,j}$, with $\star\in\{0,q\}$, when discussing the two cases together.
Thus $W_{m,j}$ propagates an initial value at $\tau_j$, whereas $V_{\star,m,j}$ integrates an inhomogeneous term on $J_{j+1}$ and propagates its contribution to $J_m$. We distinguish these operators from their scalar upper bounds, denoted by $\bm{W^{(J_m,\tau_j)}}$ and $\bm{V_{\star}^{(J_m,\tau_j)}}$.

We seek computable constants satisfying
\begin{equation}\label{eq:WVbounds}
 \|{W}_{m,j}\|_{B(\ell_{\nu}^{1},X_{m})}\le{\bm{W^{(J_m,\tau_j)}}},\qquad
 \|{V}_{\star,m,j}\|_{B(X_{j+1},X_m)}\le{\bm{V_{\star}^{(J_m,\tau_j)}}}.
\end{equation}
The following estimates use these bounds with $\star=0$ for the defect and with $\star=q$ for the derivative. Section~\ref{sec: solution_operator} explains how to compute them.

\subsubsection{\texorpdfstring{The bound \boldmath $Y_m$\boldmath}{The bound Ym}}
There are two sources of defect. The differential residual measures how well $\bu_{j}$ satisfies the PDE inside $J_{j}$. The endpoint mismatch measures how well it matches the preceding approximation, or the prescribed initial condition on the first step. Define
\begin{equation}\label{eq:local_defects}
 \begin{aligned}
 y_1^j(s)&\bydef\bu_j'(s)-\cL\bu_j(s)-\cQ\cN(\bu_j(s)),\\
 y_{2}^{1}&\bydef{\bu}_{1}(0)-u_{0},\\
 y_2^j&\bydef\bu_j(\tau_{j-1})-\bu_{j-1}(\tau_{j-1})
 \quad (2\le j\le M).
 \end{aligned}
\end{equation}
The residual $y_1^j$ already includes all terms of the PDE. Its integral contribution is therefore described by $V_{0,m,j-1}$, without an additional factor $\cQ$.
Let $d_{m}\bydef{\bu}_{m}-T_{m}(\bu)$. In the recurrence below, the preceding-step contribution $e_{m-1}(d_{m-1})$ is taken to be zero when $m=1$. Applying the variation of constants formula to the residual gives \eqref{eq:defect_recursion}.
This is the residual counterpart of the fixed-point formula \eqref{eq:fixed-point-form}. Indeed, the definition of $y_1^m$ shows that $\bu_m$ satisfies the same linearized equation, with inhomogeneous term $\cQ G_m(\bu_m)+y_1^m$. Its initial value differs from the preceding output by $y_2^m+e_{m-1}(d_{m-1})$.

\begin{equation}\label{eq:defect_recursion}
 d_{m}(t)=U_{m}(t,\tau_{m-1})\bigl(y_{2}^{m}+e_{m-1}(d_{m-1})\bigr)
 +\int_{\tau_{m-1}}^{t} U_{m}(t,s)y_{1}^{m}(s)\,{d}{s}.
\end{equation}
Iterating \eqref{eq:defect_recursion}, with the propagation factors in \eqref{eq:Tm_unrolled}, expresses the endpoint and residual contributions through the operators in \eqref{eq:W_operator} and \eqref{eq:V_operator}:
\begin{equation}\label{eq:defect_unrolled}
 d_{m}=\sum_{j=0}^{m-1}\bigl(W_{m,j}{y}_{2}^{j+1}+V_{0,m,j}{y}_{1}^{j+1}\bigr).
\end{equation}
Thus, if $\|y_1^j\|_{X_j}\le Y_1^j$ and $\|y_2^j\|_\nu\le Y_2^j$, we may take
\begin{equation}\label{eq:Ym_bound}
 Y_m\bydef\sum_{j=0}^{m-1}
 \left(\bm{W^{(J_m,\tau_j)}}Y_2^{j+1}
 +\bm{V_0^{(J_m,\tau_j)}}Y_1^{j+1}\right).
\end{equation}
Even when $\bu_{j}$ has finite Fourier support, its nonlinear residual can have modes outside the numerical truncation. These modes must be included in $Y_1^{j}$.

\subsubsection{\texorpdfstring{The bound \boldmath $Z_m^i$\boldmath}{The bound Zmi}}
We now bound the effect of changing one trajectory component while keeping the initial condition fixed. Since $D{G}_{i}(\bu_{i})=0$, this effect becomes small near the approximation. We first bound the local nonlinear derivative by $Z_i(r_i)$. Propagation to the observation step will then give $Z_m^i(r_i)$. Suppose
\begin{equation}\label{eq:Z2-bounds}
 \|DG_{i}(u_{i})\|_{B(X_{i})}\le Z_{i}(r_{i})
 \quad{\text{for }}u_{i}\in B_{r_{i}}(\bu_{i}).
\end{equation}
For the polynomial nonlinearity of Section~\ref{sec:functional_analytic_setup}, a computable choice is
\begin{equation}\label{eq:Zi_polynomial}
 Z_{i}(r)\bydef\sum_{p=2}^{P} p|c_{p}|
 \left((\|{\bu}_{i}\|_{X_{i}}+r)^{p-1}-\|{\bu}_{i}\|_{X_{i}}^{p-1}\right).
\end{equation}
Write $u_i=\bu_i+w_i$. Indeed, expanding $D{\cN}(\bu_{i}+w_{i})-D{\cN}(\bu_{i})$ and applying the Banach-algebra inequality bounds every term by its corresponding power of $\|{\bu}_{i}\|_{X_{i}}$ and $\|{w}_{i}\|_{X_{i}}$. In particular, $Z_{i}(0)=0$. For $N(z)=-z^{3}$, this gives $Z_{i}(r)=6\|{\bu}_{i}\|_{X_{i}}{r}+3{r}^2$.

For $h_{i}\in X_{i}$, differentiating the unrolled fixed-point formula gives
\begin{equation}\label{eq:partial_Tm}
 D_{i}T_m(u)h_{i}=V_{q,m,i-1}\bigl(DG_{i}(u_{i})h_{i}\bigr)
 \qquad (1\le i\le m).
\end{equation}
Here $V_{q,m,i-1}$ includes the factor $\cQ$ acting on the nonlinear derivative $DG_i(u_i)h_i$.
When $i<m$, the inhomogeneous term is integrated over the whole source step and then propagated forward. When $i=m$, it is integrated only up to the observation time $t$. This distinction is already included in \eqref{eq:V_operator}. Consequently,
\begin{equation}\label{eq:Zmi_bound}
 \|D_{i}T_m(u)\|_{B(X_i,X_m)}
 \le{\bm{V_q^{(J_m,\tau_{i-1})}}}Z_i(r_i)
 \bydef Z_m^i(r_i).
\end{equation}
The superscript $i$ identifies the source step and the subscript $m$ identifies the observation step. These bounds complete the contraction test once the required operator bounds have been computed.

\subsection{Linear propagation bounds}\label{sec: solution_operator}

In this subsection, we derive computable bounds satisfying \eqref{eq:WVbounds}. Let $0\le j<m\le M$ and $\star\in\{0,q\}$. For $\phi\in\ell_\nu^1$ and $f\in X_{j+1}$, we consider the mild solution $v$ on $[\tau_j,\tau_m]$ of
\begin{equation}\label{eq:model_linear_problem}
 v'(t)=\bigl(\cL+\cQ D\cN(\bu_k(t))\bigr)v(t)+\ichi_{J_{j+1}}(t)\,\cQ_\star f(t)\quad(t\in J_k,\ j<k\le m),\qquad v(\tau_j)=\phi,
\end{equation}
where $\cQ_0\bydef I$ and $\cQ_q\bydef\cQ$. Similarly, we set $Q_{\bbm,0}\bydef I$ and $Q_{\bbm,q}\bydef Q_{\bbm}$, where $Q_{\bbm}$ is defined below \eqref{eq:psi_def}. From the variation of constants formula, we have $v|_{J_m}=W_{m,j}\phi+V_{\star,m,j}f$. Note that the inhomogeneous term acts only on $J_{j+1}$, and that $v$ solves the homogeneous equation on $J_{j+2},\dots,J_m$ (see Figure~\ref{fig:propagation_timeline}).

If these bounds were obtained by taking operator norms on each time step, the cancellations between Fourier modes would be lost, and the resulting bounds would grow geometrically with the number of time steps (see \cite[Section~6]{MR4852997}). To avoid this, we proceed as follows. In Section~\ref{sec:one-step}, we bound each Fourier mode of a finite set separately and the infinite tail as a whole. In Section~\ref{sec:multistep}, we multiply the signed finite-dimensional propagators over several time steps before taking absolute values, and we bound the coupling between the finite part and the tail as an additive correction. These bounds are combined in Section~\ref{sec:propagation_assembly}.

\subsubsection{One step}\label{sec:one-step}

We begin with the case $j=m-1$, in which $v$ solves
\begin{equation}\label{eq:linear_one_step}
 v'=\bigl(\cL+\cQ D\cN(\bu_m)\bigr)v+\cQ_\star f\quad\text{on }J_m,\qquad v(\tau_{m-1})=\phi.
\end{equation}

To bound $v$, we let $\bbm=(K_1,\dots,K_d)$ be a Fourier truncation size such that, with $\FK\bydef\{\bk\ge0:\bk<\bbm\}$,
\begin{equation}\label{eq:mu_bbm_def}
 \mu_{\bbm}\bydef\sup_{\bk\notin\FK}\mu_{\bk}<0.
\end{equation}
The same $\bbm$ is used on each time step. Any sequence $a$ is uniquely decomposed as $a=a^{(\bbm)}+a^{(\infty)}$, where $a^{(\bbm)}$ contains the modes $\bk\in\FK$ and $a^{(\infty)}$ the remaining ones. We denote $\m\bydef K_1\cdots K_d$ and $D_\nu\bydef\operatorname{diag}(\alpha_{\bk}\nu^{\bk})_{\bk\in\FK}$. In order to bound each mode in $\FK$ separately, we introduce the nonnegative vectors
\begin{equation}\label{eq:profile_vector}
 z(\phi)\bydef\begin{pmatrix}D_\nu|\phi^{(\bbm)}|\\\|\phi^{(\infty)}\|_\nu\end{pmatrix},
 \qquad
 z(w)\bydef\begin{pmatrix}D_\nu|w^{(\bbm)}|_\infty\\\|w^{(\infty)}\|_{X_k}\end{pmatrix}\in\R_+^{\m+1}
 \qquad(\phi\in\ell_\nu^1,\ w\in X_k),
\end{equation}
where $|\cdot|$ and the supremum $|\cdot|_\infty$ over $J_k$ are taken componentwise. Note that $z$ is subadditive and that
\begin{equation}\label{eq:profile_properties}
 \|z(\phi)\|_1=\|\phi\|_\nu,\qquad \|z(w)\|_1=\|w\|_{X_k}.
\end{equation}
Throughout this section, inequalities between vectors and matrices are understood componentwise.

We first bound the finite part $v^{(\bbm)}$. Let $A_{\bbm}(t)\in\R^{\m\times\m}$ denote the matrix representation of $\cL+\cQ D\cN(\bu_m(t))$ restricted to $\FK$, and let $C_m^{(\bbm)}(t,s)$ be its fundamental solution matrix. From the variation of constants formula, we have
\begin{equation}\label{eq:vm_coupled}
 v^{(\bbm)}(t)=C_m^{(\bbm)}(t,\tau_{m-1})\phi^{(\bbm)}+\int_{\tau_{m-1}}^tC_m^{(\bbm)}(t,s)\Bigl(\bigl(\cQ D\cN(\bu_m(s))v^{(\infty)}(s)\bigr)^{(\bbm)}+(\cQ_\star f)^{(\bbm)}(s)\Bigr)ds.
\end{equation}
The first term is bounded by using a rigorous enclosure of the fundamental solution matrix by a Chebyshev series expansion in time \cite[Section~3.1.1]{MR4852997}, which yields
\begin{equation}\label{eq:Phi_sup_m}
 \Phi_{\sup,m}\bydef D_\nu\,|C_m^{(\bbm)}(\cdot,\tau_{m-1})|_\infty\,D_\nu^{-1}\in\R_+^{\m\times\m}.
\end{equation}
Estimating the integral term in the same way would require rigorous control of $C_m^{(\bbm)}(t,s)$ with respect to $s$, that is, a backward integration of the truncated system, which is generally numerically unstable. Following \cite{JBMaxime}, we instead compare $A_{\bbm}(t)$ with a constant matrix. Let $(A_{\bbm})_0$ be the matrix of zeroth Chebyshev coefficients of $A_{\bbm}$ on $J_m$, and let $\bA\bydef V\Lambda V^{-1}$, where $(A_{\bbm})_0\approx V\Lambda V^{-1}$ is a numerical eigenvalue decomposition with $\Lambda=\operatorname{diag}(\theta_\ell)$. The solution of $c'=A_{\bbm}(t)c+b$, $c(\tau_{m-1})=0$, is given by $c(t)=\int_{\tau_{m-1}}^te^{\bA(t-s)}\bigl((A_{\bbm}(s)-\bA)c(s)+b(s)\bigr)ds$, and it follows that
\begin{equation}\label{eq:c1-estimate}
 |c|_\infty\le\tV^{(\bbm)}_m\triangle_{\sup}|c|_\infty+\widehat V^{(\bbm)}_m|b|_\infty,
\end{equation}
Here, we verify that $V$ is invertible, $\triangle_{\sup}\in\R_+^{\m\times\m}$ satisfies $|V^{-1}(A_{\bbm}(t)-\bA)|\le\triangle_{\sup}$ for $t\in J_m$, and, with $\Theta_m\bydef\operatorname{diag}\bigl(\int_0^{h_m}e^{\mathrm{Re}\,\theta_\ell r}dr\bigr)_\ell$,
\begin{equation}\label{eq:bV-esitmates}
 \tV^{(\bbm)}_m\bydef|V|\,\Theta_m,\qquad \widehat V^{(\bbm)}_m\bydef|V|\,\Theta_m\,|V^{-1}|.
\end{equation}
Define
\begin{equation}\label{eq:Neumann-correction}
 B_{\mathrm{neu},m}\bydef I-\tV^{(\bbm)}_m\triangle_{\sup}
\end{equation}
and, if $B_{\mathrm{neu},m}$ is invertible,
\begin{equation}\label{eq:Vk0-estimate}
 \mathcal V_{\bbm,m}\bydef B_{\mathrm{neu},m}^{-1}\widehat V^{(\bbm)}_m\in\R_+^{\m\times\m}.
\end{equation}
\begin{lemma}[{\bf Finite-mode integral bound}]\label{lem:finite_integral}
Assume that $V$ is invertible, that $|V^{-1}(A_{\bbm}(t)-\bA)|\le\triangle_{\sup}$ for $t\in J_m$, and that $B_{\mathrm{neu},m}$ is invertible with $B_{\mathrm{neu},m}^{-1}\ge0$. Then the solution of $c'=A_{\bbm}(t)c+b$ on $J_m$ with $c(\tau_{m-1})=0$ satisfies $|c|_\infty\le\mathcal V_{\bbm,m}|b|_\infty$.
\end{lemma}
\begin{proof}
Write $e^{\bA r}=Ve^{\Lambda r}V^{-1}$ in the integral representation of $c$ above. Taking absolute values and using $\int_{\tau_{m-1}}^t|e^{\theta_\ell(t-s)}|\,ds\le\int_0^{h_m}e^{\mathrm{Re}\,\theta_\ell r}\,dr$ gives \eqref{eq:c1-estimate}, that is, $B_{\mathrm{neu},m}|c|_\infty\le\widehat V^{(\bbm)}_m|b|_\infty$. Multiplying by $B_{\mathrm{neu},m}^{-1}\ge0$ preserves this inequality and proves the claim.
\end{proof}
Note that only the forward fundamental solution matrix and explicit exponentials are used. The difference $A_{\bbm}-\bA$ appears only in $B_{\mathrm{neu},m}$ (compare with the condition $Z<1$ in \cite{JBMaxime2}).

Next, we bound the coupling term in \eqref{eq:vm_coupled}. For each $\bk\in\FK$, the contribution of the tail to the $\bk$\,th mode is bounded by
\begin{equation}\label{eq:psi_def}
 \bigl|\bigl(\cQ D\cN(\bu_m(t))a^{(\infty)}\bigr)_{\bk}\bigr|\le\qk[\Psi_m]_{\bk}\|a^{(\infty)}\|_\nu,
 \qquad
 [\Psi_m]_{\bk}\bydef\sup_{\bj\notin\FK}\frac{|D\cN(\bu_m)_{\bk,\bj}|_\infty}{\alpha_{\bj}\nu^{\bj}}.
\end{equation}
Note that the supremum is taken over finitely many modes, since $D\cN(\bu_m)$ has finite bandwidth. Setting $Q_{\bbm}\bydef\operatorname{diag}(\qk)_{\bk\in\FK}$, $\mathcal V_{\bbm,m}^{(w)}\bydef D_\nu\mathcal V_{\bbm,m}D_\nu^{-1}$, and
\begin{equation}\label{eq:cin_def}
 c_{\mathrm{in},m}\bydef D_\nu\,\mathcal V_{\bbm,m}Q_{\bbm}\Psi_m\in\R_+^{\m},
\end{equation}
we obtain from \eqref{eq:vm_coupled}
\begin{equation}\label{eq:vbbm_abs_bound}
 D_\nu|v^{(\bbm)}|_\infty\le\Phi_{\sup,m}D_\nu|\phi^{(\bbm)}|+c_{\mathrm{in},m}\|v^{(\infty)}\|_{X_m}+\mathcal V_{\bbm,m}^{(w)}Q_{\bbm,\star}D_\nu|f^{(\bbm)}|_\infty.
\end{equation}

Next, we bound the tail $v^{(\infty)}$. Let $\cL_\infty$ be the restriction of $\cL$ to the tail. Since $\cL_\infty$ is diagonal, $\|e^{\cL_\infty t}a^{(\infty)}\|_\nu\le e^{\mu_{\bbm}t}\|a^{(\infty)}\|_\nu$ holds. The variation of constants formula for the tail gives
\[
 v^{(\infty)}(t)=e^{\cL_\infty(t-\tau_{m-1})}\phi^{(\infty)}+\int_{\tau_{m-1}}^te^{\cL_\infty(t-s)}\Bigl(\bigl(\cQ D\cN(\bu_m(s))v(s)\bigr)^{(\infty)}+(\cQ_\star f)^{(\infty)}(s)\Bigr)ds.
\]
We define the constants
\begin{equation}\label{eq:Vinfq}
 V_{\infty,0}^{(J_m,\tau_{m-1})}\bydef\sup_{\bk\notin\FK}\int_0^{h_m}e^{\mu_{\bk}r}dr,\qquad
 V_{\infty,q}^{(J_m,\tau_{m-1})}\bydef\sup_{\bk\notin\FK}\int_0^{h_m}e^{\mu_{\bk}r}\qk\,dr,
\end{equation}
which are finite since the order of $\cQ$ is strictly lower than that of $\cL$. Here, the smoothing effect of the semigroup compensates for the loss of regularity in the nonlinear term (see Appendix~\ref{app:one_step_details} for the Ohta--Kawasaki equation). We bound the contributions of the finite part and of the tail to the tail by
\begin{equation}\label{eq:cout_def}
 c_{\mathrm{out},m}^T\bydef V_{\infty,q}^{(J_m,\tau_{m-1})}\varPhi_m^TD_\nu^{-1},\quad
 [\varPhi_m]_{\bj}\bydef\sum_{\bk\notin\FK}|D\cN(\bu_m)_{\bk,\bj}|_\infty\alpha_{\bk}\nu^{\bk},
\end{equation}
\begin{equation}\label{eq:g22_def}
 g_{22,m}\bydef V_{\infty,q}^{(J_m,\tau_{m-1})}\,g\bigl(\|\bu_m\|_{X_m}\bigr),
\end{equation}
where $g$ is defined in \eqref{eq:N_assumption}. Note that the sum defining $\varPhi_m$ has finitely many nonzero terms. Hence, we obtain
\begin{equation}\label{eq:vinf_abs_bound}
 \|v^{(\infty)}\|_{X_m}\le\|\phi^{(\infty)}\|_\nu+c_{\mathrm{out},m}^TD_\nu|v^{(\bbm)}|_\infty+g_{22,m}\|v^{(\infty)}\|_{X_m}+V_{\infty,\star}^{(J_m,\tau_{m-1})}\|f^{(\infty)}\|_{X_m}.
\end{equation}

Finally, combining \eqref{eq:vbbm_abs_bound} and \eqref{eq:vinf_abs_bound}, we obtain (see Lemma~\ref{lem:finite_tail} below)
\begin{equation}\label{eq:coupled_components}
 z(v)\le\bU_mz(\phi)+\mathsf G_mz(v)+\bV_{\star,m}z(f),
\end{equation}
where the nonnegative $(\m+1)\times(\m+1)$ matrices are defined by
\begin{equation}\label{eq:bUm_bVm_def}
 \bU_m\bydef\begin{pmatrix}\Phi_{\sup,m}&0\\0&1\end{pmatrix},\qquad
 \bV_{\star,m}\bydef\begin{pmatrix}\mathcal V_{\bbm,m}^{(w)}Q_{\bbm,\star}&0\\0&V_{\infty,\star}^{(J_m,\tau_{m-1})}\end{pmatrix},
\end{equation}
\begin{equation}\label{eq:Gm_def}
 \mathsf G_m\bydef\begin{pmatrix}0&c_{\mathrm{in},m}\\c_{\mathrm{out},m}^T&g_{22,m}\end{pmatrix}.
\end{equation}
\begin{lemma}[{\bf Finite-tail inequality}]\label{lem:finite_tail}
Under the assumptions of Lemma~\ref{lem:finite_integral}, the solution $v$ of \eqref{eq:linear_one_step} satisfies \eqref{eq:coupled_components}.
\end{lemma}
\begin{proof}
For the finite part, we bound the first term of \eqref{eq:vm_coupled} by \eqref{eq:Phi_sup_m}, and the integral term by Lemma~\ref{lem:finite_integral} with $b=\bigl(\cQ D\cN(\bu_m)v^{(\infty)}\bigr)^{(\bbm)}+(\cQ_\star f)^{(\bbm)}$, together with \eqref{eq:psi_def}. This gives \eqref{eq:vbbm_abs_bound}. For the tail, we use the variation of constants formula for $v^{(\infty)}$, the bound $\|e^{\cL_\infty t}a^{(\infty)}\|_\nu\le e^{\mu_{\bbm}t}\|a^{(\infty)}\|_\nu$, the definitions \eqref{eq:Vinfq}, \eqref{eq:cout_def}, and \eqref{eq:g22_def}, and the estimate \eqref{eq:N_assumption}. This gives \eqref{eq:vinf_abs_bound}. Combining the two bounds proves \eqref{eq:coupled_components}.
\end{proof}
The upper left block of $\mathsf G_m$ vanishes, since the interactions between the finite modes are already included in $C_m^{(\bbm)}$. Hence, $I-\mathsf G_m$ has a nonnegative inverse if and only if the Schur complement
\begin{equation}\label{eq:schur_def}
 s_m\bydef1-g_{22,m}-c_{\mathrm{out},m}^Tc_{\mathrm{in},m}
\end{equation}
is positive. In this case, the inverse is given explicitly by
\begin{equation}\label{eq:Bm_schur}
 B_m\bydef(I-\mathsf G_m)^{-1}=\begin{pmatrix}I+s_m^{-1}c_{\mathrm{in},m}c_{\mathrm{out},m}^T&s_m^{-1}c_{\mathrm{in},m}\\s_m^{-1}c_{\mathrm{out},m}^T&s_m^{-1}\end{pmatrix}\ge0.
\end{equation}
Note that $s_m>0$ if and only if $g_{22,m}+c_{\mathrm{out},m}^Tc_{\mathrm{in},m}<1$.

\begin{proposition}[{\bf One-step bound}]\label{prop:one_step}
If $B_{\mathrm{neu},m}^{-1}\ge0$ and $s_m>0$, then the solution of \eqref{eq:linear_one_step} satisfies
\begin{equation}\label{eq:onestep_resolved}
 z(v)\le\mathsf U_m\,z(\phi)+\mathsf V_{\star,m}\,z(f),\qquad
 \mathsf U_m\bydef B_m\bU_m,\quad \mathsf V_{\star,m}\bydef B_m\bV_{\star,m},
\end{equation}
and the same bound holds for $z(v(\tau_m))$.
\end{proposition}

\begin{proof}
By Lemma~\ref{lem:finite_tail}, \eqref{eq:coupled_components} holds, so that $(I-\mathsf G_m)z(v)\le\bU_mz(\phi)+\bV_{\star,m}z(f)$. Multiplying this inequality by $B_m\ge0$, we obtain the first bound. The endpoint bound holds because the entries of $z(v)$ are suprema over $J_m\ni\tau_m$.
\end{proof}

\subsubsection{Several steps}\label{sec:multistep}

Iterating Proposition~\ref{prop:one_step} would bound $U(\tau_i,\tau_j)$ by a product of the nonnegative matrices $\mathsf U_k$, and the cancellations between Fourier modes would be lost. To avoid this, we decompose the evolution operator on each time step as
\begin{equation}\label{eq:Ek_def}
 E_k\bydef\operatorname{diag}\bigl(C_k^{(\bbm)}(\tau_k,\tau_{k-1}),e^{\cL_\infty h_k}\bigr),\qquad
 \widetilde C_k\bydef U_k(\tau_k,\tau_{k-1})-E_k,
\end{equation}
and use the telescoping identity
\begin{equation}\label{eq:telescoping}
 U(\tau_i,\tau_j)-E_i\cdots E_{j+1}=\sum_{k=j+1}^iE_i\cdots E_{k+1}\,\widetilde C_k\,U(\tau_{k-1},\tau_j).
\end{equation}
The product $E_i\cdots E_{j+1}$ is computed without taking absolute values, while the remainders $\widetilde C_k$, which contain the coupling between the finite part and the tail, are bounded through their block norms.

We first bound the product $E_i\cdots E_{j+1}$. Let $\Phi_{\mathrm{end},k}\in\mathbb{IR}^{\m\times\m}$ be an interval enclosure of the fundamental solution matrix at the endpoint,
\begin{equation}\label{eq:Phi_def}
 C_k^{(\bbm)}(\tau_k,\tau_{k-1})\in\Phi_{\mathrm{end},k},
\end{equation}
which is obtained from the same Chebyshev enclosure as $\Phi_{\sup,k}$ or as in \eqref{eq:Phi_endpt_enclosure}. We define the signed interval products
\begin{equation}\label{eq:bPhi}
 \Phi_{\mathrm{end},i,j}\bydef\Phi_{\mathrm{end},i}\cdots\Phi_{\mathrm{end},j+1},\qquad\Phi_{\mathrm{end},j,j}\bydef I,
\end{equation}
and take absolute values only after the products are formed. Denoting by $\operatorname{mag}(A)$ an entrywise upper bound of the absolute values of an interval matrix $A$, and setting
\begin{equation}\label{eq:theta_def}
 \varsigma_{i,j}\bydef e^{\mu_{\bbm}(\tau_i-\tau_j)},
\end{equation}
we bound the product $E_i\cdots E_{j+1}$ by means of
\begin{equation}\label{eq:R_uncoupled}
 R^{(\mathrm{unc})}_{i,j}\bydef\operatorname{diag}\bigl(D_\nu\operatorname{mag}(\Phi_{\mathrm{end},i,j})D_\nu^{-1},\varsigma_{i,j}\bigr).
\end{equation}
A similar approach is used in \cite[Definition~4.3]{JBMaxime2}.

Next, we bound the remainders $\widetilde C_k$. For a block operator or matrix $A$ with respect to the decomposition into the finite part and the tail, and for $x\in\R_+^{\m+1}$, we define
\begin{equation}\label{eq:block_norm_matrix}
 \mathfrak b(A)\bydef\begin{pmatrix}\|A_{11}\|_1&\|A_{12}\|_1\\\|A_{21}\|_1&\|A_{22}\|_1\end{pmatrix},\qquad
 \operatorname{pair}(x)\bydef\begin{pmatrix}\sum_{\bk\in\FK}x_{\bk}\\x_\infty\end{pmatrix},
\end{equation}
where the norms are induced by the weighted finite space and the tail space (for matrices acting on the vectors $z$, the matrix $1$-norm is used). The map $\mathfrak b$ is subadditive and submultiplicative, and $\operatorname{pair}(Nx)\le\mathfrak b(N)\operatorname{pair}(x)$ holds for $N\ge0$. The block norms of $E_i\cdots E_{j+1}$ are bounded by
\begin{equation}\label{eq:R_uncoupled_blocks}
 \bar E_{i,j}\bydef\mathfrak b\bigl(R^{(\mathrm{unc})}_{i,j}\bigr)=\operatorname{diag}\bigl(\mathfrak n(i,j),\varsigma_{i,j}\bigr),\qquad
 \mathfrak n(i,j)\bydef\|D_\nu\operatorname{mag}(\Phi_{\mathrm{end},i,j})D_\nu^{-1}\|_1.
\end{equation}
From the variation of constants formula on $J_k$, $\widetilde C_k\phi$ is the value at $\tau_k$ of the terms involving $\mathsf G_k$ in \eqref{eq:coupled_components}, and its block norms are bounded by
\begin{equation}\label{eq:Delta_k_def}
 \widetilde\triangle_k\bydef\mathfrak b(\mathsf G_k\mathsf U_k).
\end{equation}
Motivated by \eqref{eq:telescoping}, we define recursively
\begin{equation}\label{eq:decomposition}
 \widetilde\triangle_{i,j}\bydef\sum_{k=j+1}^i\bar E_{i,k}\,\widetilde\triangle_k\,\bigl(\bar E_{k-1,j}+\widetilde\triangle_{k-1,j}\bigr),\qquad\widetilde\triangle_{j,j}\bydef0,
\end{equation}
where the three factors bound the block norms of $E_i\cdots E_{k+1}$, $\widetilde C_k$, and $U(\tau_{k-1},\tau_j)$, respectively.

\begin{lemma}[{\bf Endpoint propagation}]\label{lem:endpoint_propagation}
Let $0\le j<i\le M$, and assume the hypotheses of Proposition~\ref{prop:one_step} and \eqref{eq:Phi_def} on each time step $J_k$, $j<k\le i$. Then $z(E_i\cdots E_{j+1}\phi)\le R^{(\mathrm{unc})}_{i,j}z(\phi)$ for all $\phi\in\ell_\nu^1$, $\mathfrak b(\widetilde C_k)\le\widetilde\triangle_k$, and
\begin{equation}\label{eq:product_remainder_bound}
 \mathfrak b\bigl(U(\tau_i,\tau_j)-E_i\cdots E_{j+1}\bigr)\le\widetilde\triangle_{i,j},
\end{equation}
\begin{equation}\label{eq:full_endpoint_bound}
 \mathfrak b\bigl(U(\tau_i,\tau_j)\bigr)\le\bar E_{i,j}+\widetilde\triangle_{i,j}.
\end{equation}
\end{lemma}

The proof is given in Appendix~\ref{app:multistep_proof}. The one-step quantities can be computed independently. Forming all signed finite endpoint products over $M$ steps requires $O(M^2)$ finite-matrix multiplications, while the coupling corrections are assembled separately through the recurrence \eqref{eq:decomposition}.

\subsubsection{\texorpdfstring{The bounds \boldmath$W^{(J_m,\tau_j)}$\boldmath~and \boldmath$V_\star^{(J_m,\tau_j)}$\boldmath}{The bounds W and V}}\label{sec:propagation_assembly}

We now combine the above bounds (see Figure~\ref{fig:propagation_timeline}). The initial data at $\tau_j$ are propagated to $\tau_{m-1}$ and then bounded on $J_m$ by $\mathsf U_m$, while the inhomogeneous term is bounded on $J_{j+1}$ by $\mathsf V_{\star,j+1}$ and then propagated in the same way. Note that the uncoupled products are multiplied by the one-step matrices before the block norms are taken, so that the bound on each Fourier mode is kept, whereas the corrections are bounded through their block norms.

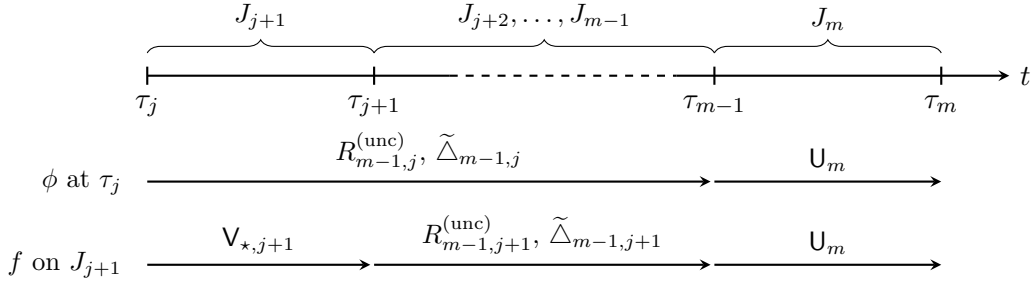
\begin{figure}[h!]
\centering
\begin{tikzpicture}[x=1cm,y=1cm,>=stealth]
 \draw[thick] (0,0) -- (4,0);
 \draw[thick,dashed] (4,0) -- (7,0);
 \draw[thick,->] (7,0) -- (11.4,0) node[right] {$t$};
 \foreach \x/\lab in {0/{\tau_j},3/{\tau_{j+1}},7.5/{\tau_{m-1}},10.5/{\tau_m}}{
   \draw[thick] (\x,0.12) -- (\x,-0.12) node[below] {$\lab$};}
 \draw[decorate,decoration={brace,amplitude=5pt}] (0,0.3) -- (3,0.3)
   node[midway,above=5pt] {\small $J_{j+1}$};
 \draw[decorate,decoration={brace,amplitude=5pt}] (3,0.3) -- (7.5,0.3)
   node[midway,above=5pt] {\small $J_{j+2},\dots,J_{m-1}$};
 \draw[decorate,decoration={brace,amplitude=5pt}] (7.5,0.3) -- (10.5,0.3)
   node[midway,above=5pt] {\small $J_m$};
 \node[left] at (-0.2,-1.4) {\small $\phi$ at $\tau_j$};
 \draw[->,thick] (0,-1.4) -- node[above]{\small $R^{(\mathrm{unc})}_{m-1,j}$, $\widetilde\triangle_{m-1,j}$} (7.45,-1.4);
 \draw[->,thick] (7.5,-1.4) -- node[above]{\small $\mathsf U_m$} (10.5,-1.4);
 \node[left] at (-0.2,-2.5) {\small $f$ on $J_{j+1}$};
 \draw[->,thick] (0,-2.5) -- node[above]{\small $\mathsf V_{\star,j+1}$} (2.95,-2.5);
 \draw[->,thick] (3,-2.5) -- node[above]{\small $R^{(\mathrm{unc})}_{m-1,j+1}$, $\widetilde\triangle_{m-1,j+1}$} (7.45,-2.5);
 \draw[->,thick] (7.5,-2.5) -- node[above]{\small $\mathsf U_m$} (10.5,-2.5);
\end{tikzpicture}
\caption{Construction of the bounds on $W_{m,j}$ (top) and $V_{\star,m,j}$ (bottom) for $j\le m-2$. The matrices $\mathsf U_m$ and $\mathsf V_{\star,j+1}$ are defined in Section~\ref{sec:one-step}, and $R^{(\mathrm{unc})}$ and $\widetilde\triangle$ in Section~\ref{sec:multistep}.}\label{fig:propagation_timeline}
\end{figure}

\begin{proposition}[{\bf Multi-step propagator bounds}]\label{prop:multistep_bounds}
Assume the hypotheses of Lemma~\ref{lem:endpoint_propagation} on each time step $J_k$, $j<k\le m$. Then \eqref{eq:WVbounds} holds with
\begin{align}
 \bm{W^{(J_m,\tau_j)}}&\bydef\bigl\|\mathfrak b\bigl(\mathsf U_mR^{(\mathrm{unc})}_{m-1,j}\bigr)+\mathfrak b(\mathsf U_m)\,\widetilde\triangle_{m-1,j}\bigr\|_1,\label{eq:SR_split}\\
 \bm{V_\star^{(J_m,\tau_j)}}&\bydef\bigl\|\mathfrak b\bigl(\mathsf U_mR^{(\mathrm{unc})}_{m-1,j+1}\mathsf V_{\star,j+1}\bigr)+\mathfrak b(\mathsf U_m)\,\widetilde\triangle_{m-1,j+1}\,\mathfrak b(\mathsf V_{\star,j+1})\bigr\|_1\label{eq:SRV_split}
\end{align}
for $j\le m-2$, and
\begin{equation}\label{eq:WV_onestep}
 \bm{W^{(J_m,\tau_{m-1})}}\bydef\|\mathfrak b(\mathsf U_m)\|_1,\qquad\bm{V_\star^{(J_m,\tau_{m-1})}}\bydef\|\mathfrak b(\mathsf V_{\star,m})\|_1.
\end{equation}
\end{proposition}

The proof is given in Appendix~\ref{app:multistep_proof}. For $\star=0$, these bounds are used in $Y_m$ in \eqref{eq:Ym_bound}, and for $\star=q$, in $Z_m^i$ in \eqref{eq:Zmi_bound}. In practice, $\bbm$ is chosen sufficiently large that $s_k>0$, since $c_{\mathrm{out},k}$ and $g_{22,k}$ are proportional to $V_{\infty,q}^{(J_k,\tau_{k-1})}$, which decreases as $\bbm$ increases. The step sizes are chosen sufficiently small that $B_{\mathrm{neu},k}^{-1}\ge0$, since $\tV^{(\bbm)}_k\le h_k|V|$ when $\mathrm{Re}\,\theta_\ell\le0$.

\section{Global existence via a fixed-point argument}\label{sec:global_in_time}

In this section, we extend the solution obtained in
Section~\ref{sec:multi-stepping} to all $t\ge0$. Let $\bu_*\in\ell_\nu^1$
be a numerical approximation of a stable equilibrium. For simplicity
of presentation, we do not assume that $f(\bu_*)=0$ (see
Remark~\ref{rem:approximate_equilibrium} below). We first formulate a
fixed-point problem about $\bu_*$ on an unbounded interval. If the
endpoint of the finite-time solution is too far from $\bu_*$ to apply
this argument, we add finite time steps and enclose the solution on
each of them. 
The finite-time trajectory on $[0,\tau_M]$ is validated on the product of its $M$ step spaces. Its verified endpoint enclosure then supplies the incoming data for the intermediate steps, which are validated sequentially when needed. Finally, the resulting endpoint enclosure is used in a separate contraction argument on the unbounded interval. The required linear and nonlinear bounds are verified at the numerical approximations; no separate proof of existence or stability of the equilibrium is assumed.
% The validation proceeds sequentially: the verified endpoint
% enclosure from each stage supplies the incoming data for the next. The required linear and nonlinear bounds are verified
% at the numerical approximations; no separate proof of existence or
% stability of the equilibrium is assumed.
\begin{remark}[{\bf Choice of an approximate equilibrium}]\label{rem:approximate_equilibrium}
A separate enclosure of an equilibrium can be obtained by
computer-assisted methods such as \cite{MR2338393,MR2718657}.
Here we work directly with $\bu_*$ and include $f(\bu_*)$ in
\eqref{eq:Y_global}. The resulting proof gives global existence and
a uniform bound about $\bu_*$.
\end{remark}
A fixed-point argument on an unbounded interval also appears in
\cite{JBMaxime2}. We use the evolution bounds of
Section~\ref{sec:multi-stepping}, in place of the trapping region
used in \cite{MR4852997}. The intermediate steps are described after
the infinite interval, since they are used to meet its initial error bound; the overall structure is shown in Figure~\ref{fig:global_timeline}. From Section~\ref{sec:aut_chain} on, we restrict to the cubic nonlinearity $\cN(u)=\pm u^3$ of Section~\ref{sec:OK_application}; for a general polynomial, the cubic remainder bounds below are replaced by those obtained from \eqref{eq:Zi_polynomial}.

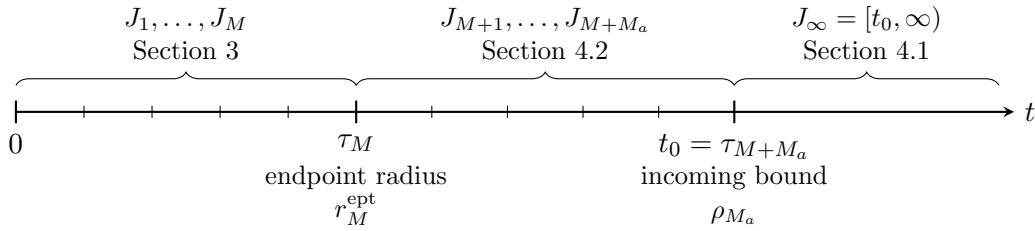
\begin{figure}[h!]
\centering
\begin{tikzpicture}[x=1cm,y=1cm,>=stealth]
 \draw[thick] (0,0) -- (9.5,0);
 \draw[thick,->] (9.5,0) -- (13.2,0) node[right] {$t$};
 \foreach \x in {0.9,1.8,2.7,3.6,5.5,6.5,7.5,8.5}{\draw (\x,0.08) -- (\x,-0.08);}
 \foreach \x/\lab in {0/{0},4.5/{\tau_M},9.5/{t_0=\tau_{M+M_a}}}{
   \draw[thick] (\x,0.15) -- (\x,-0.15) node[below] {$\lab$};}
 \draw[decorate,decoration={brace,amplitude=5pt}] (0,0.35) -- (4.5,0.35)
   node[midway,above=6pt,align=center] {\small $J_1,\dots,J_M$\\[-1pt] \small Section~\ref{sec:multi-stepping}};
 \draw[decorate,decoration={brace,amplitude=5pt}] (4.5,0.35) -- (9.5,0.35)
   node[midway,above=6pt,align=center] {\small $J_{M+1},\dots,J_{M+M_a}$\\[-1pt] \small Section~\ref{sec:aut_chain}};
 \draw[decorate,decoration={brace,amplitude=5pt}] (9.5,0.35) -- (13,0.35)
   node[midway,above=6pt,align=center] {\small $J_\infty=[t_0,\infty)$\\[-1pt] \small Section~\ref{sec:infty_step}};
 \node[align=center] at (4.5,-1.1) {\small endpoint radius\\[-1pt] \small $r_M^{\mathrm{ept}}$};
 \node[align=center] at (9.5,-1.1) {\small incoming bound\\[-1pt] \small $\rho_{M_a}$};
\end{tikzpicture}
\caption{Structure of the global-existence argument. The finite-time validation on $[0,\tau_M]$ is followed by $M_a$ validated intermediate steps, whose endpoint enclosures are propagated until the incoming bound $\rho_{M_a}$ at $t_0$ satisfies the infinite-step test \eqref{eq:radii_infty} on $J_\infty$.}\label{fig:global_timeline}
\end{figure}

We use the finite-tail decomposition and weighted block norms of
Section~\ref{sec:one-step}. As before, boldface $W$ and $V$ denote
scalar upper bounds. Matrix inequalities are componentwise.
The Fourier support $\bN_*$ of $\bu_*$ need not coincide with the
trajectory truncation $\bN$ or the finite linear blocks $\bbm_a$ and
$\bbm$. All modes of $f(\bu_*)$ are included in the residual bound.
When comparing two finite approximations, we extend their coefficients
by zero to the union of their supports.
\subsection{The infinite step}\label{sec:infty_step}

Let $M_a$ be the number of intermediate steps, set $t_0\bydef\tau_{M+M_a}$, and let $J_\infty\bydef[t_0,\infty)$. On $J_\infty$, we take $\bu(t)\equiv\bu_*$ and define
\[
 A_\infty\bydef\cL+\cQ D\cN(\bu_*),\qquad
 G_*(u)\bydef\cN(u)-D\cN(\bu_*)u.
\]
To describe the matching between intervals, write
$u=(u_1,\ldots,u_{M+M_a},u_\infty)$, with $u_\infty\in X_\infty$. Since $f(v)=A_\infty v+\cQ G_*(v)$, the infinite-interval component is defined by
\begin{align}\label{eq:T_infty}
 (T_\infty(u))(t)
 &=e^{A_\infty(t-t_0)}\sigma_{M+M_a}^{(a)}(u)
 +\int_{t_0}^t e^{A_\infty(t-s)}\cQ G_*(u_\infty(s))\,ds,
 \qquad t\in J_\infty.
\end{align}
The endpoint map $\sigma_{M+M_a}^{(a)}$ is defined in
\eqref{eq:sigma_a_def}; for $M_a=0$, it is $\sigma_M$.
For the infinite-step validation, the endpoint of the preceding validated solution is fixed.
% In the sequential validation, the preceding solution has already been validated, and its endpoint is held fixed when applying the contraction argument on $J_\infty$.
We seek a fixed point near the constant function $\bu_*$ in
\[
 X_\infty\bydef\ell^1(C_b^0(J_\infty)),\qquad
 \|u\|_{X_\infty}=\sum_{\bk\ge0}\alpha_{\bk}\nu^{\bk}
 \sup_{t\in J_\infty}|u_{\bk}(t)|.
\]
Here $C_b^0(J_\infty)$ denotes the bounded continuous functions,
and $X_\infty$ consists of coefficient sequences with finite displayed
norm. It is a Banach space, with the same weighted modewise supremum
norm as on the finite intervals.

As in \eqref{eq:W_operator} and \eqref{eq:V_operator}, define
\[
 W_\infty\phi=e^{A_\infty(\cdot-t_0)}\phi,\qquad
 (V_{\star,\infty}f)(t)=\int_{t_0}^t
 e^{A_\infty(t-s)}\cQ_\star f(s)\,ds,
 \quad\star\in\{0,q\},
\]
where $\cQ_0=\mathrm{Id}$ and $\cQ_q=\cQ$, as in
\eqref{eq:linear_one_step}. Their norms will be bounded by $\bm{W_\infty}$ and
$\bm{V_{\star,\infty}}$, respectively.

\paragraph{Linear bounds.}
We apply the estimates of Section~\ref{sec:one-step} to $A_\infty$.
Choose $\bbm$ so that $\mu_{\bbm}<0$, and a numerical diagonalization
$\bA=V\Lambda V^{-1}$ of $A_{\bbm}$ with $\Re\eta_i<0$.
The invertibility of $V$ is checked rigorously. Integrating the
exponentials over $[0,\infty)$ gives
\[
 \Theta_\infty=\operatorname{diag}\!\left(-\frac1{\Re\eta_i}\right),
 \qquad V_{\infty,\star}^{(J_\infty,t_0)}
 =\sup_{\bk\notin\FK}\frac{(\kk\bL)^{2\star}}{|\mu_{\bk}|},
 \quad\star\in\{0,q\}.
\]
The diagonalization error and the finite-tail coupling are bounded
as in \eqref{eq:Neumann-correction}-\eqref{eq:onestep_resolved}.
If $\|K_\infty\|_\infty<1$ and $s_\infty>0$, with these quantities
defined in Appendix~\ref{app:infinite_linear_bounds}, the resulting constants satisfy
\[
 \|W_\infty\|\le\bm{W_\infty},\qquad
 \|V_{\star,\infty}\|\le\bm{V_{\star,\infty}}.
\]

\paragraph{The terminal radii polynomial.}
Fix the incoming endpoint $u_{\mathrm{in}}$ and denote the resulting
local map by $T_\infty^{u_{\mathrm{in}}}$, obtained by replacing
$\sigma_{M+M_a}^{(a)}(u)$ in \eqref{eq:T_infty} by $u_{\mathrm{in}}$.
Let $\|u_{\mathrm{in}}-\bu_*\|_\nu\le\rho_{M_a}$, where
$u_{\mathrm{in}}$ is the endpoint of the previously validated solution.
Using
$A_\infty\bu_*+\cQ G_*(\bu_*)=f(\bu_*)$ and the constant-function identity
\[
 \bu_*=e^{A_\infty(t-t_0)}\bu_*
 -\int_{t_0}^t e^{A_\infty(t-s)}A_\infty\bu_*\,ds,
\]
we obtain
\begin{align}\label{eq:Y_global_derivation}
 \bu_*-(T_\infty^{u_{\mathrm{in}}}(\bu_*))(t)
 &=e^{A_\infty(t-t_0)}(\bu_*-u_{\mathrm{in}})
 -\int_{t_0}^t e^{A_\infty(t-s)}f(\bu_*)\,ds.
\end{align}
Taking the $X_\infty$-norm gives the defect bound
\begin{align}\label{eq:Y_global}
 Y_\infty^{(M_a)}(\rho_{M_a})
 \bydef\bm{W_\infty}\rho_{M_a}+\bm{V_{0,\infty}}\|f(\bu_*)\|_\nu.
\end{align}
We use $\bm{V_{0,\infty}}$ because $f(\bu_*)$ already contains
the multiplier $\cQ$.

For the derivative, \eqref{eq:Zi_polynomial} gives
\[
 z_*(r)\bydef\sum_{p=2}^P p|c_p|
 \left((\|\bu_*\|_\nu+r)^{p-1}-\|\bu_*\|_\nu^{p-1}\right).
\]
Expansion about $\bu_*$ and the Banach-algebra inequality give
$\|DG_*(u)\|_{B(X_\infty)}\le z_*(r)$ on
$B_r(\bu_*)$. 
% Since $DG_*(\bu_*)=0$, differentiation of \eqref{eq:T_infty} gives
With the incoming endpoint held fixed, differentiation of
$T_\infty^{u_{\mathrm{in}}}$ with respect to its trajectory
argument gives

\begin{align}\label{eq:Z_global}
 Z_\infty(r)\bydef\bm{V_{q,\infty}}z_*(r).
\end{align}
Here $\bm{V_{q,\infty}}$ includes the factor $\cQ$, and
$Z_\infty(0)=0$.

The preceding integral bounds give continuous Fr\'echet
differentiability on a ball in $X_\infty$. By the one-component
case of the Radii Polynomial Theorem~\ref{thm:RadiiPolynomial},
it suffices to find $r_\infty>0$ such that
\begin{align}\label{eq:radii_infty}
 Y_\infty^{(M_a)}+\int_0^{r_\infty}Z_\infty(s)\,ds\le r_\infty,
 \qquad Z_\infty(r_\infty)<1.
\end{align}

For $\cN(u)=\pm u^3$, \eqref{eq:radii_infty} reads
\[
 Y_\infty^{(M_a)}+\bm{V_{q,\infty}}
 (3\|\bu_*\|_\nu r_\infty^2+r_\infty^3)\le r_\infty,
 \qquad 3\bm{V_{q,\infty}}r_\infty
 (2\|\bu_*\|_\nu+r_\infty)<1.
\]
The residual term in \eqref{eq:Y_global} can be sharpened by the
mode-resolved estimate \eqref{eq:residual_volterra_resolved} in
Appendix~\ref{app:global_residual_bounds}.

At the end of the finite-time validation, the distance to $\bu_*$
is bounded by
\[
 \rho_0\bydef r_M+\|\bu_*-\bu_M(\tau_M)\|_\nu.
\]
If this bound is too large to satisfy \eqref{eq:radii_infty}, we
continue the solution over additional finite intervals. Estimates for
choosing the number of steps are given in Appendix~\ref{app:global_predictors}.
The final radius is verified using the full nonlinear inequalities.

\subsection{Validated intermediate steps}\label{sec:aut_chain}
We now describe the additional time steps. Let $\bu_{M+k}$ be the
approximation on the $k$\,th added interval. Although the PDE is
autonomous, its linearization along $\bu_{M+k}$ generally depends
on time. We use the superscript $(a)$ for these steps. When the reference is constant, $\bu_{M+k}\equiv\bu_*$, the linearization is autonomous and we speak of autonomous steps; their linear bounds can then be computed once and reused (see Appendix~\ref{app:intermediate_endpoint_bounds}).
Choose positive lengths $h_{a,k}$ and set
\begin{align}\label{eq:aut_grid}
 \tau_{M+k}\bydef\tau_M+\sum_{\ell=1}^kh_{a,\ell},
 \qquad J_{M+k}\bydef[\tau_{M+k-1},\tau_{M+k}].
\end{align}
On each added interval use $X_{M+k}=\ell^1(C^0(J_{M+k}))$,
with the same weighted modewise supremum norm as in Section~\ref{sec:trajectory_validation}.
A fixed length $h_a$ is the special case $h_{a,k}=h_a$.
The approximation on each interval is a finite Chebyshev-Fourier
polynomial. Its full residual, including all Fourier modes created by the
nonlinearity, is evaluated in interval arithmetic.

As in Section~\ref{sec:multi-stepping}, define
\begin{align}\label{eq:G_a_def}
 G_{M+k}^{(a)}(v)(t)&\bydef\cN(v(t))-D\cN(\bu_{M+k}(t))v(t),
 &DG_{M+k}^{(a)}(\bu_{M+k})&=0.
\end{align}
With $U_{M+k}(t,s)$ generated by
$\cL+\cQ D\cN(\bu_{M+k}(t))$, the added maps are
\begin{align}\label{eq:T_aut_def}
 (T_{M+k}^{(a)}(u))(t)
 &\bydef U_{M+k}(t,\tau_{M+k-1})\sigma_{M+k-1}^{(a)}(u)
 +\int_{\tau_{M+k-1}}^tU_{M+k}(t,s)\cQ G_{M+k}^{(a)}(u_{M+k})(s)\,ds,\\
 \sigma_{M+k}^{(a)}(u)&\bydef(T_{M+k}^{(a)}(u))(\tau_{M+k}),
 \qquad\sigma_M^{(a)}\bydef\sigma_M.\label{eq:sigma_a_def}
\end{align}
For a fixed incoming endpoint $v_{\mathrm{in}}$, let
$T_{M+k}^{(a),v_{\mathrm{in}}}:X_{M+k}\to X_{M+k}$ denote
the local map obtained by replacing the preceding endpoint in
\eqref{eq:T_aut_def} by $v_{\mathrm{in}}$.
On each step, this value is the endpoint of the preceding
validated solution. The bounds are required to hold for every initial
value in its enclosure.

\paragraph{Local validation and endpoint matching.}
Let $\bu_j^{\mathrm{ept}}$ denote the stored numerical endpoint on
step $j$, and let $r_j^{\mathrm{ept}}$ enclose the true endpoint error
about it. Initialize the endpoint radius by
\[
 r_M^{\mathrm{ept}}\ge r_M+
 \|\bu_M(\tau_M)-\bu_M^{\mathrm{ept}}\|_\nu.
\]
An incoming enclosure $\mathcal E_{j-1}$ is any verified set containing
the preceding endpoint and satisfying
$\|v-\bu_{j-1}^{\mathrm{ept}}\|_\nu\le r_{j-1}^{\mathrm{ept}}$
for $v\in\mathcal E_{j-1}$. It may be the scalar ball or the affine enclosure from the appendix
intersected with this ball. The intersection is needed when the scalar
radius is smaller than the bound for the full affine set. The local
defect bound must hold uniformly on the chosen set. Choose a fixed intermediate block
$\bbm_a$ with $\mu_{\bbm_a}<0$.
Section~\ref{sec:one-step} supplies bounds for the propagation of
initial data and for the integral terms on each added interval, under
the same invertibility and positivity conditions. Let
$Y_{M+k}^{(a)}$ bound the local defect uniformly over the incoming
endpoint enclosure. As in \eqref{eq:Ym_bound}, it includes the incoming
error and mismatch and the full differential residual, propagated by
the integral operator without an additional factor $\cQ$.
For $\cN(u)=\pm u^3$, put $a_k\bydef\|\bu_{M+k}\|_{X_{M+k}}$.
The Radii Polynomial Theorem gives the local derivative bound and radii conditions
\begin{align}\label{eq:Z_aut}
 Z_{M+k}^{(a)}(r)&\bydef3\bm{V_{q,a,k}}r(2a_k+r),\\
 Y_{M+k}^{(a)}+\bm{V_{q,a,k}}(3a_kr_{M+k}^2+r_{M+k}^3)
 &\le r_{M+k},\qquad Z_{M+k}^{(a)}(r_{M+k})<1.
 \label{eq:radii_aut}
\end{align}
Both inequalities are verified in interval arithmetic.

For $j=M+k$, set
\[
 d_j\bydef\|\bu_j(\tau_{j-1})-\bu_{j-1}^{\mathrm{ept}}\|_\nu,
 \qquad e_j\bydef\|\bu_j(\tau_j)-\bu_j^{\mathrm{ept}}\|_\nu,
 \qquad F_j\bydef\bu_j'-f(\bu_j).
\]
Let $\bm{W_{a,k}^{\mathrm{ept}}}$ bound
$\|U_j(\tau_j,\tau_{j-1})\|_{B(\ell_\nu^1)}$, and let
$\bm{V_{0,a,k}}$ and $\bm{V_{q,a,k}}$ bound the integral operators
on $J_j$ without and with the factor $\cQ$, respectively.
These constants are obtained from the one-step estimates of
Section~\ref{sec:one-step}; their endpoint refinement is described
in Appendix~\ref{app:intermediate_endpoint_bounds}.

\begin{proposition}[{\bf Intermediate endpoint enclosure}]\label{prop:intermediate_enclosure}
Suppose $\cN(u)=\pm u^3$ and the local test \eqref{eq:radii_aut}
holds uniformly for initial values in $\mathcal E_{j-1}$, whose
distance to $\bu_{j-1}^{\mathrm{ept}}$ is bounded by
$r_{j-1}^{\mathrm{ept}}$.
Then the validated solution satisfies
\[
 \|\tilde u(\tau_j)-\bu_j^{\mathrm{ept}}\|_\nu
 \le r_j^{\mathrm{ept}},
\]
where a computable endpoint radius is
\begin{align}\label{eq:intermediate_endpoint_simple}
 r_j^{\mathrm{ept}}\bydef\min\Bigl\{
 &r_j+e_j,~
 \bm{W_{a,k}^{\mathrm{ept}}}(r_{j-1}^{\mathrm{ept}}+d_j)
 +\bm{V_{0,a,k}}\|F_j\|_{X_j}
 +\bm{V_{q,a,k}}(3a_k+r_j)r_j^2+e_j\Bigr\}.
\end{align}
\end{proposition}
\begin{proof}
The local ball gives $\|\tilde u(\tau_j)-\bu_j(\tau_j)\|_\nu\le r_j$,
which proves the first bound. For the second, set $w=\tilde u-\bu_j$.
Variation of constants for the error equation gives
\[
 w(\tau_j)=U_j(\tau_j,\tau_{j-1})w(\tau_{j-1})
 +\int_{\tau_{j-1}}^{\tau_j}U_j(\tau_j,s)
 \bigl(\cQ[\cN(\bu_j+w)-\cN(\bu_j)-D\cN(\bu_j)w]-F_j\bigr)(s)\,ds.
\]
The incoming error is at most $r_{j-1}^{\mathrm{ept}}+d_j$.
The Banach-algebra inequality bounds the cubic remainder by
$(3a_k+r_j)r_j^2$. Applying the stated linear bounds and adding
$e_j$ proves the second estimate.
\end{proof}

As in Section~\ref{sec:multistep}, taking scalar norms at each
step can lead to large bounds. Appendix~\ref{app:intermediate_endpoint_bounds}
gives the alternative radius \eqref{eq:intermediate_endpoint_radius},
which retains the signed finite propagator products. Either bound
can be used to enclose the endpoint.

The distance from this endpoint to $\bu_*$ is therefore bounded by
\begin{align}\label{eq:intermediate_matching}
 \rho_k\bydef r_{M+k}^{\mathrm{ept}}
             +\|\bu_*-\bu_{M+k}^{\mathrm{ept}}\|_\nu.
\end{align}
After each added step, we check \eqref{eq:radii_infty} with
$\rho_{M_a}=\rho_k$. We stop when these inequalities hold.

\subsection{Global enclosure}\label{sec:refined_cor}
We combine the finite-time solution, the added steps, and the solution
on $J_\infty$. For the first $M$ steps, assume the conditions of
the Radii Polynomial Theorem~\ref{thm:RadiiPolynomial}:
\begin{align}\label{eq:radii_finite}
 Y_m+\sum_{i=1}^m\int_0^{r_i}Z_m^i(s)\,ds\le r_m,
 \qquad\sum_{i=1}^mZ_m^i(r_i)\zeta_i<\zeta_m,
 \quad r_m>0,\quad\zeta_m>0.
\end{align}

\begin{corollary}[{\bf Global existence via validated intermediate steps}]\label{cor:global_existence}
 Assume $\cN(u)=\pm u^3$, verified one-step linear bounds, and the finite-time
 conditions~\eqref{eq:radii_finite}. Suppose every added step satisfies
 the uniform local conditions~\eqref{eq:radii_aut} for its previously
 verified incoming enclosure, and its endpoint enclosure is propagated
 as in Proposition~\ref{prop:intermediate_enclosure} and
 \eqref{eq:intermediate_matching}.
 Suppose the infinite-interval map satisfies~\eqref{eq:radii_infty} with the verified
 incoming bound $\rho_{M_a}$ and a radius $r_\infty>0$.
 Then the validated finite-time trajectory, the intermediate segments, and the infinite-interval solution join to a solution $\tilde u$ of~\eqref{eq:PDEs} for all $t\ge0$, with
\[
 \|\tilde u|_{J_m}-\bu_m\|_{X_m}\le r_m,
 \qquad m=1,\dots,M+M_a,
 \qquad
 \|\tilde u|_{J_\infty}-\bu_*\|_{X_\infty}\le r_\infty.
\]
 The solution is unique among trajectories satisfying these validation bounds. In particular,
\begin{align}\label{eq:global_enclosure}
 \|\tilde u(t)-\bu_*\|_\nu\le r_\infty,
 \qquad t\ge\tau_{M+M_a}.
\end{align}
% The solution is unique among trajectories satisfying these
% local validation bounds and
% $\|\tilde u|_{J_\infty}-\bu_*\|_{X_\infty}\le r_\infty$.
%  Then the combined map has a unique fixed point in
% \[
%  \prod_{m=1}^{M+M_a}B_{r_m}(\bu_m)\times B_{r_\infty}(\bu_*).
% \]
%  The components join to a solution of~\eqref{eq:PDEs} for all $t\ge0$, and
% \begin{align}\label{eq:global_enclosure}
%  \|\tilde u(t)-\bu_*\|_\nu\le r_\infty,
%  \qquad t\ge\tau_{M+M_a}.
% \end{align}

\end{corollary}
\begin{proof}
The Radii Polynomial Theorem gives a unique solution on $[0,\tau_M]$
and an enclosure of its endpoint. If $M_a=0$, this endpoint is used
directly on $J_\infty$. Otherwise, use it as the initial value
on $J_{M+1}$. The uniform bounds in \eqref{eq:radii_aut} give a unique
fixed point on this interval, and Proposition~\ref{prop:intermediate_enclosure}
(or its refinement in the appendix) encloses its endpoint. Repeating
this argument gives the solution through $\tau_{M+M_a}$. The bound
\eqref{eq:intermediate_matching} and the inequalities
\eqref{eq:radii_infty} then give a unique fixed point on $J_\infty$.

At each interface, the initial value of the next segment is the endpoint of the preceding validated solution. The segments therefore join continuously. Uniqueness on the finite-time product space, followed by uniqueness for each fixed incoming endpoint in the intermediate and infinite steps, gives uniqueness in the stated class. The last segment belongs to $B_{r_\infty}(\bu_*)$, which proves \eqref{eq:global_enclosure}.
% 
% The endpoint identities in \eqref{eq:sigma_a_def} show that the
% components join continuously. Uniqueness on each successive interval
% gives uniqueness for the combined map. The last component belongs to
% $B_{r_\infty}(\bu_*)$, which proves \eqref{eq:global_enclosure}.
\end{proof}

An alternative validation based on a direct application of the Radii Polynomial Theorem to the entire product map is described in Appendix~\ref{app:global_product_argument}. The argument above does not require verification of those coupled product-space inequalities.

% A direct application of the Radii Polynomial Theorem to the entire
% product map is given in Appendix~\ref{app:global_product_argument}.

\section{Boundary value problems projected onto linear spaces} \label{sec:BVP}

We now consider a boundary value problem in which both endpoint values are unknown. The initial value is constrained to a finite-dimensional approximation of the unstable space of the trivial equilibrium, and the terminal value is constrained to an affine space centered at a numerical approximation $\hat u\in\ell_\nu^1$ of a nontrivial saddle, with directions approximating its stable space. We do not assume that $f(\hat u)=0$, where $f$ is defined in \eqref{eq:ODE-sequence-space}. We fix the time of flight $\tau>0$ and impose one scalar condition on the initial value. The resulting problem combines the IVP fixed-point map of Section~\ref{sec:fixed-point-operator} with equations for the endpoint parameters. Throughout this section, the trivial equilibrium has $n_u$ unstable directions and the target saddle has numerical Morse index $n_u-1$, as in Section~\ref{sec:OK_application}; the dimension count below relies on this assumption.

As in the IVP formulation, the derivative of the trajectory map vanishes at the numerical approximation. The linearized BVP therefore reduces to a system for the endpoint parameters. We use this system to bound the inverse, and then derive the bounds $Y$ and $Z$ for the contraction argument. For clarity, we present the construction with the single-step map \eqref{eq:T_single}; the multi-step version used in the computations is described in Remark~\ref{rem:BVP_multistep}.

\subsection{Endpoint conditions and the zero-finding problem}

Let $n_u$ be the number of unstable directions at the trivial equilibrium. We choose finite vectors $\bv_1^u,\dots,\bv_{n_u}^u$ representing these directions and extend them by zero outside $\FN$. Let $\varPi^{(\bN)}$ denote the projection onto the modes in $\FN$, that is, $(\varPi^{(\bN)}a)_{\bk}=a_{\bk}$ for $\bk\in\FN$ and $(\varPi^{(\bN)}a)_{\bk}=0$ otherwise. Writing the extensions as $v_i^u$, we define
\begin{align}
    V^u\bydef(v_1^u\;\cdots\;v_{n_u}^u)
    =\begin{pmatrix}
        \bv_1^u & \cdots & \bv_{n_u}^u\\
        0 & \cdots & 0\\[-1mm]
        \vdots & & \vdots
    \end{pmatrix}.
\end{align}
Since $\cN(0)=0$ and $D\cN(0)=0$, the linearization at the origin is $\cL$. Its unstable eigenvalues are precisely the positive multipliers $\mu_{\bk}$ introduced in Section~\ref{sec:functional_analytic_setup}. Relabeling them as $\lambda_i^u\bydef\mu_{\bk_i}>0$, the chosen vectors satisfy
\[
    \varPi^{(\bN)}\cL\varPi^{(\bN)}\bv_i^u\approx\lambda_i^u\bv_i^u,
    \qquad i=1,\dots,n_u.
\]
Because $\cL$ is diagonal, these directions can also be chosen as the corresponding canonical Fourier vectors. The left endpoint condition is
\[
    u(0)=P(\theta)\bydef V^u\theta,
    \qquad \theta\in\R^{n_u}.
\]

For the right endpoint, let $n$ denote the number of retained finite Fourier coordinates. We choose $n-n_u+1$ finite stable directions, as suggested by the numerical Morse index $n_u-1$ of the target saddle. Without a mean constraint, $n=N_1\cdots N_d=|\FN|$. If the zero Fourier mode is fixed, we remove it from the state variables and use $n=|\FN|-1$; the sequence spaces and projections in this section are then understood on the remaining coordinates. Throughout this section, the tail projection is $\varPi^{(\infty)}\bydef I-\varPi^{(\bN)}$, with the identity taken on the corresponding state space. We choose $n-n_u+1$ finite vectors $\bv_j^s$ approximating stable eigenvectors of the linearization at $\hat u$:
\[
    \varPi^{(\bN)}\bigl(\cL+\cQ D\cN(\hat u)\bigr)\varPi^{(\bN)}\bv_j^s
    \approx\lambda_j^s\bv_j^s,
    \qquad \lambda_j^s<0,
    \quad j=1,\dots,n-n_u+1.
\]
Fix an ordering of the Fourier modes in which the $n$ retained finite modes come first, and denote the canonical vectors in this ordering by $e_j$. The right endpoint parameters consist of $n-n_u+1$ finite coefficients and all tail coefficients. To avoid redundant finite coordinates, we use the closed parameter subspace
\[
    X^s\bydef\left\{\phi\in\ell_\nu^1:
    \phi_j=0\ \text{for }n-n_u+2\le j\le n\right\},
\]
with the norm inherited from $\ell_\nu^1$. If $n_u=1$, no coordinates are removed. Extending each $\bv_j^s$ by zero, we define the bounded map $V^s:X^s\to\ell_\nu^1$ by
\begin{align}
    V^s\phi\bydef\sum_{j=1}^{n-n_u+1}v_j^s\phi_j+\varPi^{(\infty)}\phi,
    \qquad
    Q(\phi)\bydef\hat u+V^s\phi.
\end{align}
Thus the right endpoint condition is $u(\tau)=Q(\phi)$ for $\phi\in X^s$. In matrix form, with the unused finite parameter coordinates omitted,
\[
    V^s=(v_1^s\;\cdots\;v_{n-n_u+1}^s\;e_{n+1}\;e_{n+2}\;\cdots).
\]
The terminal tail coefficients are free, while the retained finite coefficients satisfy the prescribed stable-space condition.

\begin{remark}
    The maps $P$ and $Q$ are fixed linear and affine maps, respectively. We write
    $\mathbb{E}^u(0)\bydef\mathrm{Range}(V^u)$ and
    $\mathbb{E}^s(\hat u)\bydef\mathrm{Range}(V^s)$ for the spaces used in the endpoint conditions. The center $\hat u$ and the vectors are computed numerically. No separate validation of the saddle or its stable space is needed for the projected BVP: the center and the endpoint maps are fixed, and their matching residual is included in $Y$. The theorem below proves the existence of a solution segment satisfying these projected endpoint conditions.
\end{remark}

We use the trajectory space $X=\ell^1(C^0([0,\tau]))$ defined in Section~\ref{sec:function-space}, with its norm $\|\cdot\|_X$. In particular, endpoint evaluation satisfies $\|u(\tau)\|_\nu\le\|u\|_X$. Define
\begin{align}\label{eq:X_BVP}
    X_{\mathtt{BVP}}\bydef\R^{n_u}\times X^s\times X,
    \qquad
    \| (\theta,\phi,u)\|_{X_{\mathtt{BVP}}}
    \bydef\max\{\|\theta\|_1,\|\phi\|_\nu,\|u\|_X\}.
\end{align}
For $x=(\theta,\phi,u)\in X_{\mathtt{BVP}}$, we write
\[
    \llbracket x\rrbracket\bydef
    \begin{pmatrix}\|\theta\|_1\\\|\phi\|_\nu\\\|u\|_X\end{pmatrix}\in\R_+^3
\]
for the vector of its component norms, and $\llbracket x\rrbracket_i$, $i\in\{\theta,\phi,u\}$, for its components. Inequalities between such vectors are understood componentwise. The endpoint parameters and the trajectory are all unknown. We seek a zero of
$F:\R^{n_u}\times X^s\times X\to\R\times\ell_\nu^1\times X$ given by
\begin{equation}\label{eq:zero_finding_BVP}
    F(\theta,\phi,u)\bydef
    \begin{pmatrix}
        l(\theta)-\delta^2\\
        u(\tau)-Q(\phi)\\
        u-T(u,P(\theta))
    \end{pmatrix},
\end{equation}
where $T(u,u_0)$ denotes the map \eqref{eq:T_single} with its initial condition displayed as a second argument. The first component imposes the amplitude condition
\begin{align}\label{eq:phase_condition}
    l(\theta)\bydef\|\theta\|_2^2=\theta^\top\theta.
\end{align}
Thus $l(\theta)=\delta^2$ fixes the initial parameter on a sphere of radius $\delta>0$. Its derivative is $l'(\theta)=2\theta^\top$. The computation in Section~\ref{sec:OK_application} uses $\delta=0.01$. Together with the $n$ finite endpoint matching equations, this scalar condition gives $n+1$ equations for $n_u+(n-n_u+1)=n+1$ finite endpoint parameters. The number of equations thus agrees with the number of finite parameters. We verify invertibility of the linearized system below.

Let $\bar x\bydef(\bar\theta,\bar\phi,\bu)$ be a numerical approximation satisfying $F(\bar x)\approx0$. For a radius vector $r=(r^\theta,r^\phi,r^u)>0$, define the closed product ball
\[
    B_r(\bar x)\bydef\left\{(\theta,\phi,u)\in X_{\mathtt{BVP}}:
    \|\theta-\bar\theta\|_1\le r^\theta,\ 
    \|\phi-\bar\phi\|_\nu\le r^\phi,\ 
    \|u-\bu\|_X\le r^u\right\}.
\]

\subsection{Reduction of the linearized boundary value problem}

The evolution operator $U(t,s)$ in $T$ is generated by the linearization along $\bu$. Consequently $D_uT(\bu,P(\bar\theta))=0$, since $DG(\bu)=0$ in \eqref{eq:G_def}. If $E_\tau u\bydef u(\tau)$ denotes endpoint evaluation, then
\[
    DF(\bar x)=
    \begin{pmatrix}
        l'(\bar\theta)&0&0\\
        0&-D_\phi Q(\bar\phi)&E_\tau\\
        -U(\cdot,0)D_\theta P(\bar\theta)&0&I
    \end{pmatrix}.
\]
To bound its inverse, let $\mathscr G=(g^\theta,g^\phi,g^u)^\top\in\R\times\ell_\nu^1\times X$ and consider $DF(\bar x)h=\mathscr G$, with $h=(h^\theta,h^\phi,h^u)\in X_{\mathtt{BVP}}$. The component equations are
\begin{align}\label{eq:LinearizedProb_BVP}
    \begin{cases}
        l'(\bar\theta)h^\theta=g^\theta,\\
        -D_\phi Q(\bar\phi)h^\phi+h^u(\tau)=g^\phi,\\
        -U(t,0)D_\theta P(\bar\theta)h^\theta+h^u(t)=g^u(t).
    \end{cases}
\end{align}
The last equation gives
\[
    h^u(t)=U(t,0)V^u h^\theta+g^u(t).
\]
Substituting at $t=\tau$ in the second equation gives
\[
    B^\dagger\begin{pmatrix}h^\theta\\h^\phi\end{pmatrix}
    =\begin{pmatrix}g^\theta\\g^\phi-g^u(\tau)\end{pmatrix},
    \qquad
    B^\dagger\bydef
    \begin{pmatrix}
        l'(\bar\theta)&0\\
        U(\tau,0)V^u&-V^s
    \end{pmatrix}.
\]
Here $B^\dagger:\R^{n_u}\times X^s\to\R\times\ell_\nu^1$ is the reduced endpoint operator. The superscript $\dagger$ is part of its notation and does not denote an adjoint.

We construct its approximate inverse from the finite Fourier projection. Define
\begin{align}\label{eq:Bbardag}
    \bar B^\dagger\bydef
    \begin{pmatrix}
        l'(\bar\theta)&0\\
        \varPi^{(\bN)}U(\tau,0)V^u&
        -\varPi^{(\bN)}V^s\varPi^{(\bN)}-\varPi^{(\infty)}
    \end{pmatrix}.
\end{align}
The projections on the parameter factor are restricted to $X^s$. Since
$V^s=\varPi^{(\bN)}V^s\varPi^{(\bN)}+\varPi^{(\infty)}$, we have
\[
    B^\dagger-\bar B^\dagger
    =\begin{pmatrix}
        0&0\\
        \varPi^{(\infty)}U(\tau,0)V^u&0
    \end{pmatrix}.
\]
The approximation therefore omits only the tail of $U(\tau,0)V^u$.

On the finite block, we compute a numerical inverse of
\begin{align}\label{eq:Bdag_finite}
    \bar B^\dagger_{\mathrm{fin}}\bydef
    \begin{pmatrix}
        l'(\bar\theta)&0\\
        \varPi^{(\bN)}U(\tau,0)V^u&-\varPi^{(\bN)}V^s\varPi^{(\bN)}
    \end{pmatrix},
\end{align}
where the zero parameter coordinates are omitted. This matrix has size $(n+1)\times(n+1)$. We extend the resulting finite matrix by minus the identity on the tail to obtain
$B:\R\times\ell_\nu^1\to\R^{n_u}\times X^s$. We bound the residual of this numerical inverse using interval arithmetic.

Let $G\bydef I-BB^\dagger$ be the residual of this approximate inverse. For a block operator $A$, we denote by $\mathfrak b(A)$ the nonnegative matrix of its block operator norms, as in \eqref{eq:block_norm_matrix}. In the present case,
\[
    \mathfrak b(G)\bydef
    \begin{pmatrix}
        \|G_{11}\|_{B(\R^{n_u})}&\|G_{12}\|_{B(X^s,\R^{n_u})}\\
        \|G_{21}\|_{B(\R^{n_u},X^s)}&\|G_{22}\|_{B(X^s)}
    \end{pmatrix},
    \qquad
    \mathfrak b(B)\bydef
    \begin{pmatrix}
        \|B_{11}\|_1&\|B_{12}\|_{B(\ell_\nu^1,\R^{n_u})}\\
        \|B_{21}\|_\nu&\|B_{22}\|_{B(\ell_\nu^1,X^s)}
    \end{pmatrix}.
\]
All norms on $\R^{n_u}$ are induced by $\|\cdot\|_1$. For the reduced system, we use the sum of the two component norms. The induced operator norm satisfies $\|G\|_{\mathrm{op}}\le\|\mathfrak b(G)\|_1$. We estimate the block bounds for the finite residual and the propagation tail separately:
\begin{align*}
    \bigl\|\mathfrak b(I-B\bar B^\dagger)\bigr\|_1&\le Z_0,\\
    \bigl\|\mathfrak b\bigl(B(\bar B^\dagger-B^\dagger)\bigr)\bigr\|_1&\le Z_1.
\end{align*}
If $Z_0+Z_1<1$, then $\|\mathfrak b(G)\|_1<1$, and the Neumann series makes $BB^\dagger=I-G$ invertible. Thus $B$ is surjective. Since its finite block is square and its tail is minus the identity, $B$ is invertible. It follows that
\[
    (B^\dagger)^{-1}=(I-G)^{-1}B.
\]
The same argument gives the componentwise estimate
\[
    \begin{pmatrix}\|h^\theta\|_1\\\|h^\phi\|_\nu\end{pmatrix}
    \le (I-\mathfrak b(G))^{-1}\mathfrak b(B)
    \begin{pmatrix}|g^\theta|\\\|g^\phi\|_\nu+\|g^u(\tau)\|_\nu\end{pmatrix}.
\]

Write $(I-\mathfrak b(G))^{-1}\mathfrak b(B)=(M_{ab})_{a,b=1,2}$ and define
\begin{align}\label{eq:Wtheta_phi_constants}
    \begin{aligned}
        &W_1^\theta\bydef M_{11},&&W_1^\phi\bydef M_{12},&&W_1^u\bydef M_{12},\\
        &W_2^\theta\bydef M_{21},&&W_2^\phi\bydef M_{22},&&W_2^u\bydef M_{22}.
    \end{aligned}
\end{align}
To bound the trajectory correction, we also define
\begin{align}\label{eq:U_DP}
    U_{DP}\bydef\|U(\cdot,0)D_\theta P(\bar\theta)\|_{B(\R^{n_u},X)}
    =\|U(\cdot,0)V^u\|_{B(\R^{n_u},X)}.
\end{align}
Each column $U(\cdot,0)v_i^u$ solves \eqref{eq:model_linear_problem} with zero forcing and initial condition $v_i^u$. We bound these columns using $\bm{W^{(J_m,\tau_j)}}$ in \eqref{eq:WVbounds}. For the single-interval norm on $X$, the bounds over the individual steps are summed, since $\|w\|_X\le\sum_m\|w|_{J_m}\|_{X_m}$. In the product space of Remark~\ref{rem:BVP_multistep}, the maximum over the steps is taken instead. Set
\begin{align}\label{eq:Wu_constants}
    W_3^\bullet\bydef U_{DP}W_1^\bullet,
    \qquad \bullet\in\{\theta,\phi,u\}.
\end{align}
The last equation of \eqref{eq:LinearizedProb_BVP} then gives the inverse bound
\begin{align*}
    \llbracket DF(\bar x)^{-1}\mathscr G\rrbracket
    &=\begin{pmatrix}\|h^\theta\|_1\\\|h^\phi\|_\nu\\\|h^u\|_X\end{pmatrix}\\
    &\le
    \begin{pmatrix}
        W_1^\theta|g^\theta|+W_1^\phi\|g^\phi\|_\nu+W_1^u\|g^u(\tau)\|_\nu\\
        W_2^\theta|g^\theta|+W_2^\phi\|g^\phi\|_\nu+W_2^u\|g^u(\tau)\|_\nu\\
        W_3^\theta|g^\theta|+W_3^\phi\|g^\phi\|_\nu+W_3^u\|g^u(\tau)\|_\nu+\|g^u\|_X
    \end{pmatrix}.
\end{align*}
This bounds the inverse of $DF(\bar x)$ in terms of the reduced inverse and the evolution operator.

We use the resulting exact inverse to define the Newton-like map
\begin{align}\label{eq:T_BVP}
    T_{\mathtt{BVP}}(x)\bydef x-DF(\bar x)^{-1}F(x).
\end{align}
Its fixed points are the zeros of $F$. The bounds above suffice to estimate its defect and derivative without computing the exact inverse.

\subsection{\texorpdfstring{The bound \boldmath $Y$\boldmath}{The bound Y}}\label{sec:BVP_Y}

Write the three components of $F(\bar x)$ as
\begin{align*}
    g^\theta&\bydef l(\bar\theta)-\delta^2,\\
    g^\phi&\bydef\bu(\tau)-Q(\bar\phi),\\
    g^u&\bydef\bu-T(\bu,P(\bar\theta)).
\end{align*}
We bound $g^u$ as in Section~\ref{sec:multi-stepping}, with initial condition $P(\bar\theta)$. A separate estimate of $g^u(\tau)$ gives a sharper endpoint bound. Since
$T_{\mathtt{BVP}}(\bar x)-\bar x=-DF(\bar x)^{-1}F(\bar x)$, the inverse bound yields
\[
    \llbracket T_{\mathtt{BVP}}(\bar x)-\bar x\rrbracket\le
    \begin{pmatrix}
        W_1^\theta|g^\theta|+W_1^\phi\|g^\phi\|_\nu+W_1^u\|g^u(\tau)\|_\nu\\
        W_2^\theta|g^\theta|+W_2^\phi\|g^\phi\|_\nu+W_2^u\|g^u(\tau)\|_\nu\\
        W_3^\theta|g^\theta|+W_3^\phi\|g^\phi\|_\nu+W_3^u\|g^u(\tau)\|_\nu+\|g^u\|_X
    \end{pmatrix}
    \bydef Y.
\]

\subsection{\texorpdfstring{The bound \boldmath $Z$\boldmath}{The bound Z}}\label{sec:BVP_Z}

For $x\in B_r(\bar x)$ and $\|h\|_{X_{\mathtt{BVP}}}\le1$, differentiating \eqref{eq:T_BVP} gives
\begin{align*}
    DT_{\mathtt{BVP}}(x)h
    &=DF(\bar x)^{-1}\bigl(DF(\bar x)-DF(x)\bigr)h\\
    &=DF(\bar x)^{-1}
    \begin{pmatrix}
        2(\bar\theta-\theta)^\top h^\theta\\
        0\\
        \bigl(D_uT(u,P(\theta))-D_uT(\bu,P(\bar\theta))\bigr)h^u
    \end{pmatrix}.
\end{align*}
Since $Q$ is affine and $DP(\theta)=V^u$ is constant, only the amplitude condition and the trajectory map contribute to this difference. For the amplitude condition,
\[
    \bigl|2(\bar\theta-\theta)^\top h^\theta\bigr|
    \le2r^\theta.
\]

Let $\tilde Z$ be a non-decreasing bound for the derivative of the trajectory map:
\begin{align}\label{eq:Ztilde_BVP}
    \|D_uT(u,P(\theta))-D_uT(\bu,P(\bar\theta))\|_{B(X)}
    \le\tilde Z(r^u),
    \qquad (\theta,\phi,u)\in B_r(\bar x).
\end{align}
The estimates of Section~\ref{sec:multi-stepping} provide such a bound. Endpoint evaluation has norm at most one, so the same bound applies at $t=\tau$. The inverse estimate then gives
\begin{align}\label{eq:Z_BVP}
    \llbracket DT_{\mathtt{BVP}}(x)h\rrbracket
    \le
    \begin{pmatrix}
        2W_1^\theta r^\theta+W_1^u\tilde Z(r^u)\\
        2W_2^\theta r^\theta+W_2^u\tilde Z(r^u)\\
        2W_3^\theta r^\theta+(W_3^u+1)\tilde Z(r^u)
    \end{pmatrix}
    \bydef Z(r).
\end{align}
The terminal parametrization is affine, so $Z(r)$ is independent of $r^\phi$. The terms in $r^\theta$ come from the amplitude condition.

In the multi-step setting, equip $\bm X=\prod_{m=1}^M X_m$ with the maximum of the step norms and use the same trajectory radius $r^u$ on every step. A bound for the full trajectory derivative is then
\[
    \tilde Z(r^u)\bydef\max_{1\le m\le M}\sum_{i=1}^m Z_m^i(r^u),
\]
where $Z_m^i$ are defined in \eqref{eq:Zmi_bound}. Taking the maximum over the row sums bounds the derivative on the whole product space. These bounds use the same propagation estimates $\bm{V_q^{(J_m,\tau_j)}}$ as in the IVP.

\subsection{Existence theorem for the projected boundary value problem}\label{sec:BVP_existence}

We now apply the contraction argument used in the Radii Polynomial Theorem~\ref{thm:RadiiPolynomial} to the projected boundary value problem. Since $Z(r)$ in \eqref{eq:Z_BVP} bounds the derivative on the whole ball, the conditions below use $Y^i+Z^i(r)\max_jr^j$ in place of the integrated form of Theorem~\ref{thm:RadiiPolynomial}.

\begin{theorem}\label{thm:BVP_existence}
    Let $\bar x=(\bar\theta,\bar\phi,\bu)$ be a numerical approximation to a zero of $F$ in \eqref{eq:zero_finding_BVP}. Let $B$ be the approximate inverse constructed from \eqref{eq:Bdag_finite}, and assume that
    \[
        \bigl\|\mathfrak b(I-BB^\dagger)\bigr\|_1\le Z_0+Z_1<1.
    \]
    Let $Y=(Y^\theta,Y^\phi,Y^u)$ and $Z(r)=(Z^\theta(r),Z^\phi(r),Z^u(r))$ be the componentwise bounds of Sections~\ref{sec:BVP_Y} and~\ref{sec:BVP_Z}, and set $\|Z(r)\|\bydef\max_i Z^i(r)$. Suppose that there is a radius vector $r=(r^\theta,r^\phi,r^u)>0$ for which
    \begin{align}\label{eq:radii_BVP}
        Y^i+Z^i(r)\max_{j\in\{\theta,\phi,u\}}r^j\le r^i
        \quad(i\in\{\theta,\phi,u\}),
        \qquad \|Z(r)\|<1.
    \end{align}
    Then $F$ has a unique zero $\tilde x=(\tilde\theta,\tilde\phi,\tilde u)$ in $B_r(\bar x)$. In particular, $\tilde u$ is a solution of \eqref{eq:PDEs} on $[0,\tau]$ satisfying
    \[
        \tilde u(0)=P(\tilde\theta)\in\mathbb{E}^u(0),
        \qquad
        \tilde u(\tau)=Q(\tilde\phi)\in\hat u+\mathbb{E}^s(\hat u),
        \qquad l(\tilde\theta)=\delta^2,
    \]
    where $\mathbb{E}^u(0)=\mathrm{Range}(V^u)$ and $\mathbb{E}^s(\hat u)=\mathrm{Range}(V^s)$ are the prescribed projected endpoint spaces.
\end{theorem}

\begin{proof}
    The reduced inverse construction and \eqref{eq:LinearizedProb_BVP} give a bounded inverse of $DF(\bar x)$. Therefore the fixed points of $T_{\mathtt{BVP}}$ in \eqref{eq:T_BVP} coincide with the zeros of $F$. The estimates of Sections~\ref{sec:BVP_Y} and~\ref{sec:BVP_Z} give
    \[
        \llbracket T_{\mathtt{BVP}}(\bar x)-\bar x\rrbracket\le Y,
        \qquad
        \llbracket DT_{\mathtt{BVP}}(x)h\rrbracket\le Z(r)
    \]
    for $x\in B_r(\bar x)$ and $\|h\|_{X_{\mathtt{BVP}}}\le1$.
    Since $B_r(\bar x)$ is convex, the mean value inequality yields, componentwise,
    \[
        \llbracket T_{\mathtt{BVP}}(x)-\bar x\rrbracket_i
        \le Y^i+Z^i(r)\|x-\bar x\|_{X_{\mathtt{BVP}}}
        \le Y^i+Z^i(r)\max_j r^j
        \le r^i.
    \]
    Hence $T_{\mathtt{BVP}}$ maps the closed product ball into itself. The same inequality applied to $x,y\in B_r(\bar x)$ gives
    \[
        \|T_{\mathtt{BVP}}(x)-T_{\mathtt{BVP}}(y)\|_{X_{\mathtt{BVP}}}
        \le\|Z(r)\|\,\|x-y\|_{X_{\mathtt{BVP}}}.
    \]
    Thus $T_{\mathtt{BVP}}$ is a contraction on the complete metric space $B_r(\bar x)$ and has a unique fixed point there. The third component of $F(\tilde x)=0$ gives the PDE and its initial value, the second gives the terminal condition, and the first gives the amplitude condition.
\end{proof}

\begin{remark}[{\bf Multi-step implementation}]\label{rem:BVP_multistep}
    For a longer time of flight, we replace $T$ by the multi-step map $\bm T=(T_m)_{m=1}^M$ of \eqref{eq:fixed-point-form} and $X$ by the corresponding product space. The trajectory defect is bounded on every step, and the full derivative is controlled by the maximum of the row sums described above. The inverse reduction is unchanged: its endpoint block uses
    \[
        U(\tau,0)V^u
        =U_M(\tau_M,\tau_{M-1})\cdots U_1(\tau_1,\tau_0)V^u,
    \]
    together with $V^s$. We bound these products and $U_{DP}$ using Proposition~\ref{prop:multistep_bounds}, in particular \eqref{eq:SR_split}. For the product trajectory norm, we take the maximum of the bounds on the individual steps. The finite endpoint matrix \eqref{eq:Bdag_finite} is the only matrix that must be inverted numerically.
\end{remark}

\section{Application to the 2D Ohta--Kawasaki equation}\label{sec:OK_application}

We now apply the preceding estimates to the three trajectories shown in Figure~\ref{fig:dynamics_overview}. All three start near the same unstable (zero) equilibrium, but they describe different dynamics: a global solution numerically approaching a stripe pattern, a finite orbit segment satisfying projected boundary conditions at a saddle, and a global solution numerically approaching a spot pattern. We present the three computer-assisted results in this order. The first and third use the fixed-point argument on an unbounded interval from Section~\ref{sec:global_in_time}; the second uses the projected BVP formulation of Section~\ref{sec:BVP}.

The equation is
\begin{align}\label{eq:OK_equation}
 u_t=-\Delta\left(\varepsilon^2\Delta u+u-u^3\right)-\sigma(u-\bar m),
\end{align}
posed on $\Omega=[0,\pi/L_1]\times[0,\pi/L_2]$ with the boundary conditions induced by even periodic extension. Here $\varepsilon>0$ is the interfacial parameter, $\sigma\ge0$ is the long-range interaction strength, and $\bar m$ is the prescribed mean, which is preserved when the initial datum has that mean. The equation was introduced as a model for microphase separation in diblock copolymers~\cite{Ohta-Kawasaki}. Computer-assisted studies of its equilibria and bifurcation diagrams include~\cite{MR3636312,MR3904424,MR4159300,MR4191556}; a proof of connecting orbits in one space dimension is given in~\cite{MR3773757}.

{
The two-dimensional phase diagram for energy minimizers is studied
numerically and asymptotically in \cite{ChoksiMarasWilliams2011}.
Branch interactions and their influence on long-term pattern selection
in one dimension are investigated in \cite{JohnsonSanderWanner2013}.
These studies provide context for the stripe and spot patterns considered here.
}

Throughout this section, we fix
\begin{align}\label{eq:OK_params}
 \sigma=1,\qquad \bar m=0,\qquad \bL=(L_1,L_2)=(1,1.1),\qquad \varepsilon=0.39.
\end{align}
We use the cosine expansion~\eqref{eq:cosine-expansion} and the sequence-space notation of Section~\ref{sec:functional_analytic_setup}, with $\nu=1$. In particular, $u$ denotes both the solution and its Fourier coefficients, and the zero coefficient is fixed at zero. The equation has the form~\eqref{eq:ODE-sequence-space}, with
\begin{align}\label{eq:OK_mu}
 \mu_{\bk}=-\sigma+(\kk\bL)^2-\varepsilon^2(\kk\bL)^4,
 \qquad \cN(u)=u\ast u\ast u,\qquad q=1.
\end{align}
Thus the cubic term is acted on by the Laplacian multiplier $\cQ$ defined in~\eqref{eq:def_Q}. No additional normalization of the Fourier coefficients is introduced.

All computations in this section were performed on a Windows 10 machine equipped with an Intel Core i9-10900K CPU running at 3.70\,GHz, using MATLAB R2024b and INTLAB version~13 \cite{Ru99a} for interval arithmetic. The code is partially implemented with the generative AI (ChatGPT 6.1 sol) and is publicly available under the licensing terms described in the repository \cite{bib:codes}.

\subsection{Equilibria and the three trajectories}\label{sec:OK_linear}

The signs of $\mu_{\bk}$ determine the linear stability of the origin. For $(\kk\bL)^2>\sigma$, the mode $\bk$ changes stability at
\[
 \varepsilon_{\bk}^2=\frac{(\kk\bL)^2-\sigma}{(\kk\bL)^4}.
\]
The four modes relevant at the parameter value~\eqref{eq:OK_params} have critical values
\begin{align}\label{eq:OK_bifs}
 \varepsilon_{(1,1)}\approx0.497738,\quad
 \varepsilon_{(2,0)}\approx0.433013,\quad
 \varepsilon_{(0,2)}\approx0.404874,\quad
 \varepsilon_{(2,1)}\approx0.393825.
\end{align}
At $\varepsilon=0.39$, these are exactly the modes with $\mu_{\bk}>0$, so the origin has Morse index four.

The numerical bifurcation diagram in Figure~\ref{fig:bifurcation_diagram} identifies three target patterns: the stable stripe $u_*^{(2,0)}$, the saddle $u_3$ with dominant mode $(2,1)$ and Morse index three, and the stable spot $u_*^{(1,1)}$. The stripe and spot approximations have $\ell^1_\nu$ norms approximately $0.441$ and $0.737$, respectively. We write $\bu_*^{(2,0)}$ and $\bu_*^{(1,1)}$ for the finite Fourier approximations used as centers in the global-existence estimates. This distinction matters because Corollary~\ref{cor:global_existence} encloses a trajectory around an approximate equilibrium and includes its nonzero defect in the $Y$-bound.

The initial data for the two IVPs are the coefficient arrays stored in \path{IC_num5.mat} and \path{IC_num7.mat}, extended by zero outside their Fourier truncations. We denote them by $u_0^{(5)}$ and $u_0^{(7)}$, respectively. These are two fixed initial data near the origin; the conclusions below concern these data, rather than every point in a neighborhood. In the BVP, the initial and terminal parameters are instead unknowns determined together with the orbit.

\subsection{First trajectory: global existence near the stripe pattern}\label{sec:OK_GE_proof}

The left panel of Figure~\ref{fig:dynamics_overview} corresponds to $u_0^{(5)}$. The numerical trajectory moves away from the origin and approaches the stripe pattern. The purpose of the validation is to follow this trajectory beyond the finite integration interval and to enclose it for every subsequent time.

We first apply the validation procedure of
Section~\ref{sec:multi-stepping} on $[0,\tau_M]$, with
$M=208$ and $\tau_M=13$. The validated endpoint enclosure
and its distance to $\bu_*^{(2,0)}$ give the matching bound
$\rho_0\approx1.4109\times10^{-3}$.
The infinite-step test then succeeds directly on
$[\tau_M,\infty)$, without any intermediate steps.
The finite-time validation and the infinite step therefore
cover the whole interval $[0,\infty)$.
The corresponding data, together with those for the third
trajectory, are collected in Table~\ref{tab:OK_GE}.

\begin{table}[htbp]
	\centering
	\caption{Discretization parameters and rounded values of the validation bounds for the two global-existence computations. Here $M$ is the number of non-autonomous steps retained in the proof, $\tau_M$ is their terminal time, and $\tau_{M+M_a}$ is the start of the infinite step. The displayed decimals summarize the interval computations; the radii-polynomial tests use the unrounded interval bounds.}\label{tab:OK_GE}\vspace*{2mm}
	\small
	\setlength{\tabcolsep}{3pt}
	\begin{tabular}{l|cc}
    \hline
     & \texttt{num5} (stripe $u_*^{(2,0)}$)
             & \texttt{num7} (spot $u_*^{(1,1)}$) \\
    \hline
    Fourier truncation $\bN$
        & $[20,20]$ & $[28,28]$ \\
    variational truncation $\bbm$
        & $[7,7]$ & $[7,7]$ \\
    available non-autonomous steps
        & $208$ (step size $1/16$)
        & $138$ (step size $1/16$) \\
    validated time $\tau_M$ ($M$ steps)
        & $13.0$ ($208$) & $8.0$ ($128$) \\
    final non-autonomous radius $r_{\mathrm{end}}$
        & $5.2578\times10^{-6}$ & $2.1166\times10^{-5}$ \\
    $\|\bu_*\|_\nu$; defect at $\bu_*$ (approx.)
        & $0.44104$; $4.4325\times10^{-8}$
        & $0.73723$; $8.2518\times10^{-8}$ \\
    autonomous chain ($\bbm_a$, $h_a$, $M_a$)
        & --- & $[12,12]$, $0.1$, $69$ \\
    entry time $\tau_{M+M_a}$
        & $13.0$ & $14.9$ \\
    terminal truncation $\bbm_\infty$
        & $[20,20]$ & $[28,28]$ \\
    $\bm{W_\infty}$; $\bm{V_{0,\infty}}$;
    $\bm{V_{q,\infty}}$ (approx.)
        & $1.2017$; $11.405$; $25.205$
        & $2.1379$; $188.66$; $31.309$ \\
    $\boldsymbol{r_\infty}$; $Z$ (upper bounds)
        & $\mathbf{1.8044\times10^{-3}}$; $0.12060$
        & $\mathbf{6.2165\times10^{-3}}$; $0.86457$ \\
    \hline
\end{tabular}
\end{table}

% Mathematical review: the bridge recurrence in the implementation must be
% matched to the hypotheses of Corollary~\ref{cor:global_existence}; see
% STYLE_REVIEW_NOTES.md. This editorial revision does not rerun validation.
\begin{theorem}[{\bf Global existence for the stripe trajectory}]\label{thm:OK_GE}
 For equation~\eqref{eq:OK_equation} with parameters~\eqref{eq:OK_params} and initial datum $u_0^{(5)}$, the validation bounds in Table~\ref{tab:OK_GE} give the global-existence conclusion of Corollary~\ref{cor:global_existence}. In particular, the solution $\tilde u^{(5)}$ exists for all $t\ge0$ and satisfies
 \begin{align}\label{eq:OK_stripe_enclosure}
  \|\tilde u^{(5)}(t)-\bu_*^{(2,0)}\|_\nu\le r_\infty^{(5)}
  \qquad (t\ge\tau_{M}^{(5)}),
 \end{align}
 where the certified radius and entry time have the rounded values
 \[
  r_\infty^{(5)}\approx1.8044\times10^{-3},\qquad
  \tau_{M}^{(5)}=13.0.
 \]
\end{theorem}

The infinite-step contraction bound is approximately
$1.20595\times10^{-1}$, and the computed radii-polynomial
interval is strictly negative. These tests extend the
finite-time enclosure to the unbounded interval $[13,\infty)$.
Here $\bu_*^{(2,0)}$ is the computed approximation of the
stripe-pattern equilibrium, and
\eqref{eq:OK_stripe_enclosure} gives the uniform enclosure
around this approximation supplied by the fixed-point argument.

\subsection{Second trajectory: a projected boundary value problem}\label{sec:OK_BVP_proof}

The middle panel of Figure~\ref{fig:dynamics_overview} concerns the saddle $u_3$.
Since the origin has four unstable directions and $u_3$ has three,
the finite-dimensional endpoint parameters consist of four initial
coefficients and $n-3$ terminal coefficients.
These give $n+1$ unknowns for the $n$ finite endpoint matching conditions.
The scalar condition~\eqref{eq:phase_condition} supplies the remaining
equation, and the resulting system is the projected BVP of
Section~\ref{sec:BVP}.

We fix the time of flight $\tau=0.1$ and the phase radius $\delta=0.01$. The endpoint conditions are $u(0)=P(\theta)$ and $u(\tau)=Q(\phi)$, with the affine map $Q$ including the tail as in Section~\ref{sec:BVP}. In the numerical data, the four unstable columns in the numerical data are ordered as $(2,0)$, $(2,1)$, $(1,1)$, and $(0,2)$. 
Newton refinement gives $\bar\theta\approx(0,0.01,0,0)$, so the orbit leaves the origin predominantly in the $(2,1)$ direction. The finite terminal parameter satisfies $\|\bar\phi\|_\nu\approx0.06406$.

The computation uses $\bN=(12,12)$, with $n=143$ active modes after removing the fixed zero mode, and $\bbm=(3,3)$ for the variational equation. We divide $[0,0.1]$ into $M=4$ steps. The first three have length $2^{-5}$, and the last one is $[0.09375,0.1]$.
The numbers of retained Chebyshev coefficients on these steps are $20$, $12$, $9$, and $6$, respectively.
The smallest one-step Schur complement is approximately $0.99827$, and the evolution operator satisfies $\|U(\cdot,0)\|\approx1.1325$. The initial value part gives $Y\approx1.6437\times10^{-11}$, $r\approx1.6438\times10^{-11}$, and $Z\approx3.4430\times10^{-12}$. For the reduced endpoint operator, $\|I-BB^\dagger\|_{\mathrm{op}}\le2.40\times10^{-6}$ and $U_{DP}\approx4.1963$. These bounds enter the block estimates of Sections~\ref{sec:BVP_Y} and~\ref{sec:BVP_Z}.
 
% Mathematical review: recheck the reported BVP radius against the phase
% contribution and full-path norm in the revised Z-bound; see
% STYLE_REVIEW_NOTES.md. The displayed decimals are rounded log output.
\begin{theorem}[{\bf Existence of the projected orbit segment}]\label{thm:OK_BVP}
 Let $\bar x=(\bar\theta,\bar\phi,\bu)$ be the numerical BVP solution described above. The componentwise validation of Theorem~\ref{thm:BVP_existence} yields a unique zero $\tilde x=(\tilde\theta,\tilde\phi,\tilde u)$ of $F$ in $B_{r_0}(\bar x)$, where the radius vector has the rounded values
 \begin{align}\label{eq:OK_BVP_radii}
  r_0=(r^\theta,r^\phi,r^u)
  \approx\bigl(8.2301\times10^{-12},\ 3.3385\times10^{-11},\ 5.0973\times10^{-11}\bigr).
 \end{align}
 In particular, $\tilde u$ solves~\eqref{eq:OK_equation} on $[0,0.1]$ and satisfies
 \[
  \tilde u(0)=P(\tilde\theta),\qquad
  \tilde u(0.1)=Q(\tilde\phi),\qquad
  l(\tilde\theta)=\delta^2.
 \]
\end{theorem}

This theorem certifies a finite orbit segment whose endpoints lie in the prescribed projected affine spaces. It does not place these endpoints on the nonlinear unstable manifold of the origin or on the nonlinear stable manifold of $u_3$. A heteroclinic connection from the origin to $u_3$ requires these additional endpoint conditions, and the projected BVP serves as the starting point for that problem.

\subsection{Third trajectory: global existence near the spot pattern}\label{sec:OK_spot_proof}

The right panel of Figure~\ref{fig:dynamics_overview} corresponds to $u_0^{(7)}$. This numerical trajectory passes near the $(0,2)$ saddle before approaching the spot pattern. We use this observation to choose the approximation to validate; the global-existence result itself follows from the same framework used for the stripe trajectory, with an intermediate chain inserted before the infinite step.

For this orbit, we retain $M=128$ non-autonomous steps of length $1/16$, ending at $\tau_M=8$, and then use $M_a=69$ autonomous intermediate steps of length $h_a=0.1$. The infinite step starts at $\tau_{M+M_a}=14.9$. The Fourier truncation is $\bN=(28,28)$, while the non-autonomous and autonomous intermediate steps use the variational truncations $\bbm=(7,7)$ and $\bbm_a=(12,12)$, respectively. The one-step estimates are computed and assembled along the intermediate chain.

\begin{theorem}[{\bf Global existence for the spot trajectory}]\label{thm:OK_GE_spot}
 For equation~\eqref{eq:OK_equation} with parameters~\eqref{eq:OK_params} and initial datum $u_0^{(7)}$, the validation bounds in Table~\ref{tab:OK_GE} give the global-existence conclusion of Corollary~\ref{cor:global_existence}. The solution $\tilde u^{(7)}$ exists for all $t\ge0$ and satisfies
 \begin{align}\label{eq:OK_spot_enclosure}
  \|\tilde u^{(7)}(t)-\bu_*^{(1,1)}\|_\nu\le r_\infty^{(7)}
  \qquad (t\ge\tau_{M+M_a}^{(7)}),
 \end{align}
 where the certified radius and entry time have the rounded values
 \[
  r_\infty^{(7)}\approx6.2165\times10^{-3},\qquad
  \tau_{M+M_a}^{(7)}=14.9.
 \]
\end{theorem}

Here the infinite-step contraction bound is approximately $8.6456\times10^{-1}$, again with a strictly negative radii-polynomial interval. Theorems~\ref{thm:OK_GE}, \ref{thm:OK_BVP}, and~\ref{thm:OK_GE_spot} give the three conclusions represented from left to right in Figure~\ref{fig:dynamics_overview}: global existence with a stripe enclosure, existence of a projected BVP orbit, and global existence with a spot enclosure. The two uniform enclosures are the conclusions of Corollary~\ref{cor:global_existence}; asymptotic convergence to exact equilibria requires an additional decay argument.
% 
% The computations use INTLAB~\cite{Ru99a} for interval arithmetic. The two IVP validations are implemented in \path{script_proof_GE_2DOK_num5.m} and \path{script_proof_GE_2DOK_num7.m}, which call the common routine \path{proof_GE_2DOK_core.m}; the projected BVP uses \path{script_proof_connexion4to3.m}. 
The bounds in Table~\ref{tab:OK_GE} and~\eqref{eq:OK_BVP_radii} summarize the saved validation data \cite{bib:codes}.

\section{Conclusion}

We have developed a constructive method to rigorously solve initial and projected boundary value problems for semilinear parabolic PDEs. The main contribution is the construction of evolution-operator estimates that retain information about individual Fourier modes and preserve cancellations when propagating errors across time steps. These estimates use forward variational integration and semigroup comparison, while parabolic smoothing controls the infinite tail and derivatives in the nonlinear term. The same estimates support finite-time trajectory validation, global enclosures, and projected BVPs, for which the construction of the linearized inverse reduces to a finite system for the endpoint parameters. For global validation, an intermediate chain connects the finite-time enclosure to the infinite-step test when a direct connection is not available.

We illustrated the method with three trajectories of the two-dimensional Ohta--Kawasaki equation. For two prescribed initial data, we proved global existence and obtained explicit uniform enclosures around numerical stripe and spot approximations for every time beyond certified entry times. For the third trajectory, we validated a finite orbit segment satisfying projected endpoint conditions associated with the origin and a numerical saddle, with uniqueness in an explicit product ball for the endpoint parameters and the trajectory. These results provide quantitative information about the trajectories, although they do not establish convergence to exact equilibria or a heteroclinic connection.

Further work concerns the dynamical interpretation of these validated trajectories. The eventual enclosures could be combined with validated equilibrium and decay estimates to establish convergence to exact equilibria. Similarly, combining the projected BVP construction with validated local parameterizations of unstable and stable manifolds would provide a route to proving heteroclinic connections. The evolution estimates and the finite endpoint reduction developed here provide a basis for these extensions. Further reductions in the wrapping effect, for example through affine arithmetic, may also improve the bounds over longer integration intervals.

\section*{Acknowledgement}
JPL was supported by the Natural Sciences and Engineering Research Council of Canada (NSERC).
AT was supported by the Japan Science and Technology Agency (JST) FOREST Program (Grant Number JPMJFR246A), and by JSPS KAKENHI Grant Numbers 24K00538 and 26K00619.

\bibliographystyle{plain}
\bibliography{papers}

\appendix
\section{Newton refinement of the approximate solution}\label{sec:Newton-refinement}

We refine the numerical solution of the truncated system
\eqref{eq:ODEs_finite} by Newton iteration on its Chebyshev-Fourier
coefficients, following the Chebyshev-series approach of
\cite{MR3148084}. We describe the construction on a time step $[\tau_{m-1},\tau_m]$ of length $h_m$. On the first step, the initial condition is $u_0$. On each subsequent step, it is (an approximation of) the endpoint value of the approximation on the preceding step; any mismatch is included in the defect bound through $y_2^m$ in \eqref{eq:local_defects}. The resulting approximation is validated by the
Radii Polynomial Theorem in Section~\ref{sec:trajectory_validation}.

Let

\[
  \theta(t)\bydef \frac{2(t-\tau_{m-1})}{h_m}-1.
\]
For a prescribed Chebyshev order $n$, we look for coefficients
$a=(a_{\ell,\bk})_{0\le \ell<n,\ \bk\in\FN}$ such that
\[
  \bu_{\bk}^{(\bN)}(t)=
  \sum_{\ell=0}^{n-1}a_{\ell,\bk}T_\ell(\theta(t)),
  \qquad \bk\in\FN .
\]
The coefficients are determined as a zero of a finite-dimensional map
\[
  F_{\rm ivp}(a)=0,\qquad
  F_{\rm ivp}=(F_{\rm ode},F_{\rm init}).
\]
The first block $F_{\rm ode}\in\R^{(n-1)|\FN|}$ consists of the first
$n-1$ Chebyshev coefficients of the differential equation residual
\[
  \frac{d}{dt}\bu_{\bk}^{(\bN)}(t)
  -\mu_{\bk}\bu_{\bk}^{(\bN)}(t)
  -(-1)^q(\kk\bL)^{2q}\cN_{\bk}(\bu^{(\bN)}(t)),
  \qquad \bk\in\FN .
\]
The remaining block $F_{\rm init}\in\R^{|\FN|}$ is the bordering equation
which imposes the initial condition:
\[
  (F_{\rm init})_{\bk}
  =\sum_{\ell=0}^{n-1}(-1)^\ell a_{\ell,\bk}-(u_0)_{\bk},
  \qquad \bk\in\FN .
\]
Here $u_0$ denotes the initial condition of the current step. Thus $F_{\rm ivp}$ has $n|\FN|$ equations for the $n|\FN|$ unknown
Chebyshev-Fourier coefficients.

The derivative $DF_{\rm ivp}(a)$ is assembled explicitly.  The linear part
comes from the Chebyshev differentiation matrix, the diagonal Fourier
multiplier $\mu_{\bk}$, and the bordered row corresponding to evaluation at
the initial time of the step.  The nonlinear part is obtained by applying the Fr\'echet derivative
$D\cN(\bu^{(\bN)})$ in Chebyshev-Fourier coefficient space.  Starting from
the Chebyshev interpolation of the numerical orbit, we perform Newton
iterations
\[
  a^{j+1}=a^j-\left(DF_{\rm ivp}(a^j)\right)^{-1}F_{\rm ivp}(a^j)
\]
until the finite-dimensional residual is sufficiently small.  The resulting
coefficients define the refined Chebyshev-Fourier polynomial
$\bu^{(\bN)}$, which is the approximate orbit used in the validation
estimates.

\section{Details of the finite-time validation}\label{app:linear_validation_details}
This appendix supplies the proofs  and the tail constants used in Section~\ref{sec:multi-stepping}. The main validation argument only requires their conclusions.

\subsection{Proof of the Radii Polynomial Theorem}\label{app:validation_proofs}
\begin{proof}[Proof of Theorem~\ref{thm:RadiiPolynomial}]
Let $\mathcal B=\prod_{i=1}^M B_{r_i}(\bu_i)$. We first show that $\bm T$ maps $\mathcal B$ into itself. We then prove that it is a contraction in a weighted product norm. This follows the product-space argument of \cite[Theorem~2.12]{JBMaxime}, with the integrated derivative bounds used here.

Fix $u\in\mathcal B$ and set $h_i=u_i-\bu_i$ and $a_i=\|h_i\|_{X_i}\le r_i$. The segment $u(\theta)=\bu+\theta(u-\bu)$, $0\le\theta\le1$, lies in $\mathcal B$. Since $\bm T$ is continuously Fr\'echet differentiable and $T_m$ depends only on the first $m$ components, the fundamental theorem of calculus gives
\begin{equation}\label{eq:proof_component_segment}
 T_m(u)-T_m(\bu)=\int_0^1\sum_{i=1}^m D_iT_m(u(\theta))h_i\,d\theta.
\end{equation}
For $\theta>0$ and $a_i>0$, the derivative hypothesis applies with radius $\theta a_i$, since $\|u_i(\theta)-\bu_i\|_{X_i}=\theta a_i$. It is independent of the other components of $u(\theta)$. The contribution of an index with $a_i=0$ is zero. The value at $\theta=0$ does not affect the integral. Taking norms therefore gives
\begin{equation}\label{eq:proof_integrated_derivative}
 \|T_m(u)-T_m(\bu)\|_{X_m}
 \le\sum_{i=1}^m a_i\int_0^1 Z_m^i(\theta a_i)\,d\theta
 =\sum_{i=1}^m\int_0^{a_i} Z_m^i(s)\,ds.
\end{equation}
Using the defect bound and the first inequality in \eqref{eq:validation_inequalities}, we obtain
\begin{equation}\label{eq:proof_invariant_box}
 \|T_m(u)-\bu_m\|_{X_m}
 \le Y_m+\sum_{i=1}^m\int_0^{a_i}Z_m^i(s)\,ds
 \le Y_m+\sum_{i=1}^m\int_0^{r_i}Z_m^i(s)\,ds
 \le r_m.
\end{equation}
The second inequality uses $a_i\le r_i$ and the nonnegativity of $Z_m^i$. Thus every component of $\bm T(u)$ lies in its prescribed ball, and $\bm T(\mathcal B)\subset\mathcal B$.

For the contraction estimate, use the positive weights $\zeta_i$ from the theorem and define
\begin{equation}\label{eq:proof_weighted_norm}
 \|w\|_{\bm X,\zeta}\bydef\max_{1\le i\le M}\frac{\|w_i\|_{X_i}}{\zeta_i}.
\end{equation}
This norm is equivalent to the usual maximum norm because $M$ is finite and all weights are positive. Hence $\bm X$ is complete in this norm. Each component ball is closed, so $\mathcal B$ is a complete metric space with the induced distance. The weights need not equal the radii.

Let $u,v\in\mathcal B$. By convexity, $w(\theta)=v+\theta(u-v)$ remains in $\mathcal B$. Applying the fundamental theorem of calculus along this segment, and now using the derivative bound at the fixed radius $r_i$, yields
\begin{equation}\label{eq:proof_component_contraction}
 \|T_m(u)-T_m(v)\|_{X_m}
 \le\sum_{i=1}^m Z_m^i(r_i)\|u_i-v_i\|_{X_i}
 \le\left(\sum_{i=1}^m Z_m^i(r_i)\zeta_i\right)\|u-v\|_{\bm X,\zeta}.
\end{equation}
Define
\begin{equation}\label{eq:proof_contraction_constant}
 \kappa\bydef\max_{1\le m\le M}\frac{1}{\zeta_m}\sum_{i=1}^m Z_m^i(r_i)\zeta_i<1.
\end{equation}
The strict inequality follows from the second inequality in \eqref{eq:validation_inequalities} and the finite number of components. Dividing \eqref{eq:proof_component_contraction} by $\zeta_m$ and taking the maximum gives
\begin{equation}\label{eq:proof_product_contraction}
 \|\bm T(u)-\bm T(v)\|_{\bm X,\zeta}\le\kappa\|u-v\|_{\bm X,\zeta}.
\end{equation}
Thus $\bm T$ is a contraction from $\mathcal B$ into itself. The Banach fixed-point theorem gives a unique fixed point $\tu\in\mathcal B$, with $\|\tu_m-\bu_m\|_{X_m}\le r_m$ for every $m$. In the present application, Remark~\ref{rem:matching_traces} shows that its components join continuously and satisfy the prescribed initial condition. The fixed-point formula then identifies this trajectory with a solution of the original initial value problem.
\end{proof}

\subsection{Tail constants for the Ohta--Kawasaki operator}\label{app:one_step_details}
 We bound the tail constants \eqref{eq:Vinfq} for the operator of Section~\ref{sec:OK_application}. For this operator, $q=1$ and $-\mu_{\bk}=\delta( \xi_{\bk})$, where
\[
  \xi_{\bk}\bydef(\kk\bL)^2,\qquad \delta( \xi)\bydef\varepsilon^2{ \xi}^2- \xi+\sigma,
 \qquad r_{\bbm}\bydef\min_{1\le j\le d}(K_jL_j)^2.
\]
Every tail mode satisfies $ \xi_{\bk}\ge r_{\bbm}$. We verify
$\delta(r_{\bbm})>0$, $\delta'(r_{\bbm})>0$, and $\varepsilon^2r_{\bbm}^2>\sigma$.
Then $\delta$ is positive and increasing on  $[r_{\bbm},\infty)$, while
\[
 \left(\frac{ \xi}{\delta( \xi)}\right)'=\frac{\sigma-\varepsilon^2{ \xi}^2}{\delta( \xi)^2}<0.
\]
Thus the infinite-time tail constants  (see \eqref{eq:Vinfq_infty}) are bounded by $1/\delta(r_{\bbm})$ and
$r_{\bbm}/\delta(r_{\bbm})$. On a finite interval,  $V_{\infty,0}^{(J_m,\tau_{m-1})}$ is bounded
by $(1-e^{-\delta(r_{\bbm})h_m})/\delta(r_{\bbm})$.  To obtain a sharper bound for $V_{\infty,q}^{(J_m,\tau_{m-1})}$, we evaluate the integral for each $\bk\notin\FK$ with $\bk<\bN'$, and bound it for the remaining modes by $r_{\bN'}/\delta(r_{\bN'})$, where $\bN'>\bbm$. The three conditions above also hold with $r_{\bbm}$ replaced by $r_{\bN'}\ge r_{\bbm}$, since $\delta$, $\delta'$, and $\varepsilon^2\xi^2$ are increasing on $[r_{\bbm},\infty)$. Hence $\xi_{\bk}/\delta(\xi_{\bk})\le r_{\bN'}/\delta(r_{\bN'})$ for every remaining mode, and the integral for such a mode is bounded by this quantity. The maximum of these two bounds bounds $V_{\infty,q}^{(J_m,\tau_{m-1})}$. Note that a maximum over finitely many modes alone does not give such a bound.

\subsection{Proof of the multi-step bounds}\label{app:multistep_proof}
In this subsection, we prove Lemma~\ref{lem:endpoint_propagation} and Proposition~\ref{prop:multistep_bounds}, using the vectors $z(\cdot)$  defined in \eqref{eq:profile_vector}.

\begin{proof}[Proof of Lemma~\ref{lem:endpoint_propagation}]
 The product $E_i\cdots E_{j+1}$ is block diagonal,
\begin{equation}\label{eq:Eprod_block}
 E_i\cdots E_{j+1}
 =\operatorname{diag}\bigl(C_{i,j}^{(\bbm)},
                         e^{\cL_\infty(\tau_i-\tau_j)}\bigr),
 \qquad C_{i,j}^{(\bbm)}\in{ \Phi}_{ \mathrm{end},i,j}{ ,}
\end{equation}
where
\[
 C_{i,j}^{(\bbm)}\bydef
 C_i^{(\bbm)}(\tau_i,\tau_{i-1})\cdots
 C_{j+1}^{(\bbm)}(\tau_{j+1},\tau_j)
\]
lies in $\Phi_{\mathrm{end},i,j}$ by inclusion monotonicity of interval multiplication. Hence
\[
 D_\nu|C_{i,j}^{(\bbm)}\phi^{(\bbm)}|
 \le D_\nu\operatorname{mag}(\Phi_{\mathrm{end},i,j})
 D_\nu^{-1}D_\nu|\phi^{(\bbm)}|.
\]
Since $\|e^{\cL_\infty t}\phi^{(\infty)}\|_\nu
\le e^{\mu_{\bbm}t}\|\phi^{(\infty)}\|_\nu$, this proves
$z(E_i\cdots E_{j+1}\phi)\le R_{i,j}^{(\mathrm{unc})}z(\phi)$
and, taking block norms,
$\mathfrak b(E_i\cdots E_{j+1})\le\bar E_{i,j}$.

For the one-step remainder, from the variation of constants formula on $J_k$  with $f=0$ , $\widetilde C_k\phi$  is the value at $\tau_k$  of the terms involving $\mathsf G_k$  in \eqref{eq:coupled_components}, so that $z(\widetilde C_k\phi)\le\mathsf G_k\mathsf U_k z(\phi)$ . Taking block norms yields $\mathfrak b(\widetilde C_k)\le\widetilde\triangle_k$ with \eqref{eq:Delta_k_def}. Applying $\mathfrak b$  to the telescoping identity \eqref{eq:telescoping}, using subadditivity and submultiplicativity, and arguing by induction on $i-j$  gives \eqref{eq:product_remainder_bound}. Adding $\mathfrak b(E_i\cdots E_{j+1})\le\bar E_{i,j}$  gives \eqref{eq:full_endpoint_bound}.
\end{proof}

\begin{proof}[Proof of Proposition~\ref{prop:multistep_bounds}]
\begingroup
Let $M_W$ and $M_V$ denote the $2\times2$ matrices inside the norms in \eqref{eq:SR_split}-\eqref{eq:WV_onestep}, and let $v_m\bydef v|_{J_m}$ for the solution $v$ of \eqref{eq:model_linear_problem}. We show that
\begin{equation}\label{eq:multistep_bound}
 \operatorname{pair}(z(v_m))\le M_W\operatorname{pair}(z(\phi))+M_V\operatorname{pair}(z(f)).
\end{equation}
Since the entries of $\operatorname{pair}(z(w))$ sum to the norm of $w$ by \eqref{eq:profile_properties}, the induced matrix $1$-norm then gives $\|v_m\|_{X_m}\le\|M_W\|_1\|\phi\|_\nu+\|M_V\|_1\|f\|_{X_{j+1}}$, which is \eqref{eq:WVbounds}. We use the inequalities $\operatorname{pair}(Nx)\le\mathfrak b(N)\operatorname{pair}(x)$ for $N\ge0$, and $\operatorname{pair}(z(A\phi))\le\mathfrak b(A)\operatorname{pair}(z(\phi))$ for bounded operators $A$.

If $j=m-1$, Proposition~\ref{prop:one_step} gives $z(v_m)\le\mathsf U_mz(\phi)+\mathsf V_{\star,m}z(f)$, and \eqref{eq:multistep_bound} follows by applying $\operatorname{pair}$.

Let $j\le m-2$. By linearity, we treat the initial value and the forcing separately. On $J_m$, Proposition~\ref{prop:one_step} with zero forcing gives $z(v_m)\le\mathsf U_mz(v(\tau_{m-1}))$. For the initial value, $v(\tau_{m-1})=U(\tau_{m-1},\tau_j)\phi$, which we split as in \eqref{eq:telescoping}. Lemma~\ref{lem:endpoint_propagation} gives $\mathsf U_mz(E_{m-1}\cdots E_{j+1}\phi)\le\mathsf U_mR^{(\mathrm{unc})}_{m-1,j}z(\phi)$ for the uncoupled part, and applying $\operatorname{pair}$ yields $\mathfrak b(\mathsf U_mR^{(\mathrm{unc})}_{m-1,j})$; the signed products are thus kept until the block norms are taken. For the remainder, the two properties above and \eqref{eq:product_remainder_bound} yield $\mathfrak b(\mathsf U_m)\widetilde\triangle_{m-1,j}$. For the inhomogeneous term, Proposition~\ref{prop:one_step} on $J_{j+1}$ gives $z(v(\tau_{j+1}))\le\mathsf V_{\star,j+1}z(f)$, since $\tau_{j+1}\in J_{j+1}$. On $[\tau_{j+1},\tau_m]$, $v$ solves the homogeneous equation, and the same argument with $j+1$ in place of $j$ applies; if $j=m-2$, then $R^{(\mathrm{unc})}_{m-1,m-1}=I$ and $\widetilde\triangle_{m-1,m-1}=0$. Adding the contributions proves \eqref{eq:multistep_bound}.
\endgroup
\end{proof}

\section{Details of the global existence construction}\label{app:global_details}
This appendix supplies the linear bounds, endpoint propagation,
and optional product-space argument used in Section~\ref{sec:global_in_time}.
\subsection{Autonomous linear bounds}\label{app:infinite_linear_bounds}

The estimates needed for \eqref{eq:T_infty} are a uniform bound for the initial-data term and an integrated bound for the forcing term. We first obtain them for the uncoupled finite block and tail, and then resolve their coupling.

For a finite linearized block $\bbm$, let
\[
 A_{\bbm}\bydef\cL_{\bbm}+\varPi^{(\bbm)}\cQ D\cN(\bu_*)\varPi^{(\bbm)}
 \in\R^{\m\times\m}.
\]
Choose an invertible matrix $V$ and a diagonal matrix $\Lambda=\mathrm{diag}(\eta_0,\ldots,\eta_{\m-1})$ from a numerical diagonalization, and set $\bA=V\Lambda V^{-1}$. The matrix $\bA$ approximates $A_{\bbm}$; their difference is enclosed below. We assume
\begin{enumerate}[(i)]
 \item $\Re\eta_i<0$ for $i=0,\ldots,\m-1$;
 \item $\mu_{\bbm}<0$, with $\mu_{\bbm}$ defined in \eqref{eq:mu_bbm_def}.
\end{enumerate}
These conditions make the uncoupled exponential integrals finite. Validation of the diagonalization defect and the finite-tail coupling is also required.

For $\Re\eta_i<0$ and $\mu_{\bk}<0$, direct integration gives
\[
 \int_0^\infty e^{\Re\eta_i s}\,ds=-\frac1{\Re\eta_i},
 \qquad
 \int_0^\infty e^{\mu_{\bk}s}\qk\,ds=\frac{\qk}{|\mu_{\bk}|}.
\]
Hence the infinite-time counterparts of \eqref{eq:bV-esitmates} and \eqref{eq:Vinfq} are
\begin{align}
 \widehat V^{(\bbm)}_\infty
 &\bydef |V|\,\mathrm{diag}\!\left(-\frac1{\Re\eta_i}\right)|V^{-1}|,
 \label{eq:bV_infty}\\
 V_{\infty,q}^{(J_\infty,t_0)}
 &\bydef\sup_{\bk\notin\FK}\frac{\qk}{|\mu_{\bk}|}.
 \label{eq:Vinfq_infty}
\end{align}
The first matrix bounds the integrated semigroup of $\bA$, not yet that of $A_{\bbm}$.
To account for the diagonalization defect, define
\[
 E_{\bbm}(r)\bydef e^{A_{\bbm}r},\qquad
 \widehat C^{(\bbm)}(r)\bydef Ve^{\Lambda r}V^{-1},\qquad r\ge0.
\]
We use $\widehat C$ for the exponential of the approximate matrix, to distinguish it from the true fundamental solution $C$ in Section~\ref{sec:one-step}. Variation of constants gives
\begin{align}\label{eq:aut_defect_identity}
 E_{\bbm}(r)-\widehat C^{(\bbm)}(r)
 =\int_0^r Ve^{\Lambda(r-s)}V^{-1}(A_{\bbm}-\bA)E_{\bbm}(s)\,ds.
\end{align}
Let $\triangle_{\sup}$ be an entrywise enclosure of $|V^{-1}(A_{\bbm}-\bA)|$, and set
\[
 \tV^{(\bbm)}_\infty\bydef |V|\,\mathrm{diag}\!\left(-\frac1{\Re\eta_i}\right),
 \qquad K_\infty\bydef\tV^{(\bbm)}_\infty\triangle_{\sup},
 \qquad B_{\mathrm{neu},\infty}\bydef I-K_\infty.
\]
For each finite $T>0$, taking entrywise suprema in \eqref{eq:aut_defect_identity} yields
\[
 S_T\le |\widehat C^{(\bbm)}|_{\infty,[0,T]}+K_\infty S_T,
 \qquad S_T\bydef\sup_{0\le r\le T}|E_{\bbm}(r)|.
\]
We verify $\|K_\infty\|_\infty<1$ in interval arithmetic. Then
$B_{\mathrm{neu},\infty}^{-1}=\sum_{n\ge0}K_\infty^n\ge0$, so multiplication by this inverse preserves the inequality. The resulting bound is independent of $T$. Letting $T\to\infty$ proves
\[
 \sup_{r\ge0}|E_{\bbm}(r)|
 \le B_{\mathrm{neu},\infty}^{-1}|\widehat C^{(\bbm)}|_\infty.
\]
Introducing the Fourier weights gives the finite-block observation matrix
\begin{align}\label{eq:Phi_sup_infty}
 \Phi_{\sup,\infty}\bydef D_\nu B_{\mathrm{neu},\infty}^{-1}
 |\widehat C^{(\bbm)}|_\infty D_\nu^{-1}.
\end{align}
The same correction applies to the inhomogeneous problem, and all supremum estimates below are obtained in the same way: they are first taken over $0\le r\le T$ with constants independent of $T$, and then $T\to\infty$.
Indeed, for $\dot z=A_{\bbm}z+\psi$ with $z(0)=0$, variation of constants about $\bA$ gives
\[
 |z|_\infty\le K_\infty|z|_\infty+\widehat V^{(\bbm)}_\infty|\psi|_\infty.
\]
Consequently, its integrated semigroup matrix is
\[
 \mathcal{V}_{\bbm,\infty}\bydef B_{\mathrm{neu},\infty}^{-1}\widehat V^{(\bbm)}_\infty,
 \qquad
 \mathcal{V}_{\bbm,\infty}^{(w)}\bydef D_\nu\mathcal{V}_{\bbm,\infty}D_\nu^{-1}.
\]
It remains to evaluate $|\widehat C^{(\bbm)}|_\infty$ rigorously. All times in the following calculation are elapsed times $r=t-t_0$. We partition a finite window $[0,T_{\max}]$ into intervals $T_i=[\vartheta_i,\vartheta_{i+1}]$, with $\Delta_i=\vartheta_{i+1}-\vartheta_i$ and midpoint $\vartheta_i^{\mathrm m}=(\vartheta_i+\vartheta_{i+1})/2$. Since
\[
 (\widehat C^{(\bbm)})'(r)=V\Lambda e^{\Lambda r}V^{-1},\qquad
 (\widehat C^{(\bbm)})''(r)=V\Lambda^2 e^{\Lambda r}V^{-1},
\]
Taylor's formula about the midpoint, together with $|r-\vartheta_i^{\mathrm m}|\le\Delta_i/2$, gives
\begin{align}\label{eq:sup_meanvalue}
 \sup_{r\in T_i}|\widehat C^{(\bbm)}(r)|
 &\le |\widehat C^{(\bbm)}(\vartheta_i^{\mathrm m})|
 +\frac{\Delta_i}{2}|V\Lambda e^{\Lambda \vartheta_i^{\mathrm m}}V^{-1}|
 +\frac{\Delta_i^2}{8}|V|\,
 \mathrm{diag}\!\left(|\lambda_j|^2e^{\Re\lambda_j\vartheta_i}\right)|V^{-1}|.
\end{align}
The factor $\Delta_i^2/8$ is $\frac12(\Delta_i/2)^2$, and the exponential in the second-derivative bound is largest at $\vartheta_i$ because $\Re\lambda_j<0$. The first two matrix products are evaluated before taking absolute values, so they retain cancellation. Only the remainder uses an entrywise nonnegative majorant.

The partition may be chosen adaptively to control the remainder in \eqref{eq:sup_meanvalue}. For the remaining times,
\[
 \sup_{r\ge T_{\max}}|\widehat C^{(\bbm)}(r)|
 \le |V|\,\mathrm{diag}\!\left(e^{\Re\lambda_jT_{\max}}\right)|V^{-1}|.
\]
Taking the entrywise maximum of this tail bound and the finite-window bounds encloses the full supremum. These estimates also apply to complex eigenvalues, through $\Re\lambda_j$ and $|\lambda_j|^2$.

\paragraph{Finite-tail coupling.}
The preceding bounds control the finite matrix $A_{\bbm}$ and the diagonal tail separately. We now include the off-diagonal interactions and the tail self-interaction. Let $\Psi_\infty$ and $\varPhi_\infty$ be the derivative bounds of \eqref{eq:psi_def} and \eqref{eq:cout_def}, evaluated at $\bu_*$. Write
\[
 Q_{\bbm}\bydef\mathrm{diag}(\qk)_{\bk\in\FK}.
\]
The factor $Q_{\bbm}$ is needed because \eqref{eq:psi_def} bounds the tail-to-finite forcing by $Q_{\bbm}\Psi_\infty\|v^{(\infty)}\|$. We keep this multiplier explicit in the coupling vector:
\begin{align}
 c_{\mathrm{in},\infty}
 &\bydef D_\nu|\mathcal{V}_{\bbm,\infty}|Q_{\bbm}\Psi_\infty,
 &c_{\mathrm{out},\infty}^T
 &\bydef V_{\infty,q}^{(J_\infty,t_0)}\varPhi_\infty^TD_\nu^{-1},
 \label{eq:cinout_infty}\\
 g_{22,\infty}
 &\bydef V_{\infty,q}^{(J_\infty,t_0)}g(\|\bu_*\|_\nu),
 &\mathsf G_\infty
 &\bydef\begin{pmatrix}0&c_{\mathrm{in},\infty}\\c_{\mathrm{out},\infty}^T&g_{22,\infty}\end{pmatrix}.
 \label{eq:G_infty_cpl}
\end{align}
For the coupled problem, we likewise work first on
$[t_0,t_0+T]$. Let $x$ contain the weighted finite-mode suprema and the tail supremum of a solution of the linearized problem, let $p$ be the corresponding initial-data vector, and let $b$ be the forcing vector. The one-step argument of \eqref{eq:coupled_components} gives
\[
 x\le\bU_\infty p+\mathsf G_\infty x+\bV_{0,\infty}b.
\]
We verify positivity of the Schur complement
\begin{align}\label{eq:schur_infty}
 s_\infty\bydef1-g_{22,\infty}-c_{\mathrm{out},\infty}^Tc_{\mathrm{in},\infty}>0.
\end{align}
Indeed, substituting the top-row inequality into the bottom row gives a tail bound with denominator $s_\infty$. Substituting this bound back into the top row gives
\begin{align}\label{eq:B_infty}
 B_\infty\bydef(I-\mathsf G_\infty)^{-1}
 =\begin{pmatrix}
 I+c_{\mathrm{in},\infty}s_\infty^{-1}c_{\mathrm{out},\infty}^T&c_{\mathrm{in},\infty}s_\infty^{-1}\\
 s_\infty^{-1}c_{\mathrm{out},\infty}^T&s_\infty^{-1}
 \end{pmatrix}\ge0.
\end{align}
Thus $x\le B_\infty\bU_\infty p+B_\infty\bV_{0,\infty}b$, where
\begin{align}\label{eq:bU_bV_infty}
 \bU_\infty\bydef\begin{pmatrix}\Phi_{\sup,\infty}&0\\0&1\end{pmatrix},\qquad
 \bV_{0,\infty}\bydef\begin{pmatrix}\mathcal{V}_{\bbm,\infty}^{(w)}&0\\0&V_{\infty,0}^{(J_\infty,t_0)}\end{pmatrix}.
\end{align}
The tail observation entry is $1$ because the supremum includes elapsed time zero. The tail integral with $q=0$ is obtained from \eqref{eq:Vinfq_infty}. For forcing of the form $\cQ\psi$, replace the finite integral block by $D_\nu\mathcal{V}_{\bbm,\infty}Q_{\bbm}D_\nu^{-1}$ and the tail entry by $V_{\infty,q}^{(J_\infty,t_0)}$; this defines $\bV_{q,\infty}$ as in \eqref{eq:bUm_bVm_def}.
The sum of the entries of each block vector is its weighted sequence norm. Taking matrix $1$-norms therefore gives
\begin{align}\label{eq:WVq_infty}
 \bm{W_\infty}&\bydef\bm{W^{(J_\infty,t_0)}}=\|B_\infty\bU_\infty\|_1,\\
 \bm{V_{0,\infty}}&\bydef\bm{V_0^{(J_\infty,t_0)}}=\|B_\infty\bV_{0,\infty}\|_1,\qquad
 \bm{V_{q,\infty}}\bydef\bm{V_q^{(J_\infty,t_0)}}=\|B_\infty\bV_{q,\infty}\|_1.\notag
\end{align}
Letting $T\to\infty$ proves \eqref{eq:WVbounds} on the infinite interval: they control the homogeneous semigroup, the integral without an additional multiplier, and the integral with the multiplier $\cQ$, respectively.

\begin{remark}[{\bf Truncation for the infinite step}]\label{rem:truncation_infty}
 A small truncation may leave slowly decaying modes in the tail, making \eqref{eq:Vinfq_infty} too large to verify $s_\infty>0$. Increasing the truncation moves these modes into the finite block. The positivity condition must be checked for the selected truncation and coupling data.
\end{remark}

\subsection{Mode-resolved equilibrium residual}\label{app:global_residual_bounds}

The residual contribution can be sharpened by retaining its individual
modes. Set $w_{\bk}\bydef\alpha_{\bk}\nu^{\bk}$ and define
\[
 b_{f,\infty}\bydef
 \begin{pmatrix}
 \mathcal{V}_{\bbm,\infty}^{(w)}D_\nu
     |f(\bu_*)^{(\bbm)}|\\[1mm]
 \displaystyle\sum_{\bk\notin\FK}
     w_{\bk}\frac{|f(\bu_*)_{\bk}|}{-\mu_{\bk}}
 \end{pmatrix}.
\]
For each tail mode, the integral of the constant forcing is bounded by
$|f(\bu_*)_{\bk}|/(-\mu_{\bk})$. Substitution into the same coupling
inequality gives
\begin{align}\label{eq:residual_volterra_resolved}
 \left\|\int_{t_0}^{\cdot}e^{A_\infty(\cdot-s)}f(\bu_*)\,ds\right\|_{X_\infty}
 &\le\|B_\infty b_{f,\infty}\|_1
 \le\bm{V_{0,\infty}}\|f(\bu_*)\|_\nu.
\end{align}
Hence the smaller verified quantity $\|B_\infty b_{f,\infty}\|_1$ may replace
the residual term in~\eqref{eq:Y_global} and the subsequent radii tests.
For an autonomous interval of length $h_a$, use the corresponding finite
integral block and replace $1/(-\mu_{\bk})$ by
$(1-e^{\mu_{\bk}h_a})/(-\mu_{\bk})$.

\subsection{Intermediate linear and endpoint bounds}\label{app:intermediate_endpoint_bounds}

\paragraph{Finite flow enclosures on an added interval.}
Use a fixed truncation $\bbm_a$ satisfying
$\mu_{\bbm_a}=\sup_{\bk\notin\FKa}\mu_{\bk}<0$, as in
\eqref{eq:mu_bbm_def}, with
$\FKa\bydef\{\bk\ge0:\bk<\bbm_a\}$ and weight matrix $D_\nu^{(a)}$.
The finite coefficient matrix $A_{\bbm_a}(t)$ is the restriction of
$\cL+\cQ D\cN(\bu_{M+k}(t))$ to $\FKa$, as in Section~\ref{sec:one-step}.
Choose a numerical diagonalization
$\bA_k=V^{(a,k)}\Lambda^{(a,k)}(V^{(a,k)})^{-1}$.
For elapsed time $r$, write
$\widehat C_k^{(\bbm_a)}(r)=V^{(a,k)}e^{\Lambda^{(a,k)}r}(V^{(a,k)})^{-1}$.
Here $C_{M+k}^{(\bbm_a)}(t,s)$ is the fundamental solution of
the finite system. In the following bounds,
$|C_{M+k}^{(\bbm_a)}|_\infty$ denotes the componentwise
supremum of $|C_{M+k}^{(\bbm_a)}(t,\tau_{M+k-1})|$ over $J_{M+k}$.
Let $\triangle_{\sup}^{(a,k)}$ bound
$|(V^{(a,k)})^{-1}(A_{\bbm_a}(t)-\bA_k)|$ throughout $J_{M+k}$.
The signed coefficients are transformed before taking this bound.
Use the matrices $\tV_m^{(\bbm)}$ and $\widehat V_m^{(\bbm)}$
of \eqref{eq:bV-esitmates}, with $m=M+k$, $\bbm=\bbm_a$,
and step length $h_{a,k}$; below the latter is denoted
$\widehat V_{a,k}^{(\bbm_a)}$.
Set
\begin{align}\label{eq:sigma_aut}
 K_{a,k}&\bydef\tV_{a,k}^{(\bbm_a)}\triangle_{\sup}^{(a,k)},
 &B_{\mathrm{neu},a,k}&\bydef I-K_{a,k}.
\end{align}
We verify $\|K_{a,k}\|_\infty<1$. Applying the variation of constants formula as in
Section~\ref{sec:one-step}, with a time-dependent matrix error, gives
\begin{align}\label{eq:Va_aut}
 |C_{M+k}^{(\bbm_a)}|_\infty
 &\le B_{\mathrm{neu},a,k}^{-1}
       |\widehat C_k^{(\bbm_a)}|_{\infty,[0,h_{a,k}]},
 &\mathcal V_{\bbm_a,M+k}
 &\bydef B_{\mathrm{neu},a,k}^{-1}\widehat V_{a,k}^{(\bbm_a)}.
\end{align}
At the endpoint, the same integral bounds the difference from the signed
approximate exponential. Since the true finite matrix is real, taking
the real part of the approximate exponential also covers diagonalizations
with complex conjugate eigenvalues. Hence
\begin{align}\label{eq:Phi_endpt_enclosure}
 \Phi_{\bbm_a,M+k}&\bydef
 \Re\widehat C_k^{(\bbm_a)}(h_{a,k})+[-R_{a,k},R_{a,k}],\\
 R_{a,k}&\bydef K_{a,k}B_{\mathrm{neu},a,k}^{-1}
       |\widehat C_k^{(\bbm_a)}|_{\infty,[0,h_{a,k}]}.\notag
\end{align}
In particular,
\begin{align}\label{eq:Phi_aut}
 C_{M+k}^{(\bbm_a)}(\tau_{M+k},\tau_{M+k-1})\in\Phi_{\bbm_a,M+k},
 \qquad\varsigma_{a,k}\bydef e^{\mu_{\bbm_a}h_{a,k}}.
\end{align}
\begin{remark}[{\bf Intermediate truncation}]\label{rem:aut_truncation}
 Moving slowly decaying modes into the finite block can reduce the tail
 integrals, at the cost of larger finite matrices. The selected block must
 satisfy the linear invertibility and positivity conditions and the nonlinear radii inequalities.
\end{remark}
This construction verifies a finite flow without a backward integration.
A verified Chebyshev fundamental solution can be substituted.
When the reference is constant at $\bu_*$ and all lengths are $h_a$, these
quantities can be computed once and reused.

\paragraph{Bounds within a step and at its endpoint.}
For the sequential remainder, choose a fixed real invertible matrix
$V_a$ in weighted finite coordinates, scaled so that $\|V_a\|_1\le1$.
It is a coordinate basis, distinct from the local diagonalizations $V^{(a,k)}$.
The identity is always an admissible choice. Define
\[
 z_a(v)\bydef\begin{pmatrix}
 |V_a^{-1}D_\nu^{(a)}v^{(\bbm_a)}|\\\|v^{(\infty)}\|_\nu
 \end{pmatrix},\qquad \|v\|_\nu\le\|z_a(v)\|_1.
\]
For a trajectory $w\in X_{M+k}$, use the corresponding modewise
supremum profile
\[
 z_a(w)\bydef\begin{pmatrix}
 |V_a^{-1}D_\nu^{(a)}w^{(\bbm_a)}|_\infty\\
 \|w^{(\infty)}\|_{X_{M+k}}
 \end{pmatrix}.
\]
The same normalization gives $\|w\|_{X_{M+k}}\le\|z_a(w)\|_1$.
The homogeneous observation matrices use these coordinates for both
input and output. The forcing matrices use physical weighted input
coordinates and $z_a$ output coordinates.

Compute the finite within-step bound after transforming the signed
fundamental solution, and set
\begin{align}\label{eq:Ua_two}
 \bU_{a,k}^{\mathrm{sup}}&\bydef
 \operatorname{diag}(\Phi_{\sup,a,k}^{(V)},1),\\
 \bU_{a,k}^{\mathrm{ept}}&\bydef
 \operatorname{diag}\bigl(
 |V_a^{-1}D_\nu^{(a)}\Phi_{\bbm_a,M+k}(D_\nu^{(a)})^{-1}V_a|,
 \varsigma_{a,k}\bigr).\notag
\end{align}
Here $\Phi_{\sup,a,k}^{(V)}$ bounds the modewise time supremum of
$|V_a^{-1}D_\nu^{(a)}C_{M+k}^{(\bbm_a)}(t,\tau_{M+k-1})
(D_\nu^{(a)})^{-1}V_a|$.
Transform the finite-tail coupling vectors accordingly: the
physical tail-to-finite vector is multiplied by $|V_a^{-1}|$, and the
finite-to-tail row by $|V_a|$. With these vectors, let
\begin{align}\label{eq:coupling_aut}
 \mathsf G_{a,k}^{(V)}&\bydef
 \begin{pmatrix}0&c_{\mathrm{in},a,k}^{(V)}\\
 (c_{\mathrm{out},a,k}^{(V)})^T&g_{22,a,k}\end{pmatrix},
 &B_{a,k}^{(V)}&\bydef(I-\mathsf G_{a,k}^{(V)})^{-1}.
\end{align}
Its positive Schur complement is verified on every step.
The observation and endpoint matrices are
\begin{align}\label{eq:aut_endpoint_resolved}
 U_{a,k}&\bydef B_{a,k}^{(V)}\bU_{a,k}^{\mathrm{sup}},
 &H_{a,k}&\bydef\bU_{a,k}^{\mathrm{ept}}+\mathsf G_{a,k}^{(V)}U_{a,k}.
\end{align}
Indeed, the within-step observation $x$ satisfies
$x\le\bU_{a,k}^{\mathrm{sup}}p+\mathsf G_{a,k}^{(V)}x$.
The endpoint satisfies
$x_{\mathrm{ept}}\le\bU_{a,k}^{\mathrm{ept}}p+\mathsf G_{a,k}^{(V)}x$;
substituting the first bound proves~\eqref{eq:aut_endpoint_resolved}.
Physical scalar bounds can therefore be taken as
\begin{align}\label{eq:Wa_two}
 \bm{W_{a,k}}&\bydef\|U_{a,k}\operatorname{diag}(|V_a^{-1}|,1)\|_1,
 &\bm{W_{a,k}^{\mathrm{ept}}}
 &\bydef\|H_{a,k}\operatorname{diag}(|V_a^{-1}|,1)\|_1.
\end{align}
These are coarse scalar bounds; the correlation-preserving construction
below uses the signed initial-data product directly.
Let $V_{0,a,k}$ and $V_{q,a,k}$ be the resolved forcing matrices.
The notation $\mathcal V_{\bbm_a,M+k}^{(w)}$ and
$\mathcal V_{\bbm_a,M+k}^{(w,q)}$ means
$D_\nu^{(a)}\mathcal V_{\bbm_a,M+k}(D_\nu^{(a)})^{-1}$ and
$D_\nu^{(a)}\mathcal V_{\bbm_a,M+k}Q_{\bbm_a}(D_\nu^{(a)})^{-1}$,
respectively, as in \eqref{eq:bUm_bVm_def}.
Their finite uncoupled blocks are, respectively,
$|V_a^{-1}|\mathcal V_{\bbm_a,M+k}^{(w)}$ and
$|V_a^{-1}|\mathcal V_{\bbm_a,M+k}^{(w,q)}$; their tail entries are the
ordinary and multiplier-weighted tail integrals. The coupling correction
is $B_{a,k}^{(V)}$. Thus
\begin{align}\label{eq:V0q_aut}
 \bm{V_{0,a,k}}\bydef\|V_{0,a,k}\|_1,
 \qquad\bm{V_{q,a,k}}\bydef\|V_{q,a,k}\|_1
\end{align}
bound the corresponding operator norms of the integral terms.
For the full defect
$F_{1,M+k}=\bu_{M+k}'-f(\bu_{M+k})$, retain the mode-resolved ordinary
integral as in~\eqref{eq:residual_volterra_resolved}.
Denote its uncoupled bound by $b_{f,a,k}$, so that
$B_{a,k}^{(V)}b_{f,a,k}$ bounds the resolved defect forcing.
The tail entries use each mode's own integral times
$|F_{1,M+k,\bk}|_\infty$. For a time-dependent polynomial, a coefficient
majorant may be used for this last time supremum.
\paragraph{Retaining the initial error.}
For the decomposition \eqref{eq:Ek_def}, use the intermediate
truncation throughout:
\[
 E_{M+k}\bydef\operatorname{diag}\bigl(
 C_{M+k}^{(\bbm_a)}(\tau_{M+k},\tau_{M+k-1}),
 e^{\cL_{\infty,a}h_{a,k}}\bigr),
\]
Set $\widetilde C_{M+k}\bydef
U_{M+k}(\tau_{M+k},\tau_{M+k-1})-E_{M+k}$,
where $\cL_{\infty,a}$ is the restriction of $\cL$ to the modes
outside $\FKa$. All tail components in this construction refer to
this same truncation.
Let $\bu_M^{\mathrm{ept}}$ be the stored nominal endpoint and let
$r_M^{\mathrm{ept}}$ bound the true endpoint error about it. The gap from
the exact polynomial trace is included in this radius.
Write the initial error as $\xi$, with $\|\xi\|_\nu\le r_M^{\mathrm{ept}}$.
The endpoint error after $k$ added steps is enclosed as
\begin{align}\label{eq:intermediate_affine_enclosure}
 \tilde u(\tau_{M+k})-\bu_{M+k}^{\mathrm{ept}}
 &\in E_{M+k}\cdots E_{M+1}\xi+\mathcal R_k^{(a)},\\
 \mathcal R_k^{(a)}
 &\bydef\{v:\|V_a^{-1}D_\nu^{(a)}v^{(\bbm_a)}\|_1
                 \le\gamma_{k,1}^{(a)},\;
                 \|v^{(\infty)}\|_\nu\le\gamma_{k,2}^{(a)}\}.\notag
\end{align}
The tail part of $\xi$ is included, with its joint finite-tail norm budget;
it is not set to zero. The two-component remainder starts at
$\gamma_0^{(a)}=0$. The finite block of the exact product is enclosed by the signed
interval product $\Phi_{\mathrm{end},M+k,M}^{(\bbm_a)}$ of
Section~\ref{sec:multistep}.
Its nonnegative input observation in the fixed frame is
\[
 R_{a,k}^{(\mathrm{bare})}\bydef\operatorname{diag}\bigl(
 |V_a^{-1}D_\nu^{(a)}\Phi_{\mathrm{end},M+k,M}^{(\bbm_a)}(D_\nu^{(a)})^{-1}|,
 |\varsigma_{M+k,M}|\bigr).
\]
Here $\Phi_{\mathrm{end},M+k,M}^{(\bbm_a)}$ is the signed
interval product from \eqref{eq:bPhi}, with $\Phi_{\mathrm{end},M+k}^{(\bbm_a)}
\bydef\Phi_{\bbm_a,M+k}$ from \eqref{eq:Phi_aut};
its empty product is the identity.
The tail factor $\varsigma_{M+k,M}$ is defined by
\eqref{eq:theta_def} with this fixed truncation, and
$\varsigma_{M,M}=1$. For interval matrices, absolute values in these
formulas mean the entrywise magnitude bound $\operatorname{mag}$.
No inverse of a cumulative endpoint product is used.
For a within-step observation of this initial contribution, choose
verified matrices $S_{a,k}^{(0)}$ and $S_{a,k}^{(V,0)}$ whose finite blocks
bound, respectively,
\[
 D_\nu^{(a)}|C_{M+k}^{(\bbm_a)}(\cdot,\tau_{M+k-1})
                   \Phi_{\mathrm{end},M+k-1,M}^{(\bbm_a)}|_\infty(D_\nu^{(a)})^{-1},
\]
and
\[
 |V_a^{-1}D_\nu^{(a)}C_{M+k}^{(\bbm_a)}(\cdot,\tau_{M+k-1})
                   \Phi_{\mathrm{end},M+k-1,M}^{(\bbm_a)}(D_\nu^{(a)})^{-1}|_\infty.
\]
Their tail diagonal entries are $|\varsigma_{M+k-1,M}|$ and the
other blocks vanish. The signed product is evaluated before its magnitude.
The incoming nominal-point gap is
$\epsilon_{a,k}=\bu_{M+k}(\tau_{M+k-1})-\bu_{M+k-1}^{\mathrm{ept}}$.
For initial values in the preceding affine enclosure
$\bu_{M+k-1}^{\mathrm{ept}}+E_{M+k-1}\cdots E_{M+1}\xi
+\mathcal R_{k-1}^{(a)}$, with $\|\xi\|_\nu\le r_M^{\mathrm{ept}}$,
a uniform local defect bound, also valid on its intersection with
the scalar endpoint ball, is
\begin{align}\label{eq:Y_aut}
 Y_{M+k}^{(a)}
 &\bydef r_M^{\mathrm{ept}}\bigl(
 \|S_{a,k}^{(0)}\|_1+
 \|\mathsf G_{a,k}^{(V)}B_{a,k}^{(V)}S_{a,k}^{(V,0)}\|_1\bigr)\notag\\
 &\quad+\mathbf1^T\mathfrak b(U_{a,k})\gamma_{k-1}^{(a)}
       +\|U_{a,k}z_a(\epsilon_{a,k})\|_1
       +\|B_{a,k}^{(V)}b_{f,a,k}\|_1.
\end{align}
The first line bounds the uncoupled initial contribution and its
within-step coupling correction. The second bounds the inherited
remainder, the initial gap, and the full defect forcing.
The defects are those of \eqref{eq:local_defects}: the differential
residual is $y_1^{M+k}=F_{1,M+k}$ and the incoming mismatch is
measured by $\epsilon_{a,k}$ together with the verified endpoint error.
Thus \eqref{eq:Y_aut} is the sequential counterpart of
\eqref{eq:Ym_bound}, with propagation of the residual through the integral without an additional factor $\cQ$.

\paragraph{Updating the endpoint enclosure.}
We use $\mathfrak b$ and $\operatorname{pair}$ from
\eqref{eq:block_norm_matrix}, now with finite block $\FKa$.
For a two-column nonnegative matrix, $\operatorname{rowmax}$ denotes
the vector of its two row maxima. If $p\ge0$ and $\|p\|_1\le r$,
then $Ap\le r\operatorname{rowmax}(A)$. This retains the joint
finite-tail norm budget of the incoming error.
Let
\[
 \epsilon_{a,k}^{\mathrm{ept}}\bydef
 \bu_{M+k}(\tau_{M+k})-\bu_{M+k}^{\mathrm{ept}},
 \qquad N_k\bydef(3a_k+r_{M+k})r_{M+k}^2.
\]
Splitting the true endpoint operator as $E_{M+k}+\widetilde C_{M+k}$
and using the local nonlinear bound gives
\begin{align}\label{eq:intermediate_remainder_update}
 \gamma_k^{(a)}
 &\bydef\mathfrak b(H_{a,k})\gamma_{k-1}^{(a)}\notag\\
 &\quad+r_M^{\mathrm{ept}}\operatorname{rowmax}
       \mathfrak b(\mathsf G_{a,k}^{(V)}U_{a,k}R_{a,k-1}^{(\mathrm{bare})})\notag\\
 &\quad+\mathfrak b(H_{a,k})\operatorname{pair}(z_a(\epsilon_{a,k}))
       +\operatorname{pair}(B_{a,k}^{(V)}b_{f,a,k})\notag\\
 &\quad+\operatorname{rowmax}\mathfrak b(V_{q,a,k})N_k
       +\operatorname{pair}(z_a(\epsilon_{a,k}^{\mathrm{ept}})).
\end{align}
In particular, $H_{a,k}$ is assembled before its block norms are taken;
this retains the association between each finite column and its decay.
\begin{lemma}[{\bf Affine endpoint enclosure}]\label{lem:affine_enclosure}
Assume that, for $k=1,\dots,M_a$, the local conditions \eqref{eq:radii_aut} hold uniformly for initial values in the preceding affine enclosure, that $\|K_{a,k}\|_\infty<1$, and that the Schur complement associated with $\mathsf G_{a,k}^{(V)}$ is positive. Let $\gamma_k^{(a)}$ be defined by \eqref{eq:intermediate_remainder_update} with $\gamma_0^{(a)}=0$. Then \eqref{eq:intermediate_affine_enclosure} holds for $k=0,\dots,M_a$.
\end{lemma}
\begin{proof}
For $k=0$, the claim is the definition of $\xi$. Assume that it holds for $k-1$, and write $\tilde u(\tau_{M+k-1})-\bu_{M+k-1}^{\mathrm{ept}}=E_{M+k-1}\cdots E_{M+1}\xi+\varrho$ with $\varrho\in\mathcal R_{k-1}^{(a)}$. Let $w\bydef\tilde u-\bu_{M+k}$ on $J_{M+k}$, so that $w(\tau_{M+k-1})=E_{M+k-1}\cdots E_{M+1}\xi+\varrho-\epsilon_{a,k}$. As in the proof of Proposition~\ref{prop:intermediate_enclosure},
\begin{align*}
 w(\tau_{M+k})&=U_{M+k}(\tau_{M+k},\tau_{M+k-1})w(\tau_{M+k-1})\\
 &\quad+\int_{\tau_{M+k-1}}^{\tau_{M+k}}U_{M+k}(\tau_{M+k},s)\\
 &\qquad\quad\bigl(\cQ[\cN(\bu_{M+k}+w)-\cN(\bu_{M+k})-D\cN(\bu_{M+k})w]-F_{1,M+k}\bigr)(s)\,ds.
\end{align*}
Splitting $U_{M+k}(\tau_{M+k},\tau_{M+k-1})=E_{M+k}+\widetilde C_{M+k}$ in the term containing $\xi$ gives the leading term $E_{M+k}\cdots E_{M+1}\xi$ of \eqref{eq:intermediate_affine_enclosure}. The remaining terms are bounded in the coordinates $z_a$ and grouped by $\operatorname{pair}$. First, $z_a(\widetilde C_{M+k}E_{M+k-1}\cdots E_{M+1}\xi)\le\mathsf G_{a,k}^{(V)}U_{a,k}R_{a,k-1}^{(\mathrm{bare})}z(\xi)$, and $\|z(\xi)\|_1\le r_M^{\mathrm{ept}}$ gives the $\operatorname{rowmax}$ term in the second line of \eqref{eq:intermediate_remainder_update}. Second, the endpoint matrix $H_{a,k}$ of \eqref{eq:aut_endpoint_resolved} bounds the contributions of $\varrho$ and $\epsilon_{a,k}$, which gives $\mathfrak b(H_{a,k})\gamma_{k-1}^{(a)}$ and $\mathfrak b(H_{a,k})\operatorname{pair}(z_a(\epsilon_{a,k}))$. Third, the residual integral is bounded by $\operatorname{pair}(B_{a,k}^{(V)}b_{f,a,k})$, and, since $\tilde u\in B_{r_{M+k}}(\bu_{M+k})$ by \eqref{eq:radii_aut}, the cubic estimate of Proposition~\ref{prop:intermediate_enclosure} bounds the nonlinear integral by $\operatorname{rowmax}\mathfrak b(V_{q,a,k})N_k$. Finally, replacing $\bu_{M+k}(\tau_{M+k})$ by $\bu_{M+k}^{\mathrm{ept}}$ adds $\operatorname{pair}(z_a(\epsilon_{a,k}^{\mathrm{ept}}))$. Summing these bounds gives $\gamma_k^{(a)}$ and proves the claim for $k$.
\end{proof}
A physical
endpoint radius is
\begin{align}\label{eq:intermediate_endpoint_radius}
 r_{M+k}^{\mathrm{ept}}
 &\bydef\min\Bigl\{
 r_M^{\mathrm{ept}}\max\bigl(
 \|D_\nu^{(a)}\operatorname{mag}(\Phi_{\mathrm{end},M+k,M}^{(\bbm_a)})
                   (D_\nu^{(a)})^{-1}\|_1,
 |\varsigma_{M+k,M}|\bigr)
       +\mathbf1^T\gamma_k^{(a)},\notag\\
 &\hspace{25mm}r_{M+k}+\|\epsilon_{a,k}^{\mathrm{ept}}\|_\nu\Bigr\}.
\end{align}
The second bound follows directly from the local invariant ball.

\subsection{Matching thresholds and numerical predictors}\label{app:global_predictors}
The quantities in this subsection are used only to choose the number and length of the intermediate steps and initial guesses for the radii. They are not part of the proof: every accepted choice is checked with the rigorous inequalities of Section~\ref{sec:global_in_time}.

Define $p_\infty(r;\rho)\bydef\bm{W_\infty}\rho+\bm{V_{0,\infty}}\|f(\bu_*)\|_\nu+\int_0^rZ_\infty(s)\,ds-r$. Its value at zero is the defect bound, and its initial slope is $-1$. The local threshold in the matching residual is
\begin{align}\label{eq:rho_star}
 \rho_*\bydef\sup\!\left\{\rho\ge0:\inf_{r>0}p_\infty(r;\rho)\le0\right\}.
\end{align}
This expresses the balance between the incoming residual and the available nonlinear margin. In particular, a positive threshold requires
\[
 \bm{V_{0,\infty}}\|f(\bu_*)\|_\nu
 <\sup_{r>0}\left(r-\int_0^rZ_\infty(s)\,ds\right).
\]
A candidate with $\rho_{M_a}<\rho_*$ is accepted only after a radius is selected and both inequalities in \eqref{eq:radii_infty} are verified. The full product-space test must also include the incoming endpoint perturbations.

\begin{example}[{\bf Cubic nonlinearity}]\label{ex:cubic_inf}
 For $\cN(u)=\pm u^3$, write $u=\bu_*+v$ with $\|v\|_{X_\infty}\le r$. Then
 $u^2-\bu_*^2=2\bu_*v+v^2$, so the Banach algebra inequality gives
 $\|u^2-\bu_*^2\|_{X_\infty}\le2\|\bu_*\|_\nu r+r^2$. Since $D\cN(u)w=\pm3u^2w$, we obtain
 \begin{align}\label{eq:Z_global_cubic}
 Z_\infty(r)=3\bm{V_{q,\infty}}r(2\|\bu_*\|_\nu+r).
 \end{align}
 Integrating this expression gives the cubic form of \eqref{eq:radii_infty} stated in Section~\ref{sec:infty_step}.
 For a numerical predictor, omit the nonnegative cubic term. The resulting quadratic has discriminant
 $\triangle_\infty\bydef1-12\bm{V_{q,\infty}}\|\bu_*\|_\nu Y_\infty^{(M_a)}$.
 Equivalently, the maximum of $r-3\bm{V_{q,\infty}}\|\bu_*\|_\nu r^2$ is
 $1/(12\bm{V_{q,\infty}}\|\bu_*\|_\nu)$. The full cubic has no greater margin, so
 \begin{align}\label{eq:rho_star_cubic}
 \rho_*\le\frac1{\bm{W_\infty}}\left(
 \frac1{12\bm{V_{q,\infty}}\|\bu_*\|_\nu}
 -\bm{V_{0,\infty}}\|f(\bu_*)\|_\nu\right)
 \end{align}
 whenever $\bm{V_{q,\infty}}\|\bu_*\|_\nu>0$ and the local feasible set is nonempty. If $\triangle_\infty>0$, the smaller quadratic root
 $(1-\sqrt{\triangle_\infty})/(6\bm{V_{q,\infty}}\|\bu_*\|_\nu)$
 supplies an initial guess for the radius. Validation still uses the full cubic and $Z_\infty(r_\infty)<1$.
\end{example}

\paragraph{Constant-center predictors.}
The constant reference $\bu_{M+k}\equiv\bu_*$ is a special choice that supplies simple predictors. Its linear residual recurrence alone does not enclose the nonlinear trajectory.
Let $\bar{\widehat u}=(\bu_1,\ldots,\bu_M,\bu_*,\ldots,\bu_*,\bu_*)$
be the full approximation with constant intermediate and infinite
components. For $\bu_{M+k}\equiv\bu_*$ and $h_{a,k}=h_a$,
evaluation at this approximation gives
\begin{align}\label{eq:residual_aut}
 \bu_*-\sigma_{M+k}^{(a)}(\bar{\widehat u})
 =e^{A_\infty h_a}(\bu_*-\sigma_{M+k-1}^{(a)}(\bar{\widehat u}))
  -\int_0^{h_a}e^{A_\infty s}f(\bu_*)\,ds.
\end{align}
Put
\[
 \rho_k^{(0)}\bydef
 \|\bu_*-\sigma_{M+k}^{(a)}(\bar{\widehat u})\|_\nu,
 \qquad
 \rho_0^{(0)}\bydef\|\bu_*-\sigma_M(\bu)\|_\nu.
\]
Here $\bu=(\bu_1,\ldots,\bu_M)$ in the second formula.
These quantities are residuals of the map at the approximation,
and are distinct from the true-endpoint bounds $\rho_k$.
For the constant reference and step length, omit the step index from
$\bm{W_{a,k}^{\mathrm{ept}}}$, $\bm{W_{a,k}}$, and
$\bm{V_{\star,a,k}}$, writing $\bm{W_a^{\mathrm{ept}}}$,
$\bm{W_a}$, and $\bm{V_{\star,a}}$, respectively. Then
\begin{align}\label{eq:rho_recur}
 \rho_k^{(0)}\le\bm{W_a^{\mathrm{ept}}}\rho_{k-1}^{(0)}
                         +\bm{V_{0,a}}\|f(\bu_*)\|_\nu.
\end{align}
If the constant-center endpoint bound satisfies
\begin{align}\label{eq:Waept_lt_1}
 0<\bm{W_a^{\mathrm{ept}}}<1,
\end{align}
write $w=\bm{W_a^{\mathrm{ept}}}$,
$\beta=\bm{V_{0,a}}\|f(\bu_*)\|_\nu$, and
$\rho_a^\infty=\beta/(1-w)$. The predictor satisfies
\begin{align}\label{eq:rho_decay}
 \rho_k^{(0)}\le w^k\rho_0^{(0)}+\rho_a^\infty(1-w^k).
\end{align}
When $\rho_a^\infty<\tfrac12\rho_*<\rho_0^{(0)}$, a sufficient predictor length is
\begin{align}\label{eq:Ka_estimate}
 M_a^{\mathrm{pred}}\bydef
 \left\lceil\frac{\log((\rho_0^{(0)}-\rho_a^\infty)/
                              (\tfrac12\rho_*-\rho_a^\infty))}{-\log w}\right\rceil.
\end{align}
The factor $\tfrac12$ is a safety margin; any factor in $(0,1)$ can be used.
For this special reference, the local incoming threshold is
\begin{align}\label{eq:rho_step_local}
 \rho_*^{(a)}\bydef\sup\Bigl\{\rho\ge0:\inf_{r>0}
  [\bm{W_a}\rho+\bm{V_{0,a}}\|f(\bu_*)\|_\nu
   +\bm{V_{q,a}}(3\|\bu_*\|_\nu r^2+r^3)-r]<0\Bigr\}.
\end{align}
Dropping the cubic term yields the necessary upper bound
\begin{align}\label{eq:rho_star_a_cubic}
 \rho_*^{(a)}\le\frac1{\bm{W_a}}\left(
 \frac1{12\bm{V_{q,a}}\|\bu_*\|_\nu}
                  -\bm{V_{0,a}}\|f(\bu_*)\|_\nu\right),
\end{align}
when the denominator is positive and the feasible set is nonempty.
The mode-resolved residual bound may replace the ordinary scalar term.
The finite-mode heuristic
\begin{align}\label{eq:threshold_ratio}
 \frac{\rho_*^{(a)}}{\rho_*}\sim
       \frac{\bm{V_{q,\infty}}}{\bm{V_{q,a}}}
\end{align}
can guide the choice of length. Actual acceptance still requires the
nonlinear tests and the verified endpoint enclosures above.

\subsection{A product-space application of the Radii Polynomial Theorem}\label{app:global_product_argument}

Instead of the sequential argument of Section~\ref{sec:refined_cor}, one can apply Theorem~\ref{thm:RadiiPolynomial} directly to the combined map on the extended product space
\begin{align}\label{eq:X_extended}
 \widehat{\bm X}\bydef\bm X\times\prod_{k=1}^{M_a}X_{M+k}\times X_\infty,
 \qquad X_{M+k}\bydef\ell^1(C^0(J_{M+k})).
\end{align}
Sequential validation fixes the actual incoming endpoint on each row.
If a contraction of the entire product map is used instead, all incoming
endpoint derivatives must be included. Retain $L_j^i(s)$ as a bound on
$D_i\sigma_j$ and initialize
\begin{align}\label{eq:endpoint_derivative_init}
 L_M^i(s)\bydef Z_M^i(s),\qquad i\le M.
\end{align}
Endpoint evaluation has norm at most one, so these are valid bounds.
For $j=M+k$, the endpoint and within-step bounds give
\begin{align}\label{eq:endpoint_derivative_aut}
 L_j^i(s)&\bydef\bm{W_{a,k}^{\mathrm{ept}}}L_{j-1}^i(s),\quad i<j,
 &L_j^j(s)&\bydef Z_j^{(a)}(s),\\
 \widehat Z_j^i(s)&\bydef\bm{W_{a,k}}L_{j-1}^i(s),\quad i<j,
 &\widehat Z_j^j(s)&\bydef Z_j^{(a)}(s).
 \label{eq:cross_derivative_aut}
\end{align}
For the infinite row, use
\begin{align}\label{eq:cross_derivative_infty}
 \widehat Z_\infty^i(s)&\bydef\bm{W_\infty}L_{M+M_a}^i(s),
 &\widehat Z_\infty^\infty(s)&\bydef Z_\infty(s).
\end{align}
For this alternative, compute defect bounds at the approximation of
the entire product map:
\[
 \|T_j^{(a)}(\bar{\widehat u})-\bu_j\|_{X_j}
 \le\widehat Y_j^{(a)},\qquad
 \|T_\infty(\bar{\widehat u})-\bu_*\|_{X_\infty}
 \le\widehat Y_\infty,
\]
where $\bar{\widehat u}=(\bu_1,\ldots,\bu_{M+M_a},\bu_*)$.
These bounds must include the recursive nominal endpoints
$\sigma_{j-1}^{(a)}(\bar{\widehat u})$; the sequential bounds
$Y_j^{(a)}$ and $Y_\infty^{(M_a)}$ can be reused only if they also
cover these endpoints. The sufficient product-space inequalities are
\begin{align}\label{eq:radii_aut_coupled}
 \widehat Y_j^{(a)}+\sum_{i=1}^j\int_0^{r_i}\widehat Z_j^i(s)\,ds&\le r_j,
 &\sum_{i=1}^j\widehat Z_j^i(r_i)\zeta_i&<\zeta_j,\\
 \widehat Y_\infty+
 \sum_{i=1}^{M+M_a}\int_0^{r_i}\widehat Z_\infty^i(s)\,ds
 +\int_0^{r_\infty}Z_\infty(s)\,ds&\le r_\infty,\notag\\
 \sum_{i=1}^{M+M_a}\widehat Z_\infty^i(r_i)\zeta_i
 +Z_\infty(r_\infty)\zeta_\infty&<\zeta_\infty.
 \label{eq:radii_infty_coupled}
\end{align}
Here all weights, including $\zeta_\infty$, are positive. Choosing
them all equal to one recovers the sufficient row-sum test. As in
Section~\ref{sec:trajectory_validation}, the derivative matrix is
lower triangular, so positive contraction weights exist whenever all
diagonal derivative bounds are below one. The invariant-ball inequalities
must still be checked. Under \eqref{eq:radii_aut_coupled} and \eqref{eq:radii_infty_coupled}, the proof of Theorem~\ref{thm:RadiiPolynomial}, applied to this finite product including its infinite-interval component, gives the same global enclosure as Corollary~\ref{cor:global_existence}.

\end{document}